\documentclass[12pt,reqno, final]{amsart}

\usepackage{amssymb,amsmath,graphicx,
	amsfonts,euscript}
\usepackage{color}
\usepackage{cite}
\usepackage{showkeys}

\makeatletter
\@namedef{subjclassname@2020}{%
	\textup{2020} Mathematics Subject Classification}

\makeatother

\allowdisplaybreaks

\newcommand{\be}{\begin{eqnarray}}
\newcommand{\ee}{\end{eqnarray}}
\newcommand{\bes}{\begin{eqnarray*}}
\newcommand{\ees}{\end{eqnarray*}}

\newcommand{\md}{\mathrm{d}}
\newcommand{\mdx}{\mathrm{d}x}

\def\R{\mathbb{R}}
\newcommand\divg {{\text{div}}}

\newcommand\les{\leqslant}
\newcommand\ges{\geqslant}

\allowdisplaybreaks

\newtheorem{thm}{Theorem}[section]
\newtheorem{pro}{Proposition}[section]
\newtheorem{lem}{Lemma}[section]

\newtheorem{re}{Remark}[section]

\newcommand{\beq}{\begin{equation}}
\newcommand{\eeq}{\end{equation}}
\newcommand{\ben}{\begin{eqnarray}}
\newcommand{\een}{\end{eqnarray}}
\newcommand{\beno}{\begin{eqnarray*}}
\newcommand{\eeno}{\end{eqnarray*}}

\numberwithin{equation}{section}

\begin{document}
	
\title[large bulk viscosity and small angular viscosity ]{The combined limits of large bulk viscosity and small angular viscosity for the 2D compressible micropolar fluid equations}
	
\author[\small F. B. Lin, S. Q. Liu, and J. W. Zhang]{\small Fangbin Lin$^{1}$, Shengquan Liu$^{2}$ and Jianwen Zhang$^{3}$}

\address{1. Fangbin Lin,  School of Mathematical Sciences, Xiamen University, Xiamen  361005, P. R. China}
 \email{19020230157208@stu.xmu.edu.cn}

 \address{2. Shengquan Liu,  School of Mathematics and Statistics, Liaoning University, Shenyang 110036, P. R. China}
 \email{shquanliu@163.com}

 \address{3. Jianwen Zhang (Corresponding Author), School of Science, Jimei University, Xiamen 361021, P. R. China}
 \email{jwzhang@jmu.edu.cn}

\vskip .2in
\begin{abstract}	
This paper is concerned with the combined limits of large bulk viscosity and small angular viscosity for the two-dimensional compressible micropolar fluid equations. We first establish the global well-posedness of strong solutions to the Cauchy problem with large bulk viscosity, when the initial velocity is  well-prepared in the sense that $\sqrt\nu \|\divg u_0\|_{L^2}$ is uniformly bounded. The global a priori estimates obtained are uniform with respect to both the bulk viscosity 
$\nu$ and the angular viscosity $\varepsilon$. Based on these uniform estimates, we then justify the combined limits as $\nu\to\infty$ and $\varepsilon\to0$, and show that the solutions of the compressible micropolar system converge to those of the viscous and non-diffusive incompressible micropolar system. The convergence rates for the incompressible limit and the vanishing angular viscosity limit are also derived. In particular, the convergence rate for the incompressible limit is shown to be of order $\nu^{-1/4}$
in general, and can be improved to $\nu^{-1/2}$ under an additional decay condition on the initial density. Compared with previous works, the vanishing angular viscosity limit is established without the extra technical assumption $\lim_{\varepsilon\to0+}\zeta/\varepsilon^\alpha<\infty$ for some $1/2\les \alpha<\infty$, which extends the relevant results from the incompressible case to the compressible setting. As a by-product, the global regularity of large solutions for the 2D compressible micropolar equations with/without angular viscosity is obtained, provided the bulk viscosity is large enough.
\end{abstract}
	\maketitle

\section{Introduction}\label{sec1}
The mathematical theory of micropolar fluids  was firstly initiated by Eringen \cite{Eringen1964,Eringen1966} and  has been successfully applied to simulate the motions of a variety of complex liquids such as blood, suspensions, etc. The  two-dimensional micropolar equations can be formally derived from the three-dimensional framework by assuming that the vertical  velocity vanishes (i.e. $U=(u_1,u_2,0)$) and the axes of the rotation of particles are parallel to the vertical axis (i.e. $W=(0,0,w)$). The 2D motion of compressible micropolar fluids is governed by the following equations (cf. \cite{Lu1999}):
\begin{equation}
\begin{cases}\label{1.1}
 \partial_{t} \rho+\divg(\rho u)=0,\\[1mm]
  \partial_{t}(\rho u)+\divg(\rho u \otimes u)+\nabla P(\rho )\\
  \quad =(\mu+\zeta)\Delta u+
(\mu+\lambda-\zeta)\nabla\divg u-2\zeta\nabla^{\bot} w ,\\[1mm]
  \partial_{t}(\rho   w )+\divg (\rho  u  w  )+4\zeta w=\varepsilon\Delta w +2\zeta\nabla^{\bot}\cdot u,
\end{cases}
\end{equation}
where the unknown functions $\rho=\rho(x,t)\geq0$, $u=(u_1,u_2)(x,t)$ and $ w = w (x,t)$ are the fluid
density, the velocity and the micro-rotational velocity, respectively. The pressure $P=P(\rho)$ is determined through
the $\gamma$-law equation of state:
\begin{equation}\label{1.2}
P(\rho)=A\rho^\gamma\quad {\text{with}}\quad \gamma>1\quad{\text {and}}\quad A>0.
\end{equation}
The physical parameters $\mu$, $\lambda$, $\varepsilon$ and $\zeta$ are assumed to satisfy
\begin{equation}\label{1.3}
\mu>0,\quad  2\mu+\lambda> 0,\quad \varepsilon\ge 0\quad{\text{and}}\quad \zeta> 0,
\end{equation}
where $\mu$ is the shear  viscosity coefficient,  $\lambda$ is the bulk viscosity coefficient,  $\varepsilon$ is the angular viscosity coefficient,  and $\zeta$ is the micro-rotational viscosity coefficient, respectively. Here,  $\nabla^\bot\triangleq (-\partial_2,\partial_1)$ with $\partial_i \triangleq \partial_{x_i}$ for  $i=1,2$, $\nabla^\bot w\triangleq (-\partial_2w,\partial_1 w)$, and $\nabla^\bot\cdot u\triangleq \partial_1u_2-\partial_2 u_1$ is the vorticity of the fluid.

The micropolar fluid is defined as a type of non-Newtonian fluid that incorporates the micro-rotational effects and couple stresses. When the micro-rotational effects are neglected (i.e. $w=0$), the system (\ref{1.1}) reduces to the Navier-Stokes equations for compressible  viscous fluids, which have been extensively studied by many authors, see, for example, \cite{Fe2004,FNP2001,Li1998,MN1980} and among others. Because of its physical importance, the mathematical theory of micropolar fluids has also attracted a lot of attention from many mathematicians. In particular, the one-dimensional flow of the compressible viscous micropolr fluid is well understood (cf. \cite{Mu1998-1, Mu1998-2}). Based on the compensated compactness technique developed in \cite{Li1998,Fe2004}, Amirat-Hamdache \cite{Amirat} proved the global-in-time existence of weak solutions of the 3D compressible micropolar equations with finite-energy data. In \cite{Chen2015}, the authors obtained the global weak solutions with  small-energy data, which may be discontinuous and contain vacuum. The weak solutions shown in \cite{Chen2015} are called the ``intermediate weak" solutions, which were firstly introduced by Hoff \cite{Ho1995} for compressible Navier-Stokes equations and have more regularity than the ones in \cite{Li1998,FNP2001}. However, the uniqueness of such weak solutions is still unknown. Analogously to that in \cite{MN1980}, the global existence and asymptotic behavior of smooth solutions was respectively studied by Tong-Pan-Tan \cite{TPZ2021} and Song  \cite{Song2023} in Sobolev and Besov spaces,  provided the initial perturbation around the non-vacuum equilibrium is sufficiently small.

In this paper, we aim to consider an initial value problem of (\ref{1.1}) subject to the initial conditions:
\begin{equation}\label{1.4}
(\rho, u,  w )(x,0)=(\rho_{0}, u_{0}, w_0)(x),\quad x\in\mathbb{R}^2,
\end{equation}
and the far-field conditions:
\begin{equation}\label{1.5}
(\rho, u,  w )(x,t)\rightarrow (\widetilde{\rho}, 0,0)\quad{\text{as}}\quad  |x|\to\infty,
\end{equation}
where $\widetilde{\rho}>0$ is a given positive constant.

To explain our motivation precisely, let us assume for the moment that $(\rho,u,w)$ is a smooth solution of (\ref{1.1}). It follows from  the momentum equations (\ref{1.1})$_2$ that
$$
 \rho u_t+\rho u\cdot\nabla u+\nabla P(\rho )=(\mu+\zeta)\nabla^\bot\left(\nabla^\bot\cdot u\right)+
\nu \nabla\divg u-2\zeta\nabla^{\bot} w
$$
with $ \nu\triangleq 2\mu+\lambda$. 
So, if the bulk viscosity tends to infinity (i.e. $\nu=2\mu+\lambda\to\infty$ with $\mu>0$ fixed and $\lambda\to\infty$), then one formally gets that $\nabla \divg u\to0$ as $\nu\to\infty$. Thus, it is expected that $\divg u\to0$  as $\nu\to\infty$ in some sense, provided the initial data is well-prepared  (i.e. $\divg u_0\to0$ as $\nu\to0$). In other words, as $\nu\to\infty$, the system (\ref{1.1}) turns into the 2D  micorpolar equations for nonhomogeneous incompressible viscous  fluids,
\begin{equation}
\begin{cases}\label{1.6}
 \rho_t+u\cdot\nabla \rho=0,\\[1mm]
  \rho u_t+\rho u\cdot\nabla u+\nabla \Pi =(\mu+\zeta)\Delta u
-2\zeta\nabla^{\bot} w ,\\[1mm]
\rho   w_t+\rho  u \cdot\nabla w  +4\zeta w=\varepsilon\Delta w +2\zeta\nabla^{\bot}\cdot u,\\[1mm]
\divg u=0.
\end{cases}
\end{equation}
Note that  $\nabla^\bot(\nabla^\bot\cdot u)=\Delta u$ when $\divg u=0$.

It is also known that (cf. \cite{Eringen1964,Eringen1966})  for the dilute particle suspensions, the micro-rotational diffusion is extremely weak. So, the diffusion term $\varepsilon\Delta w$ on the right-hand side of (\ref{1.6})$_3$ is often dropped in such a situation, and the system (\ref{1.6}) can be further simplified to the following viscous and non-diffusive equations,
\begin{equation}
\begin{cases}\label{1.7}
 \rho_t+u\cdot\nabla \rho=0,\\[1mm]
  \rho u_t+\rho u\cdot\nabla u+\nabla \Pi =(\mu+\zeta)\Delta u
-2\zeta\nabla^{\bot} w ,\\[1mm]
\rho   w_t+\rho  u \cdot\nabla w  +4\zeta w= 2\zeta\nabla^{\bot}\cdot u,\\[1mm]
\divg u=0.
\end{cases}
\end{equation}
Both systems (\ref{1.6}) and (\ref{1.7}) describe the motion of the nonhomogeneous incompressible micropolar fluids. We mention here that the mathematical theory of nonhomogeneous incompressible fluids was  initiated by the Russian mathematicians, see, for example, \cite{AKM1990} (also cf. \cite{Lions1996}).

There have also been numerous studies on the mathematical theory of incompressible micropolar flows. For the 3D homogeneous incompressible micropolar fluid equations (i.e. $\rho\equiv {\text{Const.}}$), the global weak solutions with large data and the global strong solutions with small data were obtained by {\L}ukaszewicz  in \cite{Lu1988} and \cite{Lu1989}, respectively. For the nonhomogeneous incompressible micropolar fluid with positive density in 3D, the local strong solutions with large data and the global strong solutions with small data were studied in \cite{BRF2003}. The global Fujita-Kato's type solutions with bounded density and small velocity in Besov space were considered in \cite{QCZ2023}. However, analogously to the 3D Navier-Stokes equations, the problem of global regularity and uniqueness of weak solutions with large data remains completely open. Compared with the 3D equations, the mathematical theory of 2D system  is better understood.  In particular, Dong-Zhang \cite{DZ2010} proved the global regularity of large solutions to the Cauchy problem  of the 2D micropolar equations for homogeneous incompressible fluids  with zero angular viscosity  (i.e. system  (\ref{1.7}) with $\rho\equiv {\rm Const.}$, $\mu>0$ and $\varepsilon=0$). Recently, Lin-Xu-Zhang \cite{LXZ2026} justified the vanishing limit of angular viscosity (i.e. $\varepsilon\to0$) of the 2D nonhomogeneous incompressible equations (\ref{1.6}) on bounded domain with Navier's type slip boundary condition, and obtained the global strong solutions of system (\ref{1.7}) with large data.

The incompressible limit at large bulk viscosity was firstly investigated by Danchin-Mucha \cite{DM2017} for the compressible Navier-Stokes equations. In \cite{DM2017}, the authors also proved the global existence of regular solutions to the 2D compressible Navier-Stokes equations with arbitrary large initial velocity and almost constant density in critical Besov spaces, provided the large bulk viscosity is large enough. Later, it was extended to the $L^p$-framework by  Chen-Zhai \cite{CZ2019}.  Recently,  Danchin-Mucha \cite{Danchin2023} proved the global existence of weak solutions of the barotropic compressible Navier-Stokes equations with large and low-regularity initial data on the torus $\mathbb{T}^2$, when the initial density is merely bounded and nonnegative (possibly discontinuous and containing vacuum regions). More interestingly, the uniqueness of weak solutions (i.e. the so-called Hoff's type ``intermediate weak" solutions, see \cite{Ho1995}) to the 2D isothermal equations was shown. The convergence rates of the incompressible limit (as $\nu\to\infty$) for the global solutions with vacuum  was studied by Liu-Zhang \cite{LZ2026}, based on some new $t$-weighted and singular $t$-weighted analysis. In \cite{LXZ2026-1}, the authors considered the global regularity of large solutions and the incompressible limit at high bulk viscosity for the 2D Cauchy problem of compressible Navier-Stokes equations with the far-field vacuum (i.e. $\rho\to 0$ as $x\to\infty$).

Motivated by the results mentioned above, the main purpose of this paper is to justify the combined limits at large bulk viscosity (i.e. $\nu\to\infty$) and small angular viscosity (i.e. $\varepsilon\to0$) from the  {\it viscous diffusive compressible} equations (\ref{1.1}) to the  {\it viscous   non-diffusive incompressible}  equations (\ref{1.7}). More precisely, the incompressible limit will be characterised by the large value of the bulk
viscosity, while the non-diffusive limit will be justified by the small value of the angular viscosity. Indeed, to the authors' knowledge, the global regularity of large solutions of the 2D viscous and non-diffusive equations for compressible micropolar fluids (i.e. (\ref{1.1}) with  $\varepsilon=0$) is still unknown. Moreover, it is worth pointing out that even for the 2D micropolar equations for the homogeneous incompressible fluids (i.e. $\rho={\rm Const.}$ in (\ref{1.6})), the non-diffusive limit as $\varepsilon \to0$ was usually studied under the following technical condition (see, for example, \cite{Chen2015}):
 \begin{equation}\label{1.8}
 \lim_{\varepsilon\to0+}\frac{\zeta}{\varepsilon^\alpha}<\infty\;\; {\text{with}}\;\; \frac{1}{2}\les \alpha<\infty.
 \end{equation}
This  particularly implies that  the rotational term $2\zeta\nabla^\bot w$ on the right-hand side of (\ref{1.6})$_2$ acts as a small force on the momentum equations, and hence, the coupling between the $u$ and $w$ equations  is very weak as   $\varepsilon\to0$.  In the present paper, even  without the  additional assumption (\ref{1.8}), we can still justify the vanishing limit of angular viscosity (i.e. $\varepsilon\to0$ ) for the 2D compressible micropolar equations (\ref{1.1}).

We shall work with the  standard Lebesgue and Sobolev spaces
$
L^p \triangleq L^p(\R^2)$, $W^{k,p}\triangleq W^{k,p}(\R^2)$ and $H^k\triangleq H^k(\R^2)=W^{k,2}(\R^2)
$
with the norms being denoted respectively by $\|\cdot\|_{L^p}$, $\|\cdot\|_{W^{k,p}}$ and $\|\cdot\|_{H^k}$  for $k\in \mathbb{Z}$ and $1\les p\les \infty$. 

\vskip 2mm

Next, let us impose the initial conditions. Let $\underline\rho$ and $\overline\rho$ be two positive constants, satisfying $0<\underline\rho<\widetilde\rho<\overline\rho<\infty$. Assume that
\begin{equation}\label{1.9}
\underline\rho \les \rho_0\les \overline\rho,\quad (\rho_0-\widetilde\rho,w_0)\in H^1\cap W^{1,4},\quad  u_0\in H^1,
\end{equation}
and that there exists a positive constant $K>0$ (not necessarily small) such that
\begin{equation}\label{1.10}
\sqrt\nu \|\divg u_0\|_{L^2}\les K.
\end{equation}
The condition (\ref{1.10}) indicates that the initial velocity is ``well-prepared" when the bulk viscosity tends to infinity, and will be used to eliminate the effect of initial layer and to derive  the convergence rate of the incompressible limit.

The main result of this paper consists of two parts. The first one (cf. Theorem \ref{thm1.1}) is about the global well-posedness of strong solutions with large bulk viscosity of the problem (\ref{1.1})--(\ref{1.5}), and the second one (cf. Theorems \ref{thm1.2} and \ref{thm1.3}) is concerned with the convergence rate of the combined limits from  (\ref{1.1}) to (\ref{1.7}) as $\nu\to\infty$ and $\varepsilon\to0$. To clarify the notations, the solutions of the viscous, non-diffusive and incompressible system (\ref{1.7}) are denoted by $(\eta,v,\chi)$ (see Lemma \ref{lem2.1}).

\begin{thm}\label{thm1.1}  Assume that \eqref{1.9} and \eqref{1.10} hold. Then for any $0<T<\infty$, there exists  a positive number $\nu_0>1$, depending  on $A$, $\mu$, $\zeta$, $\gamma$, $\widetilde\rho$, $\underline\rho$, $\overline\rho$, $K$ and $\|(u_0,w_0)\|_{H^1}$, but not on $\nu$, $\varepsilon$ and $T$, such that the problem \eqref{1.1}--\eqref{1.5} has a global unique strong solution $(\rho,u,w)$ on $\R^2\times[0,T]$, satisfying
\begin{equation}\label{1.11}
\frac{1}{4}\underline{\rho}\les\rho (x,t)\les 2 \overline{\rho}, \; \;\forall \; x\in\R^2,\; t\ges0
\end{equation}
and
\begin{equation}\label{1.12}
\begin{cases}
(\rho-\widetilde\rho, w)\in L^\infty(0,T; H^1\cap W^{1,4}),\;\; u\in L^\infty(0,T; H^1)\cap L^2(0,T; H^2),\\[2mm]
(u_t, w_t)\in L^\infty(0,T; L^2),\;\; \nabla u\in L^{1+r}(0,T; L^\infty\cap  W^{1,4}),\;\; \forall\; r\in(0, 1/4),\\[2mm]
\sqrt t u\in L^\infty(0,T; H^2),\;\; \sqrt t u_t\in L^\infty(0,T; L^2)\cap L^2(0,T; H^1),
\end{cases}
\end{equation}
provided $\nu\ges \nu_0$. Moreover, all the norms in \eqref{1.12} are uniform in $\nu$ and $\varepsilon$.
\end{thm}

With the global and uniform (in $\nu$ and $\varepsilon$) estimates in (\ref{1.11}) and (\ref{1.12}), we can pass to the limits as $\nu\to\infty$ and $\varepsilon\to0$ and show that $(\rho,u,w)\to(\eta,v,\chi)$ in some sense by using the standard compactness arguments (see, for example, \cite{DM2019, Danchin2023} and \cite{LZ2026}). Moreover, we can prove the following convergence rate.

\begin{thm}\label{thm1.2} Let the conditions of Theorem \ref{thm1.1} be in force. Then, as $\nu\to\infty$ and $\varepsilon\to0$, the solution $(\rho,u,w)$ of the problem \eqref{1.1}--\eqref{1.5} converges strongly to the one $(\eta,v,\chi)$ of the system \eqref{2.1} with $\eta_0=\rho_0$, $\chi_0=w_0$ and $v_0=\mathcal{P}u_0$ in some sense on $\R^2\times[0,T]$, satisfying
\begin{equation}\label{1.13}
\|(\rho-\eta,\mathcal{P}u-v,w-\chi)(t)\|_{L^2}^2+\int_0^t\|\nabla(\mathcal{P}u-v)\|_{L^2}^2\md s
 \les C\left(\nu^{-\frac{1}{4}}+\varepsilon\right)
\end{equation}
and
\begin{equation} \label{1.14}
\|\nabla (\mathcal{Q}u)(t)\|_{L^2}^2
+\int_0^t\|\nabla(\mathcal{Q}u)\|_{H^1}^2\md s\les  C\nu^{-1},\;\; \forall\; 0\les t\les T,
\end{equation}
where $\mathcal{P}u$ and $\mathcal{Q}u$ are the Helmholtz projectors on the divergence-free field and the potential vector field, defined respectively by
$$ \mathcal{P} u\triangleq(\mathbb{Id}+\nabla(-\Delta)^{-1}{\rm{div}} )u\quad{\text{and}}\quad \mathcal{Q}u\triangleq-\nabla(-\Delta)^{-1}{\rm{div}} u.$$
Here and hereafter, the same letter $C$ denotes the positive constant, which is independent of $\nu$ and $\varepsilon$, but may change from line to line.
\end{thm}

As that in \cite{Danchin2023,LZ2026}, it is expected that the convergence rate of the incompressible limit as $\nu\to\infty$ is of the value $\nu^{-\frac{1}{2}}$. However, the present situation for the Cauchy problem on the whole domain is slightly different from the periodic problem on the torus. The main difference lies in the fact that the Poincar${\rm{\acute{e}}}$ type inequality $\|\mathcal{Q}u\|_{L^2}\les C\|\nabla(\mathcal{Q}u)\|_{L^2}$ holds for the periodic case (see \cite{Danchin2023,LZ2026}), but fails for the Cauchy problem. Indeed, it holds for the periodic problem that
$$
\|\mathcal{Q}u\|_{L^2}\les C\|\nabla(\mathcal{Q} u)\|_{L^2}\leq C\|\divg u\|_{L^2}\les C\nu^{-\frac{1}{2}},
$$
by which the convergence rate of $\|\mathcal{P}u-v\|_{L^2}$ can be shown to be of the order $\nu^{-\frac{1}{2}}$ (faster than $\nu^{-\frac{1}{4}}$ in (\ref{1.13})). In the next theorem, we shall prove that if the incompressible density decays to the  reference density $\widetilde\rho$  sufficiently fast at infinity, then the convergence rate of the incompressible limit is still of the value  $\nu^{-\frac{1}{2}}$.

\begin{thm}\label{thm1.3} In addition to the conditions of Lemma \ref{lem2.1}, assume that
\begin{equation}\label{1.15}
\bar x^\theta(\eta_0-\widetilde\rho)\in L^4\;\; {\text{with}}\;\; \frac{1}{2}<\theta<\infty \;\; {\text{and}}\;\; \bar x\triangleq (e+|x|^2)^{\frac{1}{2}}\log^2 (e+|x|^2).
\end{equation}
Then for any $0<T<\infty$, the  incompressible density $\eta=\eta(x,t)$ of the problem \eqref{2.1}--\eqref{2.2},  besides \eqref{2.4},  satisfies
 \begin{equation}\label{1.16}
\bar x^\theta(\eta-\widetilde\rho)\in L^\infty(0, T; L^4).
\end{equation}
Moreover, the following refined convergence rates hold for any $0\les t\les T$,
\begin{equation}\label{1.17}
\|(\rho-\eta,\mathcal{P}u-v,w-\chi)(t)\|_{L^2}^2+\int_0^t\|\nabla(\mathcal{P}u-v)\|_{L^2}^2\md s
 \les C\left(\nu^{-\frac{1}{2}}+\varepsilon\right).
\end{equation}
and
\begin{equation}\label{1.18}
\|\bar x^{-1}(\mathcal{Q}u)(t)\|_{L^2}^2+\|\bar x^{-\theta} (\mathcal{Q}u)(t)\|_{L^4}^2\les C\|\nabla(\mathcal{Q} u)\|_{L^2}\les C\nu^{-1}.
\end{equation}
\end{thm}

\begin{re}\label{re1.1} As a by-product of Theorems \ref{thm1.1} and \ref{thm1.2}, we obtain the global existence of strong solutions of the 2D compressible micropolar equations with/without angular viscosity (i.e. either \eqref{1.1} or \eqref{1.1} with $\varepsilon=0$), provided the bulk viscosity  is suitably large.  In particular, the latter one extends the results in \cite{DZ2010,LXZ2026} from the incompressible system to the compressible setting.
\end{re}

\begin{re}\label{re1.2}
By means of the global uniform-in-$\varepsilon$ estimates established in \eqref{1.11} and \eqref{1.12}, one can check the vanishing limit of angular viscosity  (i.e.  $\varepsilon\to0$) for the Cauchy problem \eqref{1.1}--\eqref{1.5}  without the extra assumption \eqref{1.8}. The convergence analysis can be done in a standard way by considering the differences of two solutions and using the standard $L^2$-method.  This also improves/extends the results in \cite{Chen2015,LXZ2026}, concerning  the vanishing limit of angular viscosity for the 2D homogeneous/nonhomogeneous incompressible micropolar equations,   to the 2D compressible case. We omit the details of the proof for simplicity.
\end{re}

\begin{re}\label{re1.3}We mention here that the condition $(\rho_0-\widetilde\rho, w_0)\in W^{1,4}$ is not essential and is only used to simplify the analysis. Indeed, it can be replaced by $W^{1,q}$ for any $2<q<\infty$ after some minor modifications (see Lemma \ref{lem4.2}).
\end{re}

We end this section with some comments on the proofs of Theorems \ref{thm1.1}--\ref{thm1.3}. The proofs rely on the global a priori estimates uniform in both the large bulk viscosity
$\nu$ and small angular viscosity $\varepsilon$. Because the local well-posedness of strong solutions follows by a standard fixed-point argument (see, e.g. \cite{MN1980}), the main task is to extend the solution to any finite time and to pass to the limits as $\nu\to\infty$ and $\varepsilon\to0$. The first key step is to establish the global uniform lower and upper bounds for the density via a bootstrap argument (see Proposition \ref{pro3.1}) by choosing $\nu> 1$ suitably large.  A crucial  ingredient is the  ``effective viscous flux" $F=\nu\divg u-(P(\rho)-P(\widetilde\rho))$ (see (\ref{3.6})), which enables us to reformulate the continuity equation (\ref{1.1})$_1$ into a transport equation with dissipation (see (\ref{3.9})) and reduces the density bounds to the control of $\|F\|_{L^\infty}$. Since $F$ satisfies the elliptic equation  $\Delta F=\divg (\rho \dot u)$ with $\dot u\triangleq u_t+u\cdot\nabla u$ being the material derivative, this control is further converted into the estimates of the material derivative and its gradient (see Lemmas \ref{lem3.1} and \ref{lem3.2}), where the weighted factor $\sigma(t)=\min\{1,t\}$ is used to eliminate the effects of initial layer. In contrast to that in \cite{LZ2026}, to deal with the coupling between the velocity and the micro-rotational velocity (especially, the handling of the term $\nabla^\bot w$ in (\ref{1.1})$_2$), we have to assume that the density is strictly away from vacuum.

With the density bounds at hand, Section \ref{sec4} derives the global higher-order regularity estimates of the solutions. We first prove the global bounds of $\|(\divg u, \nabla^\bot\cdot u, w)\|_{L^p_tL^\infty_x}$ for some $1<p<\infty$ (see Lemma \ref{lem4.1}), which are essential for the gradient estimates of the density and the micro-rotational velocity, by making full use of the ``effective viscous flux" $F$, the modified vorticity $H\triangleq \nabla^\bot\cdot u-2\zeta/(\mu+\zeta) w$ (see (\ref{3.35})) and the maximal principle. Then, based on the Beale–Kato–Majda type inequality, we estimate the $L^p$-norms ($p=2,4$) of $(\nabla \rho,\nabla w)$ and obtain the global higher-order estimates in Lemmas \ref{lem4.2} and \ref{lem4.3}, which are uniform in
$\nu$ and $\varepsilon$ and yield the global well-posedness stated in Theorem \ref{thm1.1}.

Based on the global uniform bounds achieved, the combined limits as $\nu\to\infty$ and $\varepsilon\to0$ can be justified by the standard compactness arguments and the limiting system is just the viscous and non-diffusive incompressible micropolar system (\ref{2.1}) with the initial data $\eta_0=\rho_0$, $v_0=\mathcal{P}u_0$ and $\chi_0=w_0$. To obtain the convergence rate in Theorem \ref{thm1.2}, we compare the compressible solution $(\rho,u,w)$ with the incompressible solution $(\eta,v,\chi)$ and estimate the differences $(\rho-\eta,\mathcal{P}u-v,w-\chi)$ in $L^2$-norms. The key idea is to introduce the auxiliary ``incompressible" density $\rho^*$, satisfying $\rho^*_t+\mathcal{P}u\cdot\nabla\rho^*=0$, and decompose $\rho-\eta=(\rho-\rho^*)+(\rho^*-\eta)$, which separates the effect of  acoustic part $\mathcal{Q}u$ from the velocity difference $\mathcal{P}u-v$. By deriving the equations for the two differences and using the uniform estimates from Sections \ref{sec3} and \ref{sec4}, we obtain the
$t$-weighted and singular $t$-weighted estimates stated in (\ref{5.10}), (\ref{5.12}) and (\ref{5.14}). It is worth pointing out that the $t$-growth estimate $\|\phi\|_{L^2}\les C(T)t$ ensures the singular $t$-weighted estimate $t^{-1}\|\phi\|_{L^2}^2\to0$ as $t\to0$, the latter of which will be used  to compensate for the absence of the estimate $\|u_t\|_{L^2_tH^1_x}$ resulting from the low regularity of the initial data. Then, subtracting the momentum and micro-rotational equations then gives the differential inequality (\ref{5.22}), and the Gronwall inequality leads to the convergence rate (\ref{1.13}), while (\ref{1.14}) follows from the uniform bound on $\divg u$ and the div–curl estimates.

Finally, Theorem \ref{thm1.3} improves the rate to the convergence rate of the incompressible limit from $\nu^{-1/4}$ to $\nu^{-1/2}$ under the additional decay condition (\ref{1.15}).
The proof relies on the weighted estimate (\ref{1.16}) for
$\eta-\widetilde\rho$ established in Lemma \ref{lem5.1}, which provides a refined control of the term $\tilde A_1$ (cf. (\ref{5.18})) in the velocity difference equation. More precisely, this weighted estimate allows one to exploit the fast decay of $\eta-\widetilde\rho$ at infinity and to handle the contribution of $(\eta-\widetilde\rho)(\mathcal{Q}u)(\mathcal{P}u_t-v_t)$ more efficiently than that in the proof of Theorem \ref{thm1.2}. The Hardy inequality on $\R^2$ then plays a crucial role in handling the singular weight $\bar x^{-\theta}$. This refined analysis, together with the decay of $\nabla(\mathcal{Q}u)$, leads to the faster convergence rates stated in   (\ref{1.17}) and (\ref{1.18}).

The rest of the paper is organized as follows. In Section \ref{sec2}, we collect some preliminary results and elementary inequalities that will be used repeatedly in the sequel. In Section \ref{sec3}, we prove  the global uniform lower and upper bounds for the density and derive  the lower-order estimates of the velocity and the micro-rotational velocity. These estimates are global in time and uniform in $\nu$ and $\varepsilon$. Section \ref{sec4} is devoted to the global higher-order regularity estimates of the solutions, which are valid on $[0,T]$ for any $0<T<\infty$ and uniform in $(\nu,\varepsilon)$ as well. The proofs of Theorems \ref{thm1.1}, \ref{thm1.2} and \ref{thm1.3} will be done in  Subsections \ref{sec5.1}, \ref{sec5.2} and \ref{sec5.3}, respectively.

\section{Preliminaries}\label{sec2}
In this section, we collect some known results and elementary inequalities. We begin with the global existence theorem of strong  solutions
of the 2D micropolar equations  for the nonhomogeneous incompressible fluid with zero angular viscosity. As aforemention in Section \ref{sec1}, to avoid the notational confusions, we rewrite   (\ref{1.7})  in the form:
\begin{align}
\begin{cases}\label{2.1}
 \eta_t+v\cdot\nabla \eta=0,\\[1mm]
  \eta v_t+\eta v\cdot\nabla v+\nabla \Pi =(\mu+\zeta)\Delta v
-2\zeta\nabla^{\bot} \chi ,\\[1mm]
\eta  \chi_t+\eta v\cdot\nabla \chi  +4\zeta \chi= 2\zeta\nabla^{\bot}\cdot v,\\[1mm]
\divg v=0,
\end{cases}
\end{align}
subject to the following initial conditions:
\begin{align}\label{2.2}
\begin{cases}
(\eta,v,\chi)|_{t=0}=(\eta_0,v_0,\chi_0)(x),\quad x\in\R^2,\\[1mm]
(\eta,v,\chi)(x,t)\rightarrow (\tilde\rho, 0, 0)\quad{\text{as}}\quad |x|\to\infty,
\end{cases}
\end{align}
where $0<\underline\rho< \tilde\rho < \overline\rho<\infty$   are the same positive constants as those in (\ref{1.9}).

\begin{lem}\label{lem2.1}   Assume that
\begin{equation}\label{2.3}
\underline\rho \leq \eta_0\leq \overline\rho, \quad(\eta_0-\widetilde\rho,\chi_0) \in H^1\cap W^{1,4},\quad v_0\in H^1,\quad
{\rm{div}} v_0=0.
\end{equation}
Then  for any  $0<T<\infty$, there exists a global unique strong solution $(\eta,v,\chi)$ to the problem \eqref{2.1}--\eqref{2.2} on $\R^2\times[0,T]$, satisfying $\eta(x,t)>0$ for all $(x,t)\in\R^2\times[0,T]$, and
\begin{equation}\label{2.4}
\begin{cases}
(\eta-\widetilde\rho, \chi) \in L^\infty(0,T; H^1\cap W^{1,4}),\quad v\in L^\infty(0,T; H^1)\cap L^2(0,T;H^2), \\[1mm]
(v_t, \chi_t)\in L^\infty(0,T; L^2),\;\; \nabla v\in L^{1+r}(0,T; L^\infty\cap  W^{1,4}),\;\; r\in(0, 1/4),\\[2mm]
\sqrt t v\in L^\infty(0,T; H^2),\;\; \sqrt t v_t\in L^\infty(0,T; L^2)\cap L^2(0,T; H^1).
\end{cases}
\end{equation}
\end{lem}
\begin{proof}Since the local well-posedness of strong solutions can be shown  in a standard way (see, e.g. \cite{BRF2003}), it suffices to prove the global estimates (\ref{2.4}) for the problem (\ref{2.1})--(\ref{2.2}).  Indeed, following the arguments in \cite[Lemma 3.1]{LXZ2026} step by step, we can show that for any $0\les t\les T$ with $0<T<\infty$,
\begin{equation}\label{2.5}
\underline\rho\les\eta(x,t)\les\overline\rho,\quad\forall\; x\in\R^2,\; t\ges0
\end{equation}
and
\begin{equation}\label{2.6}
\begin{aligned}
&\|(\eta-\widetilde\rho)(t)\|_{L^2}+\left\|v (t)\right\|_{H^1}+\|\chi (t)\|_{L^2\cap L^4} \\
&\quad+\int_0^t \left(\left\|\left( \dot v,\dot \chi\right)\right\|_{L^2}^2+\|\nabla v\|_{L^4}^4 +\|\nabla \Pi\|_{L^2}^2\right)\md s\leq C(T),
\end{aligned}
\end{equation}
where $\dot v\triangleq v_t+v\cdot\nabla v$ and $\dot \chi\triangleq \chi_t+v\cdot\nabla \chi$. Moreover, analogously to the derivation of (3.45) in \cite[Lemma 3.2]{LXZ2026}, we have
\begin{equation}\label{2.7}
\begin{aligned}
&\frac{1}{2}\frac{\md}{\md t}\left\| \sqrt\eta\dot v\right\|_{L^2}^2+\frac{\mu+\zeta}{2}\|\nabla^\bot\cdot \dot v\|_{L^2}^2\\
&\quad\leq \frac{\md}{\md t}\int   \Pi \partial_iv_j\partial_j v_i\mdx+C\left(1+ \left\|\dot \chi\right\|_{L^2}^2+\|\nabla v\|_{L^4}^4\right) +C\left\| \sqrt\eta \dot v\right\|_{L^2}^4 .
\end{aligned}
\end{equation}

By virtue of (\ref{2.6}) and (\ref{2.14}) (see Lemma \ref{lem2.5} below), we know that
\begin{align*}
&\int   \Pi \partial_iv_j\partial_j v_i\mdx \les C\|\Pi\|_{{\rm{BMO}}}\|\partial_iv_j\partial_j v_i\|_{\mathcal{H}^1}\\
&\quad \les C\|
\nabla \Pi\|_{L^2}\|\nabla v\|_{L^2}^2\les C\|\nabla \Pi\|_{L^2}\les C\|\sqrt\eta\dot v\|_{L^2},
\end{align*}
since it is easily derived from (\ref{2.2}) that
$$
\Delta\Pi=-\divg (\eta \dot v),
$$
which, combined with (\ref{2.5}), yields
$$
\|\nabla\Pi\|_{L^2}\les C\|\sqrt\eta\dot v\|_{L^2}.
$$

With the help of (\ref{2.5}), (\ref{2.6}) and the above observations, we obtain from (\ref{2.7}) by multiplying it by $t$ and integrating  over $(0,t)$ that
$$
t\|\dot v(t)\|_{L^2}^2+\int_0^t s\|\nabla^\bot\cdot \dot v\|_{L^2}^2\md s\leq C(T),\quad\forall\ 0\les t\les T,
$$
and hence, it follows from (\ref{2.6}) and (\ref{2.12}) (see Lemma \ref{lem2.3} below) that
\begin{equation}
t\|\dot v(t)\|_{L^2}^2+\int_0^ts\|\nabla  \dot v\|_{L^2}^2\md s\leq C(T),\quad\forall\ 0\les t\les T,\label{2.8}
\end{equation}
where we have used the divergence-free condition $\divg v=0$ to get that
$$
\|\divg \dot v\|_{L^2}\les \| |\nabla v|^2\|_{L^2}\leq \|\nabla v\|_{L^4}^2.
$$

Using (\ref{2.5}), (\ref{2.6}) and (\ref{2.8}), similarly to the proof of Lemma \ref{lem4.1}, we have
\begin{equation}\label{2.9}
\int_0^T\left(\|\dot v\|_{L^q}^\frac{q+1}{q}+\|\nabla^\bot\cdot v\|_{L^\infty}^\frac{q+1}{q}+\|\chi\|_{L^\infty}^\frac{q+1}{q}\right)\les C(T),\quad\forall\ 2<q<\infty.
\end{equation}
Thus, based upon  (\ref{2.5}), (\ref{2.6}), (\ref{2.8}) and (\ref{2.9}), we can make use of the BKM's inequality (cf. Lemma \ref{lem2.4} below) and the standard regularity theory of Stokes equations to deduce in a manner similar to the proof of Lemma \ref{lem4.2} that
$$
\|(\nabla\eta,\nabla\chi)(t)\|_{L^2}+\|(\nabla\eta,\nabla \chi)(t)\|_{L^4}\leq C(T),\quad\forall\ 0\les t\les T.
$$
Then, analogously to the proof of Lemma \ref{lem4.3}, we arrive at all the desired estimates stated in (\ref{2.4}) eventually.
\end{proof}

\begin{re} In \cite{LXZ2026}, the authors actually studied the vanishing limit of angular viscosity (i.e.  $\varepsilon\to 0$) of the global classical solutions to an initial-boundary value problem of \eqref{2.1} with large data on a smooth bounded  domain $\Omega$, subject to the Navier-slip type boundary condition for the velocity (i.e. $v\cdot n=0$ and $\nabla^\bot\cdot v=0$ on $\partial\Omega$), the Dirichlet boundary condition for the micro-rotational velocity (i.e. $\chi=0$ on $\partial\Omega$). It was assumed in \cite{LXZ2026} that $(\eta_0,\chi_0)\in W^{1,4}(\Omega)$ and $v_0\in  H^2(\Omega)$, the latter of which is stronger than the on in Lemma \ref{lem2.1} (c.p.  $v_0\in H^1$).
As an immediate consequence, the limiting system (i.e. the corresponding initial-boundary value problem of \eqref{2.1}) has a global unique classical solution $(\eta,v,\chi)$  on $\Omega\times[0,T]$ for any $0<T<\infty$, satisfying \eqref{2.4}  and the additional regularity,
$$
v\in L^\infty(0,T; H^2)\cap L^2(0,T; W^{2,4}),\quad
v_t \in L^\infty(0,T; L^2)\cap L^2(0,T; H^1),
$$
which can be achieved by using the standard Stokes estimates (see, e.g. \cite{Galdi,Te1984}). Indeed, compared with that in \cite{LXZ2026}, the Cauchy problem \eqref{2.1}--\eqref{2.2} is much easier to handle, since no boundary term arises in the calculations.
\end{re}

The following Gagliardo-Nirenberg-Sobolev's inequalities (see, e.g.  \cite{Nirenberg}) will be frequently used in the derivations of the global estimates.

\begin{lem}\label{lem2.2} For $p\geq2$, $q>1$ and $r>2$, assume that $f\in H^1 $ and $g\in L^q \cap W^{1,r}$. Then there  exists  a  positive constant $C$, depending on $p$, $q$ and $r$, such that
\begin{align}\label{2.10}
\|f\|_{L^p}&\leq C \|f\|_{L^2}^{\frac{2}{p}}\|\nabla f\|_{L^2}^{\frac{p-2}{p}},
\\ \label{2.11}
\|g\|_{L^\infty}&\leq C\|g\|_{L^q}^{\frac{q(r-2)}{2r+q(r-2)}}\|\nabla g\|_{L^r}^{\frac{2r}{2r+q(r-2)}}.
\end{align}
\end{lem}

Next, let us recall the standard ``div-curl" estimates (cf.  \cite{Brezis1974}).
\begin{lem}\label{lem2.3}
For any $k\in \mathbb{Z}^+$  and  $1<p<\infty$, there exists a positive constant $C$, depending only on $k$ and $p$, such that
\begin{align}\label{2.12}
\|\nabla u\|_{W^{k,p}}\leq C(\|\nabla\cdot u\|_{W^{k,p}}+\|\nabla^\bot\cdot u\|_{W^{k,p}}),\quad\forall\ u\in W^{k+1,p}.
\end{align}
\end{lem}

It is  known that the $W^{k,p}$-estimates stated in  (\ref{2.12}) doesn't holds for the endpoint case $p=\infty$. So, to deal with $\|\nabla u\|_{L^\infty}$, we need the following
  Beale-Kato-Majda's type inequality (see \cite{Kato, Huang}).
\begin{lem}\label{lem2.4}
For any $2<q<\infty$, assume that $\nabla u\in  H^1\cap W^{1,q}$. Then there exists a positive constant $C=C(q)$, such that
\begin{align}\label{2.13}
\|\nabla u\|_{L^\infty}\leq C(\|\nabla\cdot u\|_{L^\infty}+\|\nabla^\bot\cdot u\|_{L^\infty})\ln(e+\|\nabla^2 u\|_{L^q})+C\|\nabla u\|_{L^2}+C.
\end{align}
\end{lem}

Finally, let  $\mathcal{H}^1 $ and  ${\text{BMO}} $ be  the standard Hardy  and   BMO spaces on $\R^2$, respectively. Then,  $(\mathcal{H}^1)^*\cong {\text{BMO}}$ and  $\|u\|_{\text{BMO}}\leq C\|\nabla u\|_{L^2}$ (cf.  \cite{Meyer1975}). Moreover, it follows from   \cite[Theorem II.1]{Coifman1993} that

\begin{lem}\label{lem2.5} Assume that $(u,v)\in L^2$ satisfies  $\nabla^\bot \cdot u=0$ and $\nabla\cdot v=0$ in the sense of distribution. Then,
\begin{equation}\label{2.14}
u\cdot v \in \mathcal{H}^1,\quad \|u\cdot v\|_{\mathcal{H}^1}\leq \|u\|_{L^2}\|v\|_{L^2}.
\end{equation}
\end{lem}

\section{The lower and upper bounds of the density}\label{sec3}
The main purpose of this section is to prove the lower and upper bounds of the density, and to derive the lower-order estimates of the velocity and the micro-rotational velocity. Throughout this section, we suppose  that for some given positive number $0<K<\infty$,
\begin{equation}\label{3.1}
\rho_0-\widetilde\rho \in L^2\cap L^\infty,\ \ \underline\rho \leqslant \rho_0\leqslant \overline\rho,\ \ (u_0,w_0)\in H^1\ \ {\text{and}}\  \ \sqrt\nu\|\divg u_0\|_{L^2}\leqslant K.
\end{equation}
We also assume without loss of generality  that $\nu\ges 1$ and $0<\varepsilon<1$, since we are interested in the global theory of the problem (\ref{1.1})--(\ref{1.5}) with  large bulk viscosity and small angular viscosity.

The uniform lower and upper bounds of the density will be achieved by using the bootstrap argument, based on the following proposition.

\begin{pro}\label{pro3.1} Assume that the pair of functions $(\rho, u,w)$ is a strong solution of the problem \eqref{1.1}--\eqref{1.5} on $\R^2\times[0,T]$,  satisfying
\begin{align}\label{3.2}
\frac{1}{4}\underline{\rho} \les \rho(x,t) \les 2\overline{\rho},\quad\forall\  (x,t)\in\R^2\times[0,T].
\end{align}
Then there exists a positive number $\nu_0>1$, depending  on $A$, $\mu$, $\zeta$, $\gamma$, $\widetilde\rho$, $\underline\rho$, $\overline\rho$, $K$ and the initial norms in \eqref{3.1}, but not on $\nu$, $\varepsilon$ and $T$, such that
\begin{equation}\label{3.3}
\frac{1}{2}\underline{\rho}\les\rho (x,t)\les \frac{3}{2}\overline{\rho}, \quad\forall\  (x,t)\in\R^2\times[0,T].
\end{equation}
provided that $\nu\ges \nu_0$.
\end{pro}

The proof of Proposition \ref{pro3.1} is built upon the following two lemmas, which are concerned with the lower-order estimates of the solutions and the global weighted estimates of the material derivatives of $(u,w)$, respectively. For simplicity,  we denote by $C$ the various positive constants, which may depend  on $\mu$, $\zeta$, $A$,  $\gamma$, $\widetilde\rho$, $\underline\rho$, $\overline\rho$, $K$ and the initial norms in \eqref{3.1}, but not on $\varepsilon$, $\nu$ and $T$. Sometimes, we use $C(\alpha)$ to emphasize the dependence on $\alpha$. Moreover,  we write $\|(f,g)\|_{X}\triangleq\|f\|_{X}+\|g\|_{X}$ and $\|f\|_{X\cap Y}\triangleq \|f\|_{X}+\|f\|_{Y}$ for the normed spaces $X$ and $Y$.

\begin{lem}[Global lower-order esitimates]  \label{lem3.1}Assume that  \eqref{3.2} holds. There exists a positive constant $\nu_1>1$, depending only on $\mu$, $\zeta$, $A$, $\gamma$, $\widetilde\rho$, $\underline\rho$ and $\overline\rho$, but not on $\varepsilon$, $\nu$ and $T$, such that if $\nu\ges\nu_1$, then for any $0\les t\les T$,
\begin{equation}\label{3.4}
\begin{aligned}
&\left(\|\rho-\widetilde\rho\|_{L^2}^2+  \| u \|_{H^{1}}^{2}+\|w \|_{L^2}^2+\|w\|_{L^4}^4 + \nu \|{\rm div} u \|_{L^2}^2+\varepsilon  \|\nabla w \|_{L^2}^2\right)(t)\\
&\quad +\int_0^t\left(\|(w,\nabla u)\|_{L^2}^2+\|(w,\nabla u)\|_{L^4}^4+\| (\dot u, \dot w)\|_{L^2}^2+\nu\|{\rm div} u\|_{L^2}^2\right)\md s\\
&\quad +\int_0^t\left(\nu^{-1} \|P(\rho) -  P(\widetilde\rho) \|_{L^4}^4+\varepsilon\|\nabla w\|_{L^2}^2 + \varepsilon^2\|\nabla^2w\|_{L^2}^2\right)\md s\les C(\underline\rho,\overline\rho).
\end{aligned}
\end{equation}
\end{lem}

To eliminate the initial layer due to the lower regularity of the initial data, we introduce the weighted factor $\sigma(t)\triangleq\min\{1,t\}$.

\begin{lem}[Global weighted estimates]\label{lem3.2} Let $\nu_1>1$ be the positive number determined in Lemma \ref{lem3.1}. Then for any $\nu\ges \nu_1$ and $0\les t\les T$,
\begin{equation}\label{3.5}
\sigma(t)\| \dot{u}(t)\|_{L^2}^2 +\int_0^t\sigma(s)\left( \|\nabla \dot u \|_{L^2}^2+\nu \| \dot V \|_{L^2}^2\right)\md s\les C(\underline\rho,\overline\rho),
\end{equation}
where $V \triangleq {\rm div} u$ and $\dot V \triangleq V_t+u\cdot\nabla V$.
\end{lem}

The proofs of Lemma \ref{lem3.1} and \ref{lem3.2} will be done in the coming  two subsections. With Lemmas \ref{lem3.1} and \ref{lem3.2} at hand, we are now ready to prove Proposition \ref{pro3.1}.

\begin{proof}[Proof of Proposition \ref{pro3.1}] Let $F$ and $\dot f$ be the ``effective viscous flux"  and the material derivative defined respectively by
\begin{equation}\label{3.6}
F \triangleq \nu\divg u-(P(\rho)-P(\widetilde\rho))\quad {\rm{and}}\quad \dot f=f_t+u\cdot\nabla f \triangleq {\rm D}_t f.
\end{equation}
Then it is easily seen from (\ref{1.1})$_2$ that
\begin{equation}
\Delta F=\divg (\rho\dot u)\;\;{\text{with}}\;\;\dot u\triangleq u_t+u\cdot\nabla u,\label{3.7}
\end{equation}
from which and (\ref{3.2}), we deduce
\begin{equation}\label{3.8}
\|\nabla F\|_{L^p}\leq C\|\rho\dot u\|_{L^p}\leq C(\bar\rho)\|\dot u\|_{L^p},\quad \forall\ 1<p<\infty.
\end{equation}

Next, in terms of (\ref{3.6})$_1$, we can write (\ref{1.1})$_1$ in the form:
$$
{\rm D}_t\rho=\rho_t+u\cdot\nabla \rho=-\rho\divg u=-\nu^{-1}\rho \left(F+P(\rho)-P(\widetilde\rho)\right),
$$
and hence,
\begin{equation}
{\rm D}_t\rho +\nu^{-1}\rho \left(P(\rho)-P(\widetilde\rho)\right) =-\nu^{-1}\rho F.\label{3.9}
\end{equation}

To be continued, define
$$\Theta(x,t)\triangleq \max\{\rho-\widetilde{\rho},0\}\quad {\text{and}}\quad \widetilde \Theta(x,t)\triangleq\max\{\widetilde{\rho}-\rho,0\}.$$
Then, owing to  the mean-value theorem and (\ref{3.2}), we infer from (\ref{3.9}) that there exists a positive constant $\alpha$, depending only on $\widetilde\rho$, $\underline\rho$ and $\overline\rho$, such that
$$
{\rm D}_t\Theta+ \alpha \nu^{-1}\Theta \les C{\nu}^{-1}|F|\quad {\text{and}}\quad {\rm D}_t\widetilde\Theta+ \alpha \nu^{-1}\widetilde\Theta \les C \nu^{-1}|F|,
$$
and consequently,
\begin{equation}\label{3.10}
\|\Theta(t)\|_{L^\infty}\les \|\Theta_0\|_{L^\infty}+I\quad {\text{and}}\quad \|\widetilde\Theta(t)\|_{L^{\infty}}\les \|\widetilde\Theta_0\|_{L^\infty}+I,
\end{equation}
where
$$
I\triangleq
C\nu^{-1}\int_0^te^{-\frac{\alpha(t-s)}{\nu}}\|F\|_{L^\infty}\md s.
$$

Obviously, it suffices to deal with $\|F\|_{L^\infty}$. To this end, we first make use of (\ref{2.10}), (\ref{2.11}), (\ref{2.12}) and (\ref{3.8}) to get that
\begin{equation}\label{3.11}
\begin{aligned}
\nu^{-1}\|F\|_{L^\infty}&\les C\nu^{-1}\|F\|_{L^2}^{\frac{1}{3}}\|\nabla F\|_{L^4}^\frac{2}{3}\les C\nu^{-\frac{5}{6}}\|\dot u\|_{L^4}^{\frac{2}{3}}\\
&\les C\nu^{-\frac{5}{6}}\left(\|\dot u\|_{L^2}+\|\nabla \dot u\|_{L^2}\right)^{\frac{2}{3}}\\
&\les C\nu^{-\frac{5}{6}}\left(\|\dot u\|_{L^2}+\|\nabla^\bot\cdot \dot u\|_{L^2}+\|\divg \dot u\|_{L^2}\right)^{\frac{2}{3}},
\end{aligned}
\end{equation}
since it follows from (\ref{3.2}) and (\ref{3.4}) that
\begin{equation}\label{3.12}
\nu^{-\frac{1}{2}}\|F\|_{L^2}\les \nu^{-\frac{1}{2}}\left(\nu\|\divg u\|_{L^2}+\|P(\rho)-P(\widetilde\rho)\|_{L^2}\right)\les C.
\end{equation}

Recalling that $V=\divg u$ and $\dot V=V_t+u\cdot\nabla V$, we have
$$
\divg \dot u=\divg u_t+\divg ( u\cdot\nabla  u)=\dot V+\partial_i u_j\partial_j u_i,
$$
from which it is easily seen that
\begin{equation}
\|\dot V\|_{L^2}-C\|\nabla u\|_{L^4}^2\les \|\divg \dot u\|\les \|\dot V\|_{L^2}+C\|\nabla u\|_{L^4}^2.\label{3.13}
\end{equation}
Thus, inserting (\ref{3.13}) into (\ref{3.11}), we obtain
$$
\nu^{-1}\|F\|_{L^\infty}  \les C\nu^{-\frac{5}{6}}\left(\|\dot u\|_{L^2}+\|\nabla u\|_{L^4}^2\right)^{\frac{2}{3}}+ C\nu^{-\frac{5}{6}}\left(\|\nabla^\bot\cdot \dot u\|_{L^2}+\|\dot V\|_{L^2}\right)^{\frac{2}{3}},
$$
and consequently,
\begin{equation}\label{3.14}
\begin{aligned}
I
 & \les C\nu^{-\frac{5}{6}} \int_0^te^{-\frac{\alpha(t-s)}{\nu}} \left(\|\dot u\|_{L^2}+\|\nabla u\|_{L^4}^2\right)^{\frac{2}{3}}\md s\\
&\quad +C\nu^{-\frac{5}{6}}\int_0^te^{-\frac{\alpha(t-s)}{\nu}} \left(\|\nabla^\bot\cdot \dot u\|_{L^2}+\|\dot V\|_{L^2}\right)^{\frac{2}{3}}\md s\triangleq I_1+I_2.
\end{aligned}
\end{equation}

Using (\ref{3.4}) and H\"{o}lder inequality, we find
$$
I_1
  \les C\nu^{-\frac{5}{6}} \left(\int_0^te^{-\frac{3\alpha(t-s)}{2\nu}}\md s \right)^{\frac{2}{3}}\left[\int_0^t \left(\|\dot u\|_{L^2}^2+\|\nabla u\|_{L^4}^4\right)\md s\right]^{\frac{1}{3}} \les C\nu^{-\frac{1}{6}}.
$$
Analogously, letting $ 1\wedge t\triangleq\min\{1,t\}$, by (\ref{3.5}) we derive
\begin{align*}
I_2&\les C\nu^{-\frac{5}{6}}\int_0^t \sigma^{-\frac{1}{3}}(s) e^{-\frac{\alpha(t-s)}{\nu}} \left(\sigma(s)\|\nabla^\bot\cdot \dot u\|_{L^2}^2+ \sigma(s)\|\dot V\|_{L^2}^2\right)^{\frac{1}{3}}\md s\\
&\les C\nu^{-\frac{5}{6}}\left(\int_0^t \sigma^{-\frac{1}{2}}(s) e^{-\frac{3\alpha(t-s)}{2\nu}}\md s\right)^{\frac{2}{3}} \left[\int_0^t\sigma(s)\left(\|\nabla^\bot\cdot \dot u\|_{L^2}^2+ \|\dot V\|_{L^2}^2\right) \md s\right]^{\frac{1}{3}}\\
&\les C\nu^{-\frac{5}{6}}\left(\int_0^{1\wedge t} s^{-\frac{1}{2}} \md s\right)^{\frac{2}{3}}+ C\nu^{-\frac{5}{6}}\left(\int_{1\wedge t}^t   e^{-\frac{3\alpha(t-s)}{2\nu}}\md s\right)^{\frac{2}{3}}\\
&\les C\nu^{-\frac{5}{6}}+C\nu^{-\frac{1}{6}}\les C\nu^{-\frac{1}{6}}.
\end{align*}

Now, plugging the estimates of $I_1$ and $I_2$ into (\ref{3.14}) and combining it with (\ref{3.10}), by (\ref{3.1}) we konw that there exists a positive constant $\Lambda=\Lambda(\underline\rho,\overline\rho)$, depending on $\underline\rho$ and $\overline\rho$, such that
\begin{equation}\label{3.15}
\|\Theta(t)\|_{L^\infty}\les \overline\rho-\widetilde\rho+\Lambda \nu^{-\frac{1}{6}}\quad{\text{and}}\quad \|\widetilde\Theta(t)\|_{L^\infty}\les \widetilde\rho-\underline\rho+\Lambda  \nu^{-\frac{1}{6}}.
\end{equation}
So, if $\nu$ is chosen to be large enough such that
$$
\nu\ges\nu_0\triangleq\max\left\{\nu_1, \left(2\Lambda \underline\rho^{-1}\right)^{6}\right\},
$$
then we conclude from (\ref{3.15}) that
$$
\frac{1}{2}\underline\rho\les \widetilde\rho-(\widetilde\rho-\underline\rho)-\frac{1}{2}\underline\rho \les \rho(x,t)\les\widetilde\rho+(\overline\rho-\widetilde\rho)+\frac{1}{2}\underline\rho\les \frac{3}{2}\overline\rho.
$$
The proof of (\ref{3.3}) is therefore completed.
\end{proof}

\subsection{Proof of Lemma \ref{lem3.1}} This subsection is devoted to the derivation of the global lower-order estimates of the solutions stated in (\ref{3.4}). To clarify the proof, we split it into several steps.

\vskip 2mm

\underline{\bf Step 1. The elementary  $L^2$-energy estimate}

\vskip 2mm

To begin, let
\begin{equation}\label{3.16}
E_0  \triangleq \int  \left( \frac{1}{2} \rho_0 |u_0|^2 + \frac{1}{2} \rho_0 |w_0|^2 + G(\rho_0) \right) \md x
\end{equation}
where  $G(\rho)$ is the potential energy given by
\begin{equation}\label{3.17}
G(\rho)\triangleq \rho\int_{\widetilde{\rho}}^\rho\frac{P(s)-P(\widetilde{\rho})}{s^2} \md s.
\end{equation}

It is easy to check that
\begin{equation}\label{3.18}
G(\widetilde{\rho})=G'(\widetilde{\rho})=0\quad{\text {and}}\quad
 \rho G'(\rho)=G(\rho)+P(\rho)-P(\widetilde{\rho}),
\end{equation}
which, together with (\ref{3.2}), implies that there exists a positive constant $C$, depending only on $\underline\rho$, $\overline\rho$ and $\widetilde\rho$, such that for any $0\les t\les T$,
\begin{equation}\label{3.19}
C^{-1}\|\rho-\widetilde\rho\|_{L^2}^2\les \|G(\rho)\|_{L^1}\les C\|\rho-\widetilde\rho \|_{L^2}^2.
\end{equation}

In view of (\ref{1.1})$_1$ and the identity,
$$
\nabla^\bot(\nabla^\bot\cdot u)=\Delta u-\nabla \divg u,
$$
 we can rewrite (\ref{1.1})$_2$ in the form ($ \nu=2\mu+\lambda$):
\begin{equation}
 \rho u_t+\rho u \cdot\nabla u+\nabla P(\rho )=(\mu+\zeta)\nabla^\bot(\nabla^\bot\cdot u)+
\nu\nabla\divg u-2\zeta\nabla^{\bot} w.\label{3.20}
\end{equation}

Multiplying (\ref{1.1})$_1$, (\ref{3.20}) and (\ref{1.1})$_3$ by $G'(\rho)$, $u$ and $w$, respectively, integrating by parts over $ \R^2$, and adding them together, by (\ref{3.18})$_2$ we obtain
\begin{equation}\label{3.21}
\begin{aligned}
&\frac{\md}{\md t}\left(\frac{1}{2}\|\sqrt\rho u\|_{L^2}^2+\frac{1}{2}\|\sqrt\rho w\|_{L^2}^2+\|G(\rho)\|_{L^1}\right)+ \nu\|\divg u\|_{L^2}^2\\
&\quad+\mu\|\nabla^\bot\cdot u\|_{L^2}^2+\zeta\|\nabla^\bot\cdot u-2w\|_{L^2}^2+\varepsilon\|\nabla w\|_{L^2}^2 =0,
\end{aligned}
\end{equation}
which, integrated over $(0,t)$ and combined  with (\ref{3.2}), yields
\begin{equation}
\begin{aligned}\label{3.22}
&\left(\| u\|_{L^2}^2+ \| w \|_{L^2}^2+\|G(\rho)\|_{L^1}\right)(t)+\int_0^t\left( \nu \|\divg u\|_{L^2}^2  +\mu\|\nabla^\bot\cdot u\|_{L^2}^2\right)\md s\\
&\qquad + \zeta\int_0^t\|\nabla^\bot \cdot u-2 w \|_{L^2}^2\md s +\varepsilon\int_0^t\|\nabla w\|_{L^2}^2 \md s\les C E_0,
\end{aligned}
\end{equation}
where $E_0$ and $G(\rho)$ are the ones defined in (\ref{3.16}) and (\ref{3.17}), respectively.

\vskip 2mm

\underline{\bf Step 2. A suitable $H^1$-energy functional}

\vskip 2mm

This step aims to derive the differential identity for the energy functional that contains the $H^1$-information of the velocity,
\begin{equation}\label{3.23}
\begin{aligned}
\mathcal{E}(t)&\triangleq \frac{1}{2}\left(\|\sqrt\rho u\|_{L^2}^2+ \|\sqrt\rho w\|_{L^2}^2\right)+\left(1- 2 \gamma P(\widetilde\rho) \nu^{-1} \right)\|G(\rho)\|_{L^1} \\
&\quad + \nu^{-1}\|F\|_{L^{2}}^{2}+\mu\|\nabla^\bot\cdot u\|_{L^{2}}^{2}+\zeta\|\nabla^\bot\cdot u-2w\|_{L^2}^2+\varepsilon\|\nabla w\|_{L^2}^2,
\end{aligned}
\end{equation}
where $F$ is the ``effective viscous flux" defined in (\ref{3.6})$_1$. For simplicity, we also introduce the following dissipation functional,
\begin{equation}\label{3.24}
\begin{aligned}
\mathcal{F}(t)&\triangleq \mu\|\nabla^\bot\cdot u\|_{L^2}^2+\nu\|\divg u\|_{L^2}^2+\zeta\|\nabla^\bot\cdot u-2w\|_{L^2}^2 \\
&\quad+2\left(\|\sqrt{\rho}\dot{u}\|_{L^{2}}^{2}+\|\sqrt\rho \dot w\|_{L^2}^2\right) +\varepsilon\|\nabla w\|_{L^2}^2.
\end{aligned}
\end{equation}

First,  it is easily derived from  (\ref{3.20}) that
\begin{equation}
\nabla F+(\mu+\zeta)\nabla^\bot(\nabla^\bot\cdot u)=\rho\dot u+2\zeta\nabla^\bot w,\label{3.25}
\end{equation}
where $\dot u\triangleq u_t+u\cdot\nabla u$. By direct calculation, we have
\begin{equation}
\nabla^\bot\cdot\dot u=\nabla^\bot\cdot(u_t+u\cdot\nabla u)= {\rm D}_t(\nabla^\bot \cdot u)+(\nabla^\bot\cdot u)\divg u\label{3.26}
\end{equation}
and
\begin{equation}\label{3.27}
\begin{aligned}
\divg \dot u&=\divg (u_t+u\cdot\nabla u) ={\rm D}_t(\divg u)+\partial_iu_j\partial_j u_i\\
&=\nu^{-1}{\rm D}_t \left( F+ P(\rho)-P(\widetilde\rho)\right) +2(\nabla u_1)\cdot (\nabla^\bot u_2)+(\divg u)^2.
\end{aligned}
\end{equation}

Next, multiplying (\ref{3.25}) by $2\dot u$ in $L^2$ and integrating by parts over $\R^2$, by (\ref{3.26}) and (\ref{3.27}) we deduce
\begin{equation}
\begin{aligned}\label{3.28}
 &\frac{\md}{\md t}\left( \nu^{-1}\|F\|_{L^{2}}^{2}+(\mu+\zeta) \|\nabla^\bot\cdot u\|_{L^{2}}^{2}\right)+2\|\sqrt{\rho}\dot{u}\|_{L^{2}}^{2} \\
&\quad  =-4\zeta\int\nabla^\bot w \cdot\dot{u}\mdx-(\mu+\zeta)\int (\nabla^\bot\cdot u)^{2} (\divg u) \mdx\\
&\qquad + \nu^{-1} \int F^{2} (\divg u) \mdx- 4 \int F (\nabla u_{1}) \cdot (\nabla^{\bot} u_{2}) \mdx
\\
&\qquad-2 \int F(\divg u)^{2} \mdx + 2 \gamma \nu^{-1}\int P(\rho) F (\divg u) \mdx,
\end{aligned}
\end{equation}
due to  (\ref{1.1})$_1$ and the following identity,
\begin{equation}
{\rm D}_t \left(P(\rho)-P(\widetilde\rho)\right)=P(\rho)_t+u\cdot\nabla P(\rho)=-\gamma P(\rho)\divg u.\label{3.29}
\end{equation}

Keeping in mind that  $\dot w \triangleq w_t+u\cdot\nabla w$, we have by integration by parts that
\begin{align*}
\int\nabla^\bot w \cdot\dot{u}\mdx & =  \frac{\md }{\md t}\int\nabla^\bot w \cdot u \mdx-\int \left( \nabla^\bot w_t \cdot u -\nabla^\bot w\cdot (u\cdot\nabla )u \right) \mdx\\
   & =  -\frac{\md }{\md t}\int w (\nabla^\bot \cdot u) \mdx+\int \left(  w_t +u\cdot\nabla w\right) (\nabla^\bot\cdot u) \mdx\\
   & =  -\frac{\md }{\md t}\int w (\nabla^\bot \cdot u) \mdx+\int \dot w (\nabla^\bot\cdot u)  \mdx,
  \end{align*}
which, together  with (\ref{1.1})$_3$, gives
\begin{equation}\label{3.30}
\begin{aligned}
 -4\zeta\int\nabla^\bot w \cdot\dot{u}\mdx
 &= 4\zeta\frac{\md }{\md t}\int w (\nabla^\bot \cdot u) \mdx-2 \int \dot w \left(\rho \dot w+4\zeta w-\varepsilon\Delta w\right)  \mdx\\
  & =\frac{\md }{\md t}\int\left( 4\zeta w (\nabla^\bot \cdot u) -4\zeta w^2-\varepsilon|\nabla w|^2\right)\mdx\\
  &\quad- \int \left(2\rho|\dot w|^2-4\zeta w^2\divg u-2\varepsilon u\cdot\nabla w\Delta w\right)\mdx.
 \end{aligned}
\end{equation}

It follows from  (\ref{1.1})$_1$ and (\ref{3.18})$_2$ that
$$
\partial_t G(\rho)+\divg (uG(\rho))=-\left(P(\rho)-P(\widetilde\rho)\right) \divg u,
$$
and hence, by (\ref{3.6})$_1$ we obtain
\begin{equation}\label {3.31}
\begin{aligned}
& 2\gamma\nu^{-1}\int P(\rho) F (\divg u)\mdx \\
&\quad = 2\gamma\nu^{-1}\int P(\rho) \left(\nu\divg u-(P(\rho)- P(\widetilde \rho)\right) (\divg u) \mdx\\
&\quad =2\gamma\nu^{-1} P(\widetilde\rho) \frac{\md}{\md t}\int G(\rho)\mdx+2\gamma \int  P( \rho)(\divg u)^2\mdx\\
&\qquad - 2\gamma\nu^{-1}\int  \left(P(\rho)- P(\widetilde \rho)\right)^2 (\divg u)\mdx.
\end{aligned}
\end{equation}

Thus, inserting  (\ref{3.30}) and (\ref{3.31}) into (\ref{3.28}) and adding it to (\ref{3.21}), by (\ref{3.23}) and (\ref{3.24}) we arrive at
\begin{equation}\label{3.32}
\begin{aligned}
 \mathcal{E}'(t)+\mathcal{F}(t)
 & = 2\varepsilon \int u\cdot\nabla w\Delta w \mdx+4\zeta\int  w^2(\divg u) \mdx  \\
&\quad -(\mu+\zeta)\int (\nabla^\bot\cdot u)^{2} (\divg u) \mdx + \nu^{-1} \int F^{2} (\divg u) \mdx\\
&\quad - 4 \int F (\nabla u_{1}) \cdot (\nabla^{\bot} u_{2}) \mdx
 -2 \int F(\divg u)^{2} \mdx \\
 &\quad
 - 2 \gamma  \nu^{-1}\int  \left(P(\rho)-P(\widetilde \rho) \right)^2 (\divg u)\mdx +2 \gamma \int  P( \rho)(\divg u)^2\mdx\\
 & \triangleq I_1+\ldots+I_8.
\end{aligned}
\end{equation}
Obviously, if $\nu>1$ is chosen to be large enough (e.g. $\nu\geq 4\gamma P(\widetilde \rho)$), then the functional $\mathcal{E}(t)$ is nonnegative and contains the information of $\|\nabla u\|_{L^2}$, due to Lemma \ref{lem2.3}, (\ref{3.2}), (\ref{3.19}) and (\ref{3.22}).

\vskip 2mm

\underline{\bf Step 3. The estimates of $I_1$--$I_8$}

\vskip 2mm

The right-hand side of (\ref{3.32}) will be estimated term by term. First,
based upon integration by parts, we have
\begin{equation}\label{3.33}
\begin{aligned}
I_1&=-2\varepsilon\int\left(\partial_i u_j\partial_j w\partial_i w+u_j\partial_{ij}^2w\partial_iw\right)\mdx\\
&=2\varepsilon\int \left(w\nabla\divg u \cdot \nabla w+ w\partial_i u_j \partial^2_{ij} w +\frac{1}{2} (\divg u) |\nabla w|^2 \right)\mdx\\
& =-\varepsilon\int (\divg u)|\nabla w|^2 \mdx-2\varepsilon\int w \left((\divg u)\Delta w -\partial_i u_j  \partial^2_{ij} w \right)\mdx\\
&=-\varepsilon \nu^{-1}\int \left(P(\rho)- P(\widetilde \rho)\right)|\nabla w|^2\mdx+ \varepsilon \nu^{-1}\int  w  \nabla F\cdot\nabla w \mdx\\
&\quad + \varepsilon \nu^{-1}\int  w  F\Delta w  \mdx
-2\varepsilon\int w \left((\divg u)\Delta w -\partial_i u_j  \partial^2_{ij} w \right)\mdx \\
&\triangleq I_1^1+\ldots+I_1^4,
\end{aligned}
\end{equation}
since it follows from (\ref{3.6})$_1$ that
\begin{align*}
&\varepsilon\int (\divg u)|\nabla w|^2 \mdx= \varepsilon \nu^{-1}\int \left(F+(P(\rho)- P(\widetilde \rho)\right)|\nabla w|^2\mdx \\
&\quad = \varepsilon \nu^{-1}\int \left(P(\rho)- P(\widetilde \rho)\right)|\nabla w|^2\mdx- \varepsilon \nu^{-1}\int w \left(\nabla F\cdot\nabla w + F\Delta w\right)\mdx.
\end{align*}

It is easily seen from (\ref{3.2})  that
$$
I_1^1\les C \varepsilon\nu^{-1}\|\nabla w\|_{L^2}^2.
$$

From (\ref{1.1})$_3$  and the $H^2$-regularity theory of elliptic system, by (\ref{3.2}) we get
\begin{equation}\label{3.34}
\varepsilon\|\nabla^2 w\|_{L^2}+\varepsilon^{\frac{1}{2}}\|\nabla w\|_{L^2}+\|w\|_{L^2} \leq C\left(\|\sqrt\rho\dot w\|_{L^2}+\|\nabla^\bot\cdot u\|_{L^2}\right),
\end{equation}
which, combined with (\ref{2.10}), (\ref{3.2}), (\ref{3.8})  and Cauchy-Schwarz inequality, gives
\begin{align*}
I_{1}^2
&\les  C \varepsilon \nu^{-1}\|w\|_{L^4} \|\nabla F\|_{L^2}\|\nabla w\|_{L^4} \\
&\les C \varepsilon \nu^{-1}\|w\|_{L^4} \|\nabla F\|_{L^2}\|\nabla w\|_{L^2}^{\frac{1}{2}}\|\nabla^2w\|_{L^2}^{\frac{1}{2}} \\
&\les
C   \varepsilon^\frac{1}{2}{\nu}^{-1}\|w\|_{L^4} \|\sqrt\rho \dot u\|_{L^2}\|\nabla w\|_{L^2}^{\frac{1}{2}}\left(  \|\sqrt\rho\dot w\|_{L^2}^{\frac{1}{2}} +\|\nabla^\bot\cdot u\|_{L^2}^\frac{1}{2}\right)
 \\
&\les
\frac{1}{16}\left(\|\sqrt\rho\dot u\|_{L^2}^2+\|\sqrt\rho\dot w\|_{L^2}^2\right)+C \left( \nu^{-4} \varepsilon^2\|\nabla w\|_{L^2}^2 \|w\|_{L^4}^4+ \|\nabla u\|_{L^2}^{2}\right).
\end{align*}
Analogously,
\begin{align*}
I_{1}^3
&\les  C\varepsilon \nu^{-1}\|w\|_{L^4} \|F\|_{L^4}\|\nabla^2w\|_{L^2} \\
&\les
 C \nu^{-1}\|w\|_{L^4}\|F\|_{L^2}^{\frac{1}{2}}\|\sqrt\rho\dot u\|_{L^2}^{\frac{1}{2}}   \left( \|\sqrt\rho\dot w\|_{L^2} + \|\nabla^\bot\cdot u\|_{L^2} \right)\\
&\les
\frac{1}{16}\left(\|\sqrt\rho\dot u\|_{L^2}^2+\|\sqrt\rho\dot w\|_{L^2}^2\right)  +C \left(\nu^{-4} \|F\|_{L^2}^2 \|w\|_{L^4}^4+  \|\nabla  u\|_{L^2}^{2}\right).
\end{align*}

To estimate $I_1^4$, we have to consider $\|\nabla u\|_{L^4}$. For this purpose,  let
\begin{equation}\label{3.35}
H\triangleq \nabla^\bot\cdot u-\frac{2\zeta}{\mu+\zeta} w.
\end{equation}
Then, applying $\nabla^\bot\cdot $ to both sides of (\ref{3.20}), we find
\begin{equation}\label{3.36}
(\mu+\zeta)\Delta H=\nabla^\bot\cdot (\rho\dot u).
\end{equation}
Hence, similarly to that in (\ref{3.8}), by (\ref{3.2}) we have from (\ref{3.36}) that
\begin{equation}
\|\nabla H\|_{L^p}\les C \|\sqrt\rho\dot u\|_{L^p},\quad \forall\ 1<p<\infty.\label{3.37}
\end{equation}

In terms of (\ref{3.6})$_1$ and (\ref{3.35}), we infer from  (\ref{2.10}) that
\begin{align*}
\|\nabla u\|_{L^4}&\les  C\left(\|\divg u\|_{L^4}+\|\nabla^\bot\cdot u\|_{L^4}\right)\\
&\les  C\nu^{-1}\left(\|F\|_{L^4}+\|P(\rho)- P(\widetilde \rho) \|_{L^4}\right)+ C\left(\|H\|_{L^4}+\|w\|_{L^4}\right)\\
&\les  C\left(\nu^{-1} \|F\|_{L^2}^{\frac{1}{2}}+\|H\|_{L^2}^{\frac{1}{2}}\right)\left(\|\nabla F\|_{L^2}^\frac{1}{2}+\|\nabla H\|_{L^2}^{\frac{1}{2}}\right)\\
&\quad +C\left(\nu^{-1}\|P(\rho)-P(\widetilde \rho) \|_{L^4} +\|w\|_{L^4}\right),
\end{align*}
which, combined with   (\ref{3.8})  and (\ref{3.37}), yields
\begin{equation}\label{3.38}
\begin{aligned}
\|\nabla u\|_{L^4} &\les C \left(\nu^{-1}\|F\|_{L^2}^\frac{1}{2}+\|H\|_{L^2}^{\frac{1}{2}}\right)\|\sqrt\rho\dot u\|_{L^2}^{\frac{1}{2}}\\
&\quad +C\left(\nu^{-1}\|P(\rho)- P(\widetilde \rho )\|_{L^4}+\|w\|_{L^4}\right).
\end{aligned}
\end{equation}

By virtue of (\ref{3.34}), (\ref{3.38}) and Cauchy-Schwarz inequality, we deduce
\begin{align*}
I_1^4&\les  C\varepsilon\|w\|_{L^4}\|\nabla u\|_{L^4}\|\nabla^2 w\|_{L^2}\\
&\les  C  \|w\|_{L^4}\left(\nu^{-1}\|F\|_{L^2}^\frac{1}{2}+\|H\|_{L^2}^{\frac{1}{2}}\right)\|\sqrt\rho\dot u\|_{L^2}^{\frac{1}{2}} \left( \|\sqrt\rho\dot w\|_{L^2} +\|\nabla^\bot\cdot u\|_{L^2} \right)\\
&\quad +C \|w\|_{L^4}\left(\nu^{-1}\|P(\rho)-P(\widetilde \rho) \|_{L^4}+\|w\|_{L^4}\right) \left(  \|\sqrt\rho\dot w\|_{L^2} +\|\nabla^\bot\cdot u\|_{L^2} \right)\\
&\les  \frac{1}{16}\left(\|\sqrt\rho\dot u\|_{L^2}^2+\|\sqrt\rho\dot w\|_{L^2}^2\right)+C  \left(\nu^{-4}\|F\|_{L^2}^2+\|H\|_{L^2}^2\right)\|w\|_{L^4}^4\\
&\quad+C \left( \nu^{-4}\|P(\rho)- P(\widetilde \rho) \|_{L^4}^4+\|w\|_{L^4}^4 + \|\nabla  u\|_{L^2}^2\right).
\end{align*}

Thus, substituting the estimates of $I_{1}^1$, $I_{1}^2$, $I_{1}^3$ and $I_1^4$ into (\ref{3.33}) shows
\begin{equation}\label{3.39}
\begin{aligned}
I_{1}&\les \frac{1}{4}\left(\|\sqrt\rho\dot u\|_{L^2}^2+\|\sqrt\rho\dot w\|_{L^2}^2\right)+C \left(\varepsilon\|\nabla w\|_{L^2}^2  + \|\nabla  u\|_{L^2}^2\right)  \\
&\quad +C \left(\nu^{-4}\|F\|_{L^2}^2+\|H\|_{L^2}^2+\varepsilon^2\|\nabla w\|_{L^2}^2\right)\|w\|_{L^4}^4\\
&\quad +C \left(\|w\|_{L^4}^4+ \nu^{-4}\|P(\rho)-P(\widetilde\rho)\|_{L^4}^4\right) .
\end{aligned}
\end{equation}

Using (\ref{2.10}), (\ref{3.2}), (\ref{3.8}) and (\ref{3.37}), we can bound  $I_3$ and $I_4$ by
\begin{equation}\label{3.40}
\begin{aligned}
I_3+I_4&\les  C\left(\|\nabla^\bot \cdot u\|_{L^4}^2+\nu^{-1}\|F\|_{L^4}^2\right)\|\divg u\|_{L^2}\\
&\les  C\left(\|H\|_{L^4}^2+\|w\|_{L^4}^2 +\nu^{-1}\|F\|_{L^4}^2\right)\|\divg u\|_{L^2}\\
&\les  C \left[ \left(\|H\|_{L^2}+\nu^{-1}\|F\|_{L^2}\right)\|\sqrt\rho\dot u\|_{L^2}+\|w\|_{L^4}^2\right] \|\divg u\|_{L^2} \\
&\les  \frac{1}{4}\|\sqrt\rho\dot u\|_{L^2}^2 +C \left(\nu^{-2}\|F\|_{L^2}^2+\|H\|_{L^2}^2\right)\|\divg u\|_{L^2}^2 \\
&\quad   +C \left(\|w\|_{L^4}^4+\|\divg u\|_{L^2}^2\right).
\end{aligned}
\end{equation}

Since $\nabla$ and $\nabla^\bot$ are orthogonal, it holds that $\nabla^\bot\cdot(\nabla u_1)=0$ and $\nabla\cdot(\nabla^\bot u_2)=0$. Hence, using the duality of ${\rm BMO}$ and $\mathcal{H}^1$ and the fact that  $\|F\|_{{\rm BMO}}\leq C\|\nabla F\|_{L^2}$,  we infer from (\ref{2.14}), (\ref{3.2}) and (\ref{3.8})  that
\begin{equation}\label{3.41}
\begin{aligned}
I_5&\les  C\|F\|_{{\rm BMO}}\|(\nabla u_1)\cdot(\nabla^\bot u_2)\|_{\mathcal{H}^1}\leq C\|\nabla F\|_{L^2}\|\nabla u\|_{L^2}^2\\
&\les  C \|\sqrt\rho\dot u\|_{L^2}\|\nabla u\|_{L^2}^2\les  \frac{1}{4}\|\sqrt\rho\dot u\|_{L^2}^2+C \|\nabla u\|_{L^2}^4.
\end{aligned}
\end{equation}

Replacing $\divg u=\nu^{-1}(F+P(\rho)-P(\widetilde\rho))$ again, using (\ref{2.10}), (\ref{3.8}) and Cauchy-Schwarz inequality, we deduce
\begin{equation}\label{3.42}
\begin{aligned}
 I_6+I_7 &\les C\nu^{-1}\left(\|F\|_{L^4}^2+\|P(\rho)-P(\widetilde\rho)\|_{L^4}^2\right)\|\divg u\|_{L^2} \\
&  \les C\nu^{-1}\left(\|F\|_{L^2} \|\nabla F\|_{L^2}   +\|P(\rho)-P(\widetilde\rho)\|_{L^4}^2\right)\|\divg u\|_{L^2}\\
&  \les C\nu^{-1}\left(\|F\|_{L^2}  \|\sqrt\rho\dot u\|_{L^2} +\|P(\rho)-P(\widetilde\rho)\|_{L^4}^2\right)\|\divg u\|_{L^2}\\
& \les \frac{1}{4}\|\sqrt\rho\dot u\|_{L^2}  +C \nu^{-2} \|F\|_{L^2}^2 \|\divg u\|_{L^2}^2 \\
&\quad + C \nu^{-2}\|P(\rho)- P(\widetilde \rho )\|_{L^4}^4   +C \|\divg u\|_{L^2}^2,
\end{aligned}
\end{equation}
and finally, by (\ref{3.2}) we easily obtain
\begin{equation}\label{3.43}
I_2+I_8\les
  C \left(\|w\|_{L^4}^4+ \|\divg u\|_{L^2}^2\right).
\end{equation}

Now, substituting (\ref{3.39}), (\ref{3.40}), (\ref{3.41}), (\ref{3.42}) and (\ref{3.43}) into (\ref{3.32}),  we see that for any $\nu\ges1$,
\begin{equation}\label{3.44}
\begin{aligned}
\mathcal{E}'(t)+\frac{1}{2}\mathcal{F}(t)&\les C \left(\nu^{-2}\|P(\rho)-P(\widetilde \rho) \|_{L^4}^4 + \|w\|_{L^4}^4\right)\\
&\quad +C  \left(\nu^{-4}\|F\|_{L^2}^2+\|H\|_{L^2}^2+\varepsilon^2\|\nabla w\|_{L^2}^2\right)\|w\|_{L^4}^4\\
&\quad +C \left(\nu^{-2}\|F\|_{L^2}^2+\|H\|_{L^2}^2\right)\|{\rm div} u\|_{L^2}^2\\
&\quad +C  \left( \varepsilon\|\nabla w\|_{L^2}^2  +\|\nabla  u\|_{L^2}^2  + \|\nabla u\|_{L^2}^4\right).
\end{aligned}
\end{equation}

From (\ref{3.2}), (\ref{3.19})  and (\ref{3.22}), it is easily seen that there exists a positive constant $C$, depending only on $\gamma$, $\widetilde\rho$, $\underline\rho$ and $\overline\rho$, such that for any $0\les t\les T$,
$$
\|P(\rho)-P(\widetilde \rho) \|_{L^2}^2 \les C\|\rho-\widetilde\rho\|_{L^2}^2\les C  \|G(\rho)\|_{L^1}\leq C.
$$
This, together with (\ref{3.6})$_1$ and Lemma \ref{lem2.3}, gives
\begin{equation}\label{3.45}
\nu^{-1}\|F\|_{L^2}^2\les \nu\|\divg u\|_{L^2}^2+\nu^{-1}\|P(\rho)-P(\widetilde \rho) \|_{L^2}^2\les \nu\|\divg u\|_{L^2}^2+C
\end{equation}
and
\begin{align*}
\|\nabla u\|_{L^2}^2&\les C\left(\|\nabla^\bot\cdot u\|_{L^2}^2+\|\divg u\|_{L^2}^2\right)\\
&\les C\left(\|\nabla^\bot\cdot u\|_{L^2}^2+\nu^{-2}\|F\|_{L^2}^2+\nu^{-2}\|P(\rho)-P(\widetilde \rho) \|_{L^2}^2\right)\\
&\les C \left(\|\nabla^\bot\cdot u\|_{L^2}^2+\nu^{-2}\|F\|_{L^2}^2+1\right),
\end{align*}
from which it follows that 
\begin{equation}\label{3.46}
\|\nabla u\|_{L^2}^4\les C \left(\|\nabla^\bot\cdot u\|_{L^2}^2+\nu^{-2}\|F\|_{L^2}^2\right)\|\nabla u\|_{L^2}^2+C \|\nabla u\|_{L^2}^2.
\end{equation}
Moreover, recalling the definition of $H$ in (\ref{3.35}), by (\ref{3.22})  we have
\begin{equation}\label{3.47}
\|H\|_{L^2}\les C \left( \|\nabla^\bot\cdot u\|_{L^2}+\|w\|_{L^2}\right)\les  C  \|\nabla^\bot\cdot u\|_{L^2}+C.
\end{equation}

Thus, substituting (\ref{3.45}), (\ref{3.46}) and (\ref{3.47}) into (\ref{3.44}) yields
\begin{equation}\label{3.48}
\begin{aligned}
\mathcal{E}'(t)+\frac{1}{2}\mathcal{F}(t)&\les  C_1  \left(\nu^{-2}\|P(\rho)-P(\widetilde \rho) \|_{L^4}^4 + \|w\|_{L^4}^4\right)\\
&\quad +C   \left(\|\nabla u\|_{L^2}^2 +\varepsilon^2\|\nabla w\|_{L^2}^2\right)\|w\|_{L^4}^4\\
&\quad  +C \left(\|\nabla^\bot\cdot u\|_{L^2}^2+\nu^{-2}\|F\|_{L^2}^2\right)\|\nabla u\|_{L^2}^2\\
&\quad +C  \left( \varepsilon\|\nabla w\|_{L^2}^2  +\|\nabla  u\|_{L^2}^2 \right)
\end{aligned}
\end{equation}
for the positive constants $C_1$ and $C$ depending on $\widetilde\rho$, $\underline\rho$ and $\overline\rho$.

\vskip 2mm

{\underline{\bf Step 4. The dissipations of $\|P(\rho)-P(\widetilde\rho)\|_{L^4}$ and $\|w\|_{L^4}$}}

\vskip 2mm

Obviously, we still need to deal with $\|P(\rho)-P(\widetilde\rho)\|_{L^4}$ and $\|w\|_{L^4}$. First, multiplying (\ref{3.29}) by $3(P(\rho) -P(\widetilde \rho))^2$ and integrating by parts, by (\ref{3.6})$_1$ we obtain
\begin{equation}\label{3.49}
\begin{aligned}
\frac{\md}{\md t} \int\left(P(\rho) -  P(\widetilde \rho)\right)^3 \mdx
& =\int\left(P(\rho) -P( \widetilde \rho) \right )^3  (\divg u) \mdx\\
&\quad   - 3\gamma  \int P(\rho)\left(P(\rho) -  P(\widetilde \rho)\right)^2 (\divg u) \mdx\\
&=\nu^{-1}\int \left|P(\rho)-P(\widetilde\rho)\right|^4 \mdx + {\text {Rem.}},
\end{aligned}
\end{equation}
where the remainder term ${\text{Rem}}$ is given by
$$
{\text{Rem.}}\triangleq \nu^{-1}\int \left( P(\rho)-P(\widetilde\rho)\right)^3 F \mdx  -3\gamma  \int P(\rho)\left(P(\rho) -  P(\widetilde \rho)\right)^2 (\divg u) \mdx.
$$

In view of (\ref{3.18}) and (\ref{3.19}), we obtain by (\ref{3.2}) and (\ref{3.22}) that
$$
\|P(\rho)-P(\widetilde \rho)\|_{L^2}^2+\|P(\rho)-P(\widetilde \rho) \|_{L^4}^4\les C \|G(\rho)\|_{L^1}\les C.
$$
Thus, using (\ref{2.10}), (\ref{3.6})$_1$, (\ref{3.8}) and Cauchy-Schwarz inequality, we deduce
\begin{align*}
|{\text{Rem.}}|&\les C\nu^{-1} \| P(\rho)-P(\widetilde\rho)\|_{L^4}^3\|F\|_{L^4}   + C\| P(\rho) -  P(\widetilde \rho)\|_{L^4}^2 \|\divg u\|_{L^2} \\
&\les C \nu^{-1} \|P(\rho) -P(\widetilde\rho)\|_{L^4}^3\|F\|_{L^2}^\frac{1}{2}\|\nabla F\|_{L^2}^\frac{1}{2}+ C\| P(\rho) -  P(\widetilde \rho)\|_{L^4}^2 \|\divg u\|_{L^2}\\
&\les C\nu^{-1} \|P(\rho) -P(\widetilde\rho)\|_{L^4}^3\left(\nu^{\frac{1}{2}}\|\divg u\|_{L^2}^{\frac{1}{2}}+\|P(\rho)-P(\widetilde\rho)\|_{L^2}^{\frac{1}{2}}\right) \|\sqrt{\rho}\dot{u}\|_{L^2}^{\frac{1}{2}}\\
&\quad+ C\| P(\rho) -  P(\widetilde \rho)\|_{L^4}^2 \|\divg u\|_{L^2}\\
&\les C \nu^{-\frac{1}{2}} \|P(\rho) - P(\widetilde \rho)\|_{L^4}^2\left(\|\divg u\|_{L^2}^{\frac{1}{2}}\|\sqrt{\rho}\dot{u}\|_{L^2}^{\frac{1}{2}}+\nu^{\frac{1}{2}}  \|\divg u\|_{L^2}\right)\\
&\quad+ C  \nu^{-1} \|P(\rho) - P(\widetilde \rho)\|_{L^4}^3 \|\sqrt{\rho}\dot{u}\|_{L^2}^{\frac{1}{2}}\\
&\les \frac{1}{16}\|\sqrt\rho\dot u\|_{L^2}^2+ \left((4\nu)^{-1}+C_2 \nu^{-\frac{4}{3}}\right)\|P(\rho) -P(\widetilde \rho) \|_{L^4}^4+C \nu \|\divg u\|_{L^2}^2.
\end{align*}

So, if $\nu\ges 1$ is chosen to be large enough such that
$$
\nu\ges \nu_{1,1}\triangleq\max\left\{1, (4C_2)^{3} \right\},
$$
then we derive from (\ref{3.49}) that
\begin{equation}\label{3.50}
\begin{aligned}
\nu^{-1}\|P(\rho) - P(\widetilde \rho)\|_{L^4}^4&\les 2 \frac{\md}{\md t} \int\left(P(\rho) -  P(\widetilde \rho) \right)^3 \mdx \\
&\quad +\frac{1}{8}\|\sqrt\rho\dot u\|_{L^2}^2+C \nu \|\divg u\|_{L^2}^2.
\end{aligned}
\end{equation}

In order to estimate $\|w\|_{L^4}$, we first use (\ref{3.35}) to rewrite   \eqref{1.1}$_3$ as
$$
\rho w_t+\rho u\cdot\nabla w+\frac{4\mu\zeta}{\mu+\zeta}  w-\varepsilon\Delta w=2\zeta \left(\nabla^\bot\cdot u-\frac{2\zeta}{\mu+\zeta}w\right)= 2\zeta H,
$$
which, multiplied  by $4|w|^2w$ and integrated by parts over $\mathbb{R}^2$, shows that for any $0<\delta<1$ (to be chosen later),
\begin{align*}
&\frac{\md }{\md t}\int \rho|w|^4\mdx+\frac{16\mu\zeta}{\mu+\zeta} \int |w|^4\mdx+12\varepsilon\int |w|^2 |\nabla w|^2\mdx\\
 &\quad\les C\|H\|_{L^4}\|w\|_{L^4}^3\les C\|H\|_{L^2}^{\frac{1}{2}}\|\nabla H\|_{L^2}^{\frac{1}{2}}\|w\|_{L^4}^3\\
 &\quad
 \les   \delta \|\sqrt\rho\dot u\|_{L^2}^2+\frac{8\mu\zeta}{\mu+\zeta}\|w\|_{L^4}^4+C(\delta )\|H\|_{L^2}^2\|w\|_{L^4}^4,
\end{align*}
where we have also used (\ref{2.10}) and (\ref{3.37}). This, together with (\ref{3.47})$_1$, gives
\begin{equation}\label{3.51}
\begin{aligned}
&\frac{\md }{\md t}\int \rho|w|^4\mdx+\frac{8\mu\zeta}{\mu+\zeta} \|w\|_{L^4}^4+12\varepsilon\||w||\nabla w|\|_{L^2}^2\\
&\quad \les \delta \|\sqrt\rho\dot u\|_{L^2}^2+C(\delta )\left(\|\nabla^\bot\cdot u\|_{L^2}^2+\|w\|_{L^2}^2\right)\|w\|_{L^4}^4.
\end{aligned}
\end{equation}

\vskip 2mm

{\underline{\bf Step 5. Concluding the proof}}

\vskip 2mm

To end the proof of (\ref{3.4}), we first choose $\nu$ to be large enough such that
$$
\nu\ges \nu_1\triangleq \max\left\{\nu_{1,1}, 2C_1 \right\},
$$
where $C_1$ is the positive constant in (\ref{3.48}). So, it follows from (\ref{3.48}) and (\ref{3.50}) that
\begin{equation}\label{3.52}
\begin{aligned}
&\mathcal{E}'(t)+\frac{1}{4}\mathcal{F}(t)+\frac{1}{2\nu}\|P(\rho)-P(\widetilde \rho) \|_{L^4}^4\\
&\quad \les 2 \frac{\md}{\md t} \int\left(P(\rho) -  P(\widetilde \rho) \right)^3 \mdx + C_1  \|w\|_{L^4}^4 \\
&\qquad+ C\left(\|\nabla u\|_{L^2}^2 +\varepsilon^2\|\nabla w\|_{L^2}^2\right)\|w\|_{L^4}^4\\
&\qquad  +C\left(\|\nabla^\bot\cdot u\|_{L^2}^2+\nu^{-2}\|F\|_{L^2}^2 \right) \|\nabla u\|_{L^2}^2 \\
&\qquad +C  \left( \varepsilon\|\nabla w\|_{L^2}^2  +\|\nabla  u\|_{L^2}^2 +\nu\|\divg u\|_{L^2}^2\right).
\end{aligned}
\end{equation}

Next, multiplying  (\ref{3.51}) by a suitably large number $C_3\triangleq \max\{1, C_1(\mu+\zeta)/(\mu\zeta)\}$ and taking $\delta= (8C_3)^{-1}$ sufficiently small, we obtain after adding (\ref{3.51}) and (\ref{3.52}) together that
\begin{equation}\label{3.53}
\begin{aligned}
&\widetilde{\mathcal{E}}'(t)+\widetilde{\mathcal{F}} (t)-2 \frac{\md}{\md t} \int\left(P(\rho) -  P(\widetilde \rho) \right)^3 \mdx\\
&\quad \les C  \left( \varepsilon\|\nabla w\|_{L^2}^2  +\|\nabla  u\|_{L^2}^2 +\nu\|\divg u\|_{L^2}^2\right) \\
&\qquad  +C \left(\|\nabla u\|_{L^2}^2 +\|w\|_{L^2}^2+\varepsilon^2\|\nabla w\|_{L^2}^2\right) \widetilde{\mathcal{E}}(t),
\end{aligned}
\end{equation}
where $\widetilde{\mathcal{E}}(t)$ and $\widetilde{\mathcal{F}}(t)$ are defined respectively by
$$
\widetilde{\mathcal{E}}(t)\triangleq \mathcal{E}(t)+C_3\int \rho|w|^4\mdx
$$
and
$$
\widetilde{\mathcal{F}}(t)\triangleq \frac{1}{8}\mathcal{F}(t) +\frac{1}{2\nu}\|P(\rho) -  P(\widetilde \rho)  \|_{L^4}^4+  \|w\|_{L^4}^4+\varepsilon\||w||\nabla w|\|_{L^2}^2.
$$

It readily follows from (\ref{3.22}) that
$$
\varepsilon\|\nabla w\|_{L^2}^2 +\|w\|_{L^2}^2  +\|\nabla  u\|_{L^2}^2 +\nu\|\divg u\|_{L^2}^2\in L^1(0,T),
$$
so that, we conclude from (\ref{3.53}) and Gronwall inequality that for any $0\les t\les T$,
\begin{equation}\label{3.54}
 \widetilde{\mathcal{E}}(t)+\int_0^t\widetilde{\mathcal{F}}(s)\md s\les C(\underline\rho,\overline\rho,K)
\end{equation}
for some positive constant $C$ depending on $\underline\rho$, $\overline\rho$, $\widetilde\rho$ and $ K$. Here, we have used (\ref{3.2}) and (\ref{3.22}) to get that
$$
\|P(\rho) - P( \widetilde \rho) \|_{L^3}^3\les C(\widetilde\rho,\underline\rho,\overline\rho)\|G(\rho)\|_{L^1}\les C(\widetilde\rho,\underline\rho,\overline\rho).
$$

By the definitions of $F$, $\mathcal{E}(t)$ and $\widetilde{\mathcal{E}}(t)$, we derive from (\ref{3.22}) and (\ref{3.54})  that
$$
\nu\|\divg u\|_{L^2}^2\les C\left(\nu^{-1}\|F\|_{L^2}^2+\|P(\rho)-P(\widetilde\rho)\|_{L^2}^2\right)\les C,\quad\forall\ 0\les t\les T,
$$
and moreover, by (\ref{3.34}) and (\ref{3.38})  we know
$$
\int_0^T\left(\|\nabla u\|_{L^4}^4+\varepsilon^2\|\nabla^2w\|_{L^2}^2\right)\md t\les C.
$$
Therefore, combining these estimates  with (\ref{3.2}), (\ref{3.22}) and (\ref{3.54}) leads to the desired estimates stated in (\ref{3.4}).

\subsection{Proof of Lemma \ref{lem3.2}} In this subsection, we aim to prove the global $t$-weighted estimates of the solutions. Analogously, since the calculations are lengthy, the derivation of (\ref{3.5}) is divided  into four steps.

\vskip 2mm

\underline{\bf Step 1. Building the energy functional}

\vskip 2mm

We begin with the setting-up of the energy-type functional. Applying $\partial_t+\divg(u\cdot  )$ to both sides of (\ref{3.20}) and multiplying it by $\dot u$ in $L^2$, we obtain
\begin{equation}
\begin{aligned}\label{3.55}
\frac{1}{2}\frac{\md}{\md t}\|\sqrt\rho \dot{u}\|_{L^2}^2
&=(\mu+\zeta)\int\dot{u}\cdot\left(\nabla^\bot(\nabla^\bot\cdot u_t)
+\partial_i\left(u_i\left(\nabla^\bot(\nabla^\bot\cdot u)\right)\right)\right)\mdx\\
&\quad-2\zeta\int\dot{u}\cdot\left(\nabla^\bot w _t+\partial_i
\left(u_i(\nabla^\bot w)\right )\right)\mdx\\
&\quad +\nu\int\dot{u}\cdot\left(\nabla \divg u_t+\partial_i
(u_i\nabla\divg  u)\right)\mdx\\
&\quad-\int\dot{u}\cdot\left(\nabla P_t+\partial_i
(u_i\nabla P)\right)\mdx\triangleq J_1+\ldots+J_4.
\end{aligned}
\end{equation}

By the definition of the material derivative $\dot f=f_t+u\cdot\nabla f$, we have from direct calculations that
\begin{align*}
&\nabla^\bot(\nabla^\bot\cdot u_t)
+\partial_i\left(u_i\left(\nabla^\bot(\nabla^\bot\cdot u)\right)\right)\\
&\quad=\nabla^\bot(\nabla^\bot\cdot \dot u)+\left(\divg  ((\nabla^\bot \cdot u)\partial_2 u ), -\divg  ((\nabla^\bot\cdot u)\partial_1u )\right),
\end{align*}
and hence, integrating by parts gives
\begin{equation}\label{3.56}
J_1= -(\mu+\zeta)\|\nabla^\bot\cdot \dot u\|_{L^2}^2-(\mu+\zeta)\int (\nabla^\bot \cdot u)\left( \partial_i\dot u_1\partial_2u_i-\partial_i\dot u_2\partial_1u_i\right)\mdx.
\end{equation}

Analogously, noting that
$$
\nabla^\bot w _t+\partial_i
\left(u_i(\nabla^\bot w )\right)
=\nabla^\bot\dot w+ (\nabla^\bot w\cdot \nabla ) u,
$$
we find
\begin{equation}
\begin{aligned}\label{3.57}
J_2&=2\zeta\int\dot w\nabla^\bot\cdot \dot{u}\mdx -2\zeta \int\nabla^\bot w\cdot \nabla  u\cdot \dot u \mdx \\
&=2\zeta\int \dot w\nabla^\bot\cdot \dot{u}\mdx -2\zeta \int w\left (\partial_1u\cdot\partial_2\dot{u} -\partial_2 u\cdot \partial_1 \dot{u} \right)\mdx.
\end{aligned}
\end{equation}

Recalling that $V=\divg u$ and $\dot V=V_t+u\cdot\nabla V$,
we have
\begin{align*}
\nabla \divg u_t+\partial_i
(u_i\nabla\divg  u)&=\nabla (V_t+u\cdot\nabla V)-\nabla (u\cdot\nabla \divg u)+\partial_i (u_i\nabla \divg u)\\
&=\nabla\dot V-\nabla u_i\partial_i \divg u+(\divg u)\nabla \divg u\\
&=\nabla\dot V-\partial_i ((\divg u) \nabla u_i)+2(\divg u)\nabla \divg u,
\end{align*}
and by (\ref{3.29}) we get
\begin{align*}
\nabla P_t+\partial_i (u_i\nabla P)&=-\nabla (\gamma P\divg u)-\nabla  u_i\partial_i P+(\divg u)\nabla P\\
&=-\nabla (\gamma P\divg u)-\partial_i(P\nabla  u_i)+(\divg u)\nabla P+P\nabla \divg u.
\end{align*}
Hence, the third and fourth terms on the right-hand side of (\ref{3.55}) can be written as
\begin{equation}\label{3.58}
\begin{aligned}
J_3+J_4&=-\int \divg \dot u\left(\nu \dot V+\gamma P\divg u\right)\md x+ \int (\nu\divg u-P)  \partial_i \dot u \cdot \nabla u_i \md x\\
&\quad +\int (\divg u) \dot u\cdot \nabla (\nu\divg u- P) \md x+\int  (\nu\divg u-P) \dot u\cdot \nabla \divg u \md x\\
&=-\int (\divg \dot u)\left(\nu \dot V+\gamma P\divg u\right)\md x+ \int (\nu\divg u-P)  \partial_i \dot u \cdot  \nabla u_i \md x\\
&\quad -\int (\divg u)(\nu\divg u- P)(\divg\dot u) \md x\\
&=-\nu\int |\dot V|^2\md x - \nu \int\dot V \partial_iu\cdot\nabla u_i\md x- \gamma \int P (\divg \dot u)(\divg u) \mdx\\
&\quad+ \int  (\nu\divg u-P)  \partial_i \dot u \cdot  \nabla u_i \md x -\int (\divg u) (\nu\divg u- P)\dot V \md x\\
&\quad-\int (\divg u)(\nu\divg u- P) \partial_i u\cdot\nabla u_i \md x,
\end{aligned}
\end{equation}
since it is easily checked that
\begin{equation}
\divg \dot u= \dot V+\divg (u\cdot\nabla u)-u\cdot\nabla\divg u=\dot V+\partial_i u\cdot\nabla u_i.\label{3.59}
\end{equation}

Thus, plugging (\ref{3.56}), (\ref{3.57}) and  (\ref{3.58}) into (\ref{3.55}), we arrive at
\begin{equation}
\begin{aligned}\label{3.60}
&\frac{1}{2}\frac{\md}{\md t}\|\sqrt\rho \dot{u}\|_{L^2}^2+\left((\mu+\zeta)\|\nabla^\bot\cdot \dot u\|_{L^2}^2+\nu\| \dot V\|_{L^2}^2\right)\\
&\quad=-(\mu+\zeta)\int (\nabla^\bot \cdot u)\left( \partial_i\dot u_1\partial_2u_i-\partial_i\dot u_2\partial_1u_i\right)\mdx +2\zeta\int \dot w\nabla^\bot\cdot \dot{u}\mdx
\\
&\qquad -2\zeta \int w\left (\partial_1u\cdot\partial_2\dot{u} -\partial_2 u\cdot \partial_1 \dot{u} \right)\mdx- \gamma\int P (\divg \dot u)(\divg u) \md x\\
&\qquad + \int  (\nu\divg u-P) \partial_i \dot u \cdot   \nabla u_i \md x -\int (\divg u)(\nu\divg u- P)\dot V \md x\\
&\qquad -\int (\divg u)(\nu\divg u- P) \partial_i u\cdot\nabla u_i \md x- \nu \int\dot V \partial_iu\cdot\nabla u_i\md x  \\
&\quad\triangleq N_1+\ldots+N_8.
\end{aligned}
\end{equation}

\vskip 2mm

\underline{\bf Step 2. The estimates of $N_1$--$N_7$}

\vskip 2mm

To begin, we first observe that (\ref{3.12}) and (\ref{3.13}) hold for any $\nu\ges \nu_1$, due to Lemma \ref{lem3.1}. Then, by (\ref{3.13}) we infer from (\ref{2.12}) that
\begin{equation}\label{3.61}
\begin{aligned}
N_1 &\les  C\|\nabla u\|_{L^4}^2\|\nabla \dot u\|_{L^2}\les  C\|\nabla u\|_{L^4}^2\left(\|\nabla^\bot\cdot \dot u\|_{L^2}+\|\divg \dot u\|_{L^2}\right)\\
&\les  C\|\nabla u\|_{L^4}^2\left(\|\nabla^\bot\cdot \dot u\|_{L^2}+\|\dot V\|_{L^2}+\|\nabla u\|_{L^4}^2\right)\\
&\les   \frac{1}{32}\left(\mu\|\nabla^\bot\cdot\dot u\|_{L^2}^2+ \nu \|\dot V\|_{L^2}^2\right)+C \|\nabla u\|_{L^4}^4,
\end{aligned}
\end{equation}
and analogously,
\begin{equation}\label{3.62}
\begin{aligned}
N_2+N_3+N_4 &\les  C\left(\|\dot{w}\|_{L^2}+\|w\|_{L^4}\|\nabla u\|_{L^4}+\|\divg u\|_{L^2}\right)\|\nabla \dot u\|_{L^2}\\
& \les  \frac{1}{32}\left(\mu\|\nabla^\bot\cdot\dot u\|_{L^2}^2+\nu\|\dot V\|_{L^2}^2\right) \\
&\quad  +C\left(\|\nabla u\|_{L^2}^2+\| \dot w\|_{L^2}^2+\|w\|^4_{L^4}+\|\nabla u\|_{L^4}^4 \right).
\end{aligned}
\end{equation}

Replacing  $\nu \divg u$ by the ``effective viscous flux" (i.e., $\nu\divg u=F+P(\rho)-P(\widetilde\rho)$),  we obtain after integrating by parts that
\begin{equation} \label{3.63}
\begin{aligned}
N_5&= \int  \left(F-P(\widetilde\rho) \right) \partial_i \dot u \cdot   \nabla u_i \md x\\
&= -\int   \dot u \cdot   \nabla u_i \partial_i F \md x-\int F\dot u\cdot \nabla \divg u \mdx-\int  P(\widetilde \rho) \partial_i \dot u \cdot   \nabla u_i  \md x\\
&= -\int   \dot u \cdot   \nabla u_i \partial_i F \md x-\frac{1}{\nu}\int F\dot u\cdot \nabla \left(F+P(\rho)-P(\widetilde \rho)\right)  \mdx \\
&\quad -\int  P(\widetilde \rho) \partial_i \dot u \cdot   \nabla u_i  \md x\\
&= -\int   \dot u \cdot   \nabla u_i \partial_i F \md x+\frac{1}{2\nu}\int F^2\divg\dot u\mdx +\frac{1}{\nu}\int F \divg\dot u   \left(P(\rho)-P(\widetilde \rho)\right) \mdx \\
&\quad +\frac{1}{\nu}\int \left(P(\rho)-P(\widetilde \rho)\right)\dot u \cdot \nabla F \mdx -\int  P(\widetilde\rho) \partial_i \dot u \cdot   \nabla u_i  \md x 
\triangleq \sum_{i=1}^5N_5^i .
\end{aligned}
\end{equation}

The first term on the right-hand side of (\ref{3.63}) can be bounded as follows, using (\ref{2.10}), (\ref{2.12}), (\ref{3.13}) and (\ref{3.59}),
\begin{align*}
N_5^1&\les  C\|\dot u\|_{L^4}\|\nabla u\|_{L^4}\|\nabla F\|_{L^2}\leq C\left(\|\dot u\|_{L^2}+\|\nabla \dot u\|_{L^2}\right) \|\nabla u\|_{L^4}\| \dot u\|_{L^2} \\
&\les  C\left(\|\dot u\|_{L^2}+\|\nabla^\bot\cdot \dot u\|_{L^2}+\|\dot V\|_{L^2}+\|\nabla u\|_{L^4}^2\right) \|\nabla u\|_{L^4}\| \dot u\|_{L^2} \\
& \les  \frac{1}{32}\left(\mu\|\nabla^\bot\cdot\dot u\|_{L^2}^2+\nu\|\dot V\|_{L^2}^2\right)+C\left(\|\nabla u\|_{L^4}^4+\|\dot u\|_{L^2}^2+\|\dot u\|_{L^2}^4\right).
\end{align*}

Owing to (\ref{3.4}), (\ref{3.8}), (\ref{3.13}) and (\ref{3.45}), we find
\begin{align*}
N_5^2&\les  C\nu^{-1}\|F\|_{L^4}^2\|\divg \dot u\|_{L^2}\les  C\nu^{-1}\|F\|_{L^2}\|\nabla F\|_{L^2}\left(\|\dot V\|_{L^2}+\|\nabla u\|_{L^4}^2\right)\\
&\les  C\|\dot u\|_{L^2}\left(\|\dot V\|_{L^2}+\|\nabla u\|_{L^4}^2\right)
\les \frac{\nu}{32}\|\dot V\|_{L^2}^2+C\left(\|\dot u\|_{L^2}^2+\|\nabla u\|_{L^4}^4\right).
\end{align*}
In a similar manner,
\begin{align*}
N_5^3&\les  C\nu^{-1}\|F\|_{L^4}\|P(\rho)-P(\widetilde \rho)\|_{L^4}\|\divg \dot u\|_{L^2}\\
&\les  C\nu^{-\frac{1}{2}} \|\dot u\|_{L^2}^{\frac{1}{2}}\|P(\rho)-P(\widetilde \rho)\|_{L^4}\left(\|\dot V\|_{L^2}+\|\nabla u\|_{L^4}^2\right)\\
&\les \frac{\nu}{32}\|\dot V\|_{L^2}^2+C\left(\|\dot u\|_{L^2}^2+\|\nabla u\|_{L^4}^4+\nu^{-1}\|P(\rho)-P(\widetilde \rho)\|_{L^4}\right)
\end{align*}
and
\begin{align*}
&N_5^4+N_5^5 \les C\left(\|\dot u\|_{L^2}\|\nabla F\|_{L^2}+\|\nabla\dot u\|_{L^2}\|\nabla u\|_{L^2}\right)\\
&\quad  \les \frac{1}{32}\left(\mu\|\nabla^\bot\cdot \dot u\|_{L^2}^2+\nu\|\dot V\|_{L^2}^2\right)+C\left(\|\nabla u\|_{L^2}^2+\|\dot u\|_{L^2}^2+\|\nabla u\|_{L^4}^4\right).
\end{align*}

Thus, inserting $N_5^i$ with $i=1,\ldots,5$ into (\ref{3.63}), we obtain
\begin{equation}
\begin{aligned}
N_5&\les \frac{1}{8}\left(\mu\|\nabla^\bot\cdot \dot u\|_{L^2}^2+\nu\|\dot V\|_{L^2}^2\right)+C\|\dot u\|_{L^2}^4\\
&\quad +C\left(\|\nabla u\|_{L^2}^2+\|\dot u\|_{L^2}^2+\|\nabla u\|_{L^4}^4+\nu^{-1}\|P(\rho)-P(\widetilde \rho)\|_{L^4}^4\right).
\end{aligned}\label{3.64}
\end{equation}

In a similar manner, using (\ref{3.2}), (\ref{3.6})$_1$, (\ref{3.8}) and (\ref{3.12}), we deduce
\begin{equation}\label{3.65}
\begin{aligned}
N_6 &= - \int (\divg u)\left(F-P(\widetilde \rho)\right) \dot V \md x\\
&\les C\left(\|F\|_{L^4}\|\nabla u\|_{L^4}+\|\nabla u\|_{L^2}\right)\|\dot V\|_{L^2}\\
&\les  C\nu^{\frac{1}{2}}\|\dot V\|_{L^2}\left(\nu^{-\frac{1}{2}}\|F\|^{\frac{1}{2}}_{L^2}\|\nabla F\|^{\frac{1}{2}}_{L^2}\|\nabla u\|_{L^4}+\nu^{-\frac{1}{2}}\|\nabla u\|_{L^2}\right)\\
&\les  C\nu^{\frac{1}{2}}\|\dot V\|_{L^2}\left(\|  \dot u\|^{\frac{1}{2}}_{L^2}\|\nabla u\|_{L^4}+\|\nabla u\|_{L^2}\right)\\
&\les  \frac{\nu}{32}\|\dot V\|_{L^2}^2+C\left(\|\nabla u\|^2_{L^2}+\|  \dot u\|_{L^2}^2+\|\nabla u\|_{L^4}^4\right)
\end{aligned}
\end{equation}
and
\begin{equation}\label{3.66}
\begin{aligned}
N_7
&=- \nu^{-1}\int\left(F+P(\rho)-P(\widetilde \rho) \right)\left(F-P(\widetilde \rho) \right)\partial_i u\cdot\nabla u_i\md x\\
&\les  C\nu^{-1}\|F\|_{L^4}^2\|\nabla u\|_{L^4}^2+C\|\nabla u\|_{L^2}^2\\
&\les  C\nu^{-1}\|F\|_{L^2}\|\nabla F\|_{L^2}\|\nabla u\|_{L^4}^2+C\|\nabla u\|_{L^2}^2\\
&\les  C\left(\|\nabla u\|_{L^2}^2+\|\dot u\|_{L^2}^2+ \|\nabla u\|_{L^4}^4\right).
\end{aligned}
\end{equation}

\vskip 2mm

\underline{\bf Step 3. The estimate of $N_8$}

\vskip 2mm

The treatment of the last term $N_8$ needs more work, due to the lack of the uniform bound of $\nu\|\nabla u\|_{L^4}^2$. First, recalling that $\dot V=(\divg u)_t+u\cdot\nabla(\divg u)$, by  (\ref{3.6})$_1$ and (\ref{3.29}) we write it in the form:
\begin{equation}\label{3.67}
\begin{aligned}
N_8&=-\int \left[(\nu\divg u)_t+u\cdot\nabla(\nu\divg u)\right]\partial_i u\cdot\nabla u_i\mdx \\
&=-\int \left[\left(F+P(\rho)-P(\widetilde \rho)\right)_t+u\cdot\nabla\left(F+P(\rho)-P(\widetilde \rho) \right)\right]\partial_i u\cdot\nabla u_i\mdx\\
&=-\int  F_t  \partial_i u\cdot\nabla u_i\mdx-\int \left(u\cdot\nabla F\right)\partial_i u\cdot\nabla u_i\mdx\\
&\qquad+\gamma\int P(\rho)(\divg u)\partial_i u\cdot\nabla u_i\mdx\\
& \triangleq  N_{8}^1+N_8^2+N_8^3.
\end{aligned}
\end{equation}

Since there is not any information of $F_t$, we have to rewrite the first term on the right-hand side of (\ref{3.67}) as
\begin{equation}\label{3.68}
\begin{aligned}
N_{8}^1&=-\frac{\md}{\md t}\int  F   \partial_i u\cdot\nabla u_i\mdx+\int  F  \left( \partial_i \partial_t u\cdot\nabla u_i+\partial_iu\cdot\nabla\partial_t u_{i}\right)\mdx\\
&=-\frac{\md}{\md t}\int  F   \partial_i u\cdot\nabla u_i\mdx-2\int   F \partial_t u\cdot\nabla\divg u\mdx-2\int \partial_t u\cdot\nabla u\cdot\nabla F \mdx.
\end{aligned}
\end{equation}

Keeping in mind that $u_t=\dot u-u\cdot\nabla u$, we have
\begin{align*}
&2\int   F \partial_t u\cdot\nabla\divg u\mdx= 2\int   F \dot u\cdot\nabla\divg u\mdx-2\int   F u\cdot\nabla u\cdot\nabla\divg u\mdx\\
&\quad = -2\int  (\divg u) \left( \dot u\cdot\nabla F + F\divg \dot u \right) \mdx+2\int   (\divg u) u\cdot\nabla u\cdot\nabla F\mdx\\
&\qquad +2\int   F(\divg u) \partial_i u\cdot\nabla u_i\mdx-\int  (\divg u)^2  u\cdot\nabla  F\mdx-\int  F (\divg u)^3 \mdx
\end{align*}
and
$$
2\int \partial_t u\cdot\nabla u\cdot\nabla F \mdx
 =2\int \dot u\cdot\nabla u\cdot\nabla F \mdx -2\int  u\cdot\nabla u\cdot\nabla u\cdot\nabla F \mdx,
$$
which, inserted into (\ref{3.68}), yields
\begin{equation}\label{3.69}
\begin{aligned}
N_{8}^1
&=-\frac{\md}{\md t}\int  F   \partial_i u\cdot\nabla u_i\mdx+2\int  (\divg u) \left( \dot u\cdot\nabla F + F\divg \dot u \right) \mdx\\
&\quad -2\int   (\divg u) u\cdot\nabla u\cdot\nabla F\mdx
 -2\int   F(\divg u) \partial_i u\cdot\nabla u_i\mdx\\
 &\quad +\int  (\divg u)^2  u\cdot\nabla  F\mdx+\int  F (\divg u)^3 \mdx\\
 &\quad -2\int \dot u\cdot\nabla u\cdot\nabla F \mdx +2\int  u\cdot\nabla u\cdot\nabla u\cdot\nabla F \mdx\\
 &\triangleq -\frac{\md}{\md t}\int  F   \partial_i u\cdot\nabla u_i\mdx+\sum_{i=1}^7 N_{8,i}^1.
\end{aligned}
\end{equation}

We are now in a position of handling each term on the right-hand side of (\ref{3.69}). First, similarly to the treatments of $N_5^1$, $N_5^2$ and $N_5^3$,  we have
\begin{align*}
N_{8,1}^1+N_{8,6}^1&\les C\|\nabla u\|_{L^4}\|\dot u\|_{L^4}\|\nabla F\|_{L^2}+C\nu^{-1} \|F\|_{L^4}^2\|\divg \dot u\|_{L^2} \\
&\quad +C\nu^{-1}\|F\|_{L^4}\|P(\rho)-P(\widetilde \rho)\|_{L^4} \|\divg \dot u\|_{L^2} \\
&\les \frac{1}{32}\left(\mu\|\nabla^\bot\cdot\dot u\|_{L^2}^2+\nu\|\dot V\|_{L^2}^2\right)+C\left(\|\nabla u\|_{L^2}^2+ \| \dot u\|_{L^2}^2\right)\\
&\quad+C\left( \|\nabla u\|_{L^4}^4+ \| \dot u\|_{L^2}^4+\nu^{-1}\|P(\rho)-P(\widetilde \rho) \|_{L^4}^4 \right).
\end{align*}

Replacing $\divg u $ by $\nu^{-1}(F+P(\rho)-P(\widetilde \rho))$ again, by (\ref{3.8}) and (\ref{3.12}) we deduce
\begin{align*}
& N_{8,3}^1+N_{8,5}^1
 \les C\nu^{-1}\int |F|\left( |F|+|P(\rho)-P(\widetilde \rho)|\right)|\nabla u|^2\mdx\\
&\quad \les  C\nu^{-1}\|F\|_{L^4}^2\|\nabla u\|_{L^4}^2+C\nu^{-1}\|F\|_{L^4}\|P(\rho)-P(\widetilde \rho)\|_{L^4}\|\nabla u\|_{L^4}^2\\
&\quad \les  C\|\nabla F\|_{L^2}\|\nabla u\|_{L^4}^2+C\nu^{-\frac{1}{2}}\|\nabla F\|_{L^2}^{\frac{1}{2}}\|P(\rho)-P(\widetilde \rho)\|_{L^4}\|\nabla u\|_{L^4}^2\\
&\quad \les  C\left(\|\nabla u\|_{L^4}^4+ \| \dot u\|_{L^2}^2+\nu^{-1}\|P(\rho)-P(\widetilde \rho)\|_{L^4}^4\right).
\end{align*}

Due to   (\ref{3.4}), we know that  $\|u\|_{H^1}$ is uniformly bounded. So, by (\ref{2.11}) and (\ref{3.8}) we obtain
\begin{align*}
N_{8,2}^1+N_{8,4}^1+N_{8,7}^1&\les  C\|u\|_{L^\infty}\|\nabla u\|_{L^4}^2\|\nabla F\|_{L^2}\\
& \les  C\left( \| u\|_{H^1}+\|\nabla u\|_{L^4}\right) \|\nabla u\|_{L^4}^2\| \dot u\|_{L^2}\\
& \les  C\left(\|  \dot{u}\|_{L^2}^2+\|\dot u\|_{L^2}^4 +\|\nabla u\|_{L^4}^4\right).
\end{align*}

Thus, putting the estimates of $N_{8,i}^1$ with $i=1,\dots,7$ into   (\ref{3.69}), we obtain
\begin{equation}\label{3.70}
\begin{aligned}
N_8^1
 &\les  -\frac{\md}{\md t}\int  F   \partial_i u\cdot\nabla u_i\mdx+\frac{1}{32}\left(\mu\|\nabla^\bot\cdot\dot u\|_{L^2}^2+\nu\|\dot V\|_{L^2}^2\right)+C\| \dot u\|_{L^2}^4\\
  &\quad + C\left(\|\nabla u\|_{L^2}^2+ \| \dot u\|_{L^2}^2+\|\nabla u\|_{L^4}^4+\nu^{-1}\|P(\rho)-P(\widetilde \rho)\|_{L^4}^4 \right).
\end{aligned}
\end{equation}

Analogously to the estimate of $N_{8,7}^1$, by (\ref{3.2}) we have
\begin{align*}
N_{8}^2+N_{8}^3&\les  C\|u\|_{L^\infty}\|\nabla u\|_{L^4}^2\|\nabla F\|_{L^2}+C\|\nabla u\|_{L^4}^4+C \|\nabla u\|^2_{L^2} \\
&\les  C\left(\|\dot u\|_{L^2}^2+ \|\nabla u\|_{L^4}^4 + \|  \dot u\|_{L^2}^4+\|\nabla u\|^2_{L^2}\right),
\end{align*}
which, together with (\ref{3.67}) and (\ref{3.70}), shows
\begin{equation}\label{3.71}
\begin{aligned}
N_{8}
 &\les -\frac{\md}{\md t}\int  F   \partial_i u\cdot\nabla u_i\mdx +\frac{1}{32}\left(\mu\|\nabla^\bot\cdot\dot u\|_{L^2}^2+\nu\|\dot V\|_{L^2}^2\right)+ C\| \dot u\|_{L^2}^4\\
 &\quad + C\left(\|\nabla u\|_{L^2}^2+ \| \dot u\|_{L^2}^2+\|\nabla u\|_{L^4}^4+\nu^{-1}\|P(\rho)-P(\widetilde\rho) \|_{L^4}^4 \right).
\end{aligned}
\end{equation}

\vskip 2mm

\underline{\bf Step 4. Ending the proof}

\vskip 2mm

Based upon (\ref{3.61}), (\ref{3.62}), (\ref{3.64}), (\ref{3.65}), (\ref{3.66}) and (\ref{3.71}), we infer from (\ref{3.60}) that there exists a positive constant $C$, depending on $\widetilde\rho$, $\underline\rho$ and $\overline\rho$, such that if (\ref{3.2}) holds and $\nu\ges \nu_1$ with $\nu_1$ being the same one as determined in Lemma \ref{lem3.1}, then
\begin{equation}
\begin{aligned}\label{3.72}
&\frac{\md}{\md t}\|\sqrt\rho \dot{u}\|_{L^2}^2+\frac{1}{2}\left(\mu\|\nabla^\bot\cdot \dot u\|_{L^2}^2+\nu\| \dot V\|_{L^2}^2\right)\\
&\quad\les  -2\frac{\md}{\md t}\int  F   \partial_i u\cdot\nabla u_i\mdx + C\left(\|\nabla u\|_{L^2}^2+ \| \dot u\|_{L^2}^2+ \|\dot w\|_{L^2}^2\right)\\
&\qquad+ C\left(\|w\|_{L^4}^4+\|\nabla u\|_{L^4}^4+\nu^{-1}\|P(\rho)-P(\widetilde \rho)\|_{L^4}^4 \right)+C\| \dot u\|_{L^2}^4.
\end{aligned}
\end{equation}

It remains to deal with the first term on the right-hand side of (\ref{3.72}). By direct calculations, we have
$$
 \partial_i u\cdot\nabla u_i=(\divg u)^2+2(\partial_1 u_2\partial_2 u_1-\partial_1 u_1\partial_2 u_2)=(\divg u)^2+2\nabla^\bot u_2\cdot\nabla u_1,
$$
and hence,
\begin{equation}\label{3.73}
\int  F \partial_i u\cdot\nabla u_i\mdx=\int F(\divg u)^2\mdx+2\int F\nabla^\bot u_2\cdot\nabla u_1\mdx.
\end{equation}

Analogously to the treatments of $N_{8,3}^1$ and $N_{8,5}^1$, using (\ref{3.2}), (\ref{3.8}), (\ref{3.12}), (\ref{3.13}) and  Cauchy inequality,  we obtain
\begin{equation}\label{3.74}
\begin{aligned}
&\left|\int F(\divg u)^2\mdx\right| =C\nu^{-1}\int |F|\left(|F|+|P(\rho)-P(\widetilde \rho)| \right)|\divg u| \mdx\\
&\quad \les  C\nu^{-1}\left(\|F\|_{L^4}^2+\|P(\rho)-P(\widetilde \rho)\|_{L^4}\|F\|_{L^4}\right)\|\divg u\|_{L^2}\\
&\quad \les  C\nu^{-\frac{1}{2}}\left( \|\nabla F\|_{L^2}+\|P(\rho)-P(\widetilde \rho) \|_{L^4} \|\nabla F\|^{\frac{1}{2}}_{L^2}\right)\|\divg u\|_{L^2}\\
&\quad  \les  \frac{1}{16}\|\sqrt\rho \dot u\|_{L^2}^2+ C \left(\|\divg u\|_{L^2}^2+\nu^{-1}\|P(\rho)-P(\widetilde \rho)\|_{L^4}^4\right).
\end{aligned}
\end{equation}
In the exactly same way as that used in  (\ref{3.41}), by (\ref{3.2}) we have
\begin{equation}\label{3.75}
\left|\int F\nabla^\bot u_2\cdot\nabla u_1\mdx\right|
\les  C\|\nabla F\|_{L^2}\|\nabla u\|_{L^2}^2\leq \frac{1}{16}\|\sqrt\rho\dot u\|_{L^2}^2+ C \|\nabla u\|_{L^2}^4 .
\end{equation}

Thus, substituting  (\ref{3.74}) and (\ref{3.75}) into (\ref{3.73}), we infer from (\ref{3.2}) and (\ref{3.4}) that for any $0\les t\les T$,
\begin{equation}\label{3.76}
\sigma(t)\left|\int  F   \partial_i u\cdot\nabla u_i\mdx\right|\les  \frac{1}{4} \sigma(t)\| \sqrt\rho \dot u\|_{L^2}^2+C,\quad \sigma(t)=\min\{1,t\}.
\end{equation}
and
\begin{equation}\label{3.77}
\begin{aligned}
&\int_0^t\left(\sigma(s)+\sigma'(s)\right)\left|\int  F   \partial_i u\cdot\nabla u_i\mdx\right|\md s\\
&\quad\les  C\int_0^t \left(\| \dot u\|_{L^2}^2+\|\divg u\|_{L^2}^2+\|\nabla u\|_{L^2}^4+\nu^{-1}\|P(\rho)-P(\widetilde \rho)\|_{L^4}^4\right)\md s \\
&\quad\les  C.
\end{aligned}
\end{equation}

Therefore,  multiplying (\ref{3.72}) by $\sigma(t)$ and integrating it over $(0,T)$,  we immediately arrive at the desired estimate stated in (\ref{3.5}),  based upon  (\ref{3.4}), (\ref{3.76}), (\ref{3.77}) and  Gronwall inequality.

\section{Global higher-order regularity}\label{sec4}

This section aims to prove the global higher-order estimates of the solutions, which ensure the global existence and uniqueness of strong solutions on $\R^2\times(0,T)$ for any $0<T<\infty$. From now on, we assume that the conditions of Proposition \ref{pro3.1} hold. That is,  the density is uniformly bounded from below and above and the bulk viscosity is taken to be large enough (i.e. $\nu\ges\nu_0$ with $\nu_0$ being the one determined in Proposition \ref{pro3.1}), and thus, the global  uniform estimates stated in (\ref{3.3}), (\ref{3.4}) and (\ref{3.5}) are valid.  For simplicity,  throughout this  section we denote by $C$ and $C(T)$ the various positive constants, which may depend additionally on $T$. We start with the global bound of $\|\nabla^\bot\cdot u\|_{L^p(0,T; L^\infty)}$ for some $1<p<\infty$, which plays an important role in the global and uniform estimates of the gradients of the density and the micro-rotational velocity.

\vskip 2mm

\begin{lem}\label{lem4.1}  Let $F=F(x,t)$ and $H=H(x,t)$ be the functions defined in \eqref{3.6}$_1$ and \eqref{3.35}, respectively. Then for any $2<q<\infty$,
\begin{equation}\label{4.1}
\int_0^T\left(\| \dot{u}\|_{L^q}^{\frac{q+1}{q}}+\|(\nabla F,\nabla H)\|_{L^q}^{\frac{q+1}{q}} + \nu^{-\frac{q+1}{2q}}\|F\|_{L^\infty}^{\frac{q+1}{q}}+\|H\|_{L^\infty}^{\frac{q+1}{q}}\right)\md t\\
\les C(T)
\end{equation}
and moreover,
\begin{equation}\label{4.2}
\int_0^T\left(  \nu^\frac{q+1}{2q}\|{\rm div} u\|_{L^\infty} ^\frac{q+1}{q}+\|\nabla^\bot\cdot u\|_{L^\infty}^\frac{q+1}{q}+\|w\|_{L^\infty}^\frac{q+1}{q}\right)\les C(T).
\end{equation}
\end{lem}
\begin{proof} It readily follows from (\ref{2.10}), (\ref{2.12}) and (\ref{3.13}) that  for any $2< q<\infty$,
\begin{align*}
\| \dot{u}\|_{L^q} &\les C \| \dot{u}\|_{L^2}^{\frac{2}{q}}\|\nabla \dot{u}\|_{L^2}^{\frac{q-2}{q}}\\
&\les  C\| \dot{u}\|_{L^2}^{\frac{2}{q}}\left(\|\nabla^\bot\cdot  \dot{u}\|_{L^2}  +\|\divg \dot u\|_{L^2} \right)^{\frac{q-2}{q}}  \\
&\les   C\| \dot{u}\|_{L^2}^{\frac{2}{q}}\left(\|\nabla^\bot\cdot  \dot{u}\|_{L^2} +\|\dot V\|_{L^2}   +\|\nabla u\|_{L^4}^2\right)^{\frac{q-2}{q}},
\end{align*}
which, combined with (\ref{3.4}), (\ref{3.5}), Cauchy-Schwarz and H\"{o}lder inequalities, yields
\begin{equation}\label{4.3}
\begin{aligned}
\int_0^T\| \dot{u}\|_{L^q}^{\frac{q+1}{q}}\md t
&\les C\int_0^T \| \dot{u}\|_{L^2}^{\frac{2(q+1)}{q^2}}\left(\|\nabla^\bot\cdot  \dot{u}\|_{L^2} +\|\dot V\|_{L^2}\right)^{\frac{(q-2)(q+1)}{q^2}}\md t\\
&\les C\sup_{0\les t\les T}\left(\sigma(t) \| \dot{u}\|_{L^2}^2\right)^{\frac{q+1}{q^2}}\\
&\qquad\quad\times \int_0^T \sigma(t)^{-\frac{q+1}{2q}}\left[\sigma(t)\left(\|\nabla^\bot\cdot\dot u\|_{L^2}^2+ \|\dot V\|_{L^2}^2 \right)\right]^{\frac{q^2-q-2}{2q^2}} \md t \\
&\les C\left[\int_0^T \sigma(t)\left(\|\nabla^\bot\cdot\dot u\|_{L^2}^2+ \|\dot V\|_{L^2}^2 \right)\md t\right]^{\frac{q^2-q-2}{2q^2}}\\
& \les C(T),
\end{aligned}
\end{equation}
where we have used the following simple fact that
$$
\int_0^T \sigma(t)^{-\frac{q(q+1)}{q^2+q+2}} \md t \leq C(T)+\int_0^1 t^{-\frac{q(q+1)}{q^2+q+2}} \md t\les  C(T),\quad \forall\ 2<q<\infty.
$$

As an easy result of (\ref{3.3}) and (\ref{4.3}), we infer from (\ref{3.7}) and (\ref{3.36}) that
\begin{equation}\label{4.4}
\int_0^T\left(\|\nabla F\|_{L^q}^{\frac{q+1}{q}}+\|\nabla H\|_{L^q}^{\frac{q+1}{q}}\right)\md t \les C\int_0^T\|\dot u\|_{L^q}^{\frac{q+1}{q}}\md t\leq C(T),
\end{equation}
which, together with (\ref{2.11}), (\ref{3.4}) and (\ref{3.12}), gives
\begin{equation}\label{4.5}
\begin{aligned}
&\int_0^T\left(\nu^{-\frac{1}{2}}\|  F\|_{L^\infty}+\|  H\|_{L^\infty} \right)^{\frac{q+1}{q}}\md t \\
&\quad \les C\int_0^T\left(\nu^{-\frac{ 1}{2}}\|F\|_{L^2} +\|\nabla F\|_{L^q} +\|H\|_{L^2} +\|\nabla H\|_{L^q}\right)^{\frac{q+1}{q}}\md t\\
&\quad\les C(T).
\end{aligned}
\end{equation}
This, together with (\ref{4.3}) and (\ref{4.4}), finishes the proof of (\ref{4.1}).

To prove (\ref{4.2}), we first deduce from  (\ref{3.3}), (\ref{3.6})$_1$ and (\ref{4.5}) that
\begin{equation}\label{4.6}
\int_0^T \left(\sqrt\nu \|\divg u\|_{L^\infty}\right)^\frac{q+1}{q}\md t \les  C(T)+C\int_0^T  \nu^{-\frac{q+1}{2q}}\|  F\|_{L^\infty}^\frac{q+1}{q}\md t\les C(T).
\end{equation}

Due to (\ref{3.3}), the micro-rotational velocity $w$ satisfies
$$
 -C\|H\|_{L^\infty}\les w_t+ u\cdot\nabla w+\frac{4\mu\zeta}{\mu+\zeta} \rho^{-1} w-\varepsilon\rho^{-1}\Delta w=2\zeta\rho^{-1} H\les C\|H\|_{L^\infty},
$$
and thus, it follows from (\ref{3.3}), (\ref{4.5}) and the maximum principle that
\begin{equation}\label{4.7}
\int_0^T \|w(t)\|_{L^\infty}^{\frac{q+1}{q}}\md t\les \int_0^T \|H\|_{L^\infty}^\frac{q+1}{q}\md t\les  C(T),
\end{equation}
which, combined with (\ref{3.35}) and (\ref{4.5}) again, gives
\begin{equation}\label{4.8}
\int_0^T\|\nabla^\bot\cdot u\|_{L^\infty}^\frac{q+1}{q}\md t\les  C\int_0^T\left(\|H\|_{L^\infty}^\frac{q+1}{q}+\|w\|_{L^\infty}^\frac{q+1}{q}\right)\md t\les  C(T).
\end{equation}
Therefore, collecting (\ref{4.6}), (\ref{4.7}) and (\ref{4.8}) together leads to (\ref{4.2}).
\end{proof}

\vskip 2mm

Based upon Proposition \ref{pro3.1}, Lemmas \ref{lem3.1}, \ref{lem3.2} and \ref{lem4.1}, we can now make use of the BKM's inequality in Lemma \ref{lem2.4} to
estimate the  $L^p$-norms ($p=2,4$)  of the gradients of the density and the micro-rotational velocity.

\vskip 2mm

\begin{lem}\label{lem4.2} Assume that $(\nabla\rho_0,\nabla w_0)\in L^2\cap L^4$. Then  for any $r\in(0,1/4)$,
\begin{equation}
\sup_{0\leq t\leq T}\|(\nabla\rho,\nabla w)(t)\|_{{L^2}{\cap}{L^4}}+\int_0^T\left(\|\nabla u\|_{W^{1,4}}^{1+r}+\|\nabla u\|_{L^\infty}^{1+r}\right)\md t\les C(r,T).\label{4.9}
\end{equation}
\end{lem}
\begin{proof} First, applying $\nabla$ on both sides of \eqref{1.1}$_1$ and multiplying it by $4|\nabla \rho|^2\nabla \rho$, by (\ref{3.3}) and (\ref{3.6}) we obtain after integrating by parts that
\begin{equation}
\begin{aligned}
\frac{\md}{\md t}\|\nabla\rho\|_{L^4}^4&\les C\|\nabla u\|_{L^\infty}\|\nabla\rho\|_{L^4}^4+C\|\nabla\divg u\|_{L^4}\|\nabla\rho\|_{L^4}^{3}\\
&\les  C\|\nabla u\|_{L^\infty}\|\nabla\rho\|_{L^4}^4+C\left(\|\nabla F\|_{L^4}+\|\nabla \rho\|_{L^4}\right)\|\nabla\rho\|_{L^4}^{3}
\end{aligned}\label{4.10}
\end{equation}

Next, we consider the estimate of $\|\nabla w\|_{L^4}$. It follows from (\ref{1.1})$_3$ and (\ref{3.35}) that
\begin{align*}
&\partial_i w_t+u\cdot\nabla\partial_i w+\frac{4\mu\zeta}{\mu+\zeta} \rho^{-1} \partial_i w-\varepsilon\divg\left(\rho^{-1}\nabla\partial_i w \right)\\
&\quad =-\partial_i u\cdot\nabla w +\frac{4\mu\zeta}{\mu+\zeta} \rho^{-2 }w\partial_i\rho + 2\zeta \rho^{-2}\left(\rho\partial_i H-H\partial_i\rho\right) \\ &\qquad +\varepsilon \rho^{-2}\left(\nabla\rho\cdot\nabla\partial_iw - \partial_i \rho\Delta w\right),
\end{align*}
which, multiplied by $4|\nabla w|^{2}\partial_i w$ and integrated by parts over $\mathbb{R}^2$, yields
\begin{equation}\label{4.11}
\begin{aligned}
&\frac{\md}{\md t}\int |\nabla w|^4\mdx +\frac{16\mu\zeta}{\mu+\zeta} \int \frac{|\nabla w|^4}{\rho}\mdx+
\varepsilon\overline\rho^{-1}\int |\nabla^2 w|^2|\nabla w|^2 \md x\\
&\quad\les  C\int |\nabla u| |\nabla w|^4\mdx+C\int |w||\nabla \rho| |\nabla w|^3\mdx+C \int |\nabla H| |\nabla w|^3 \mdx\\
&\qquad+ C\int |H| |\nabla \rho| |\nabla w|^3\mdx+C\varepsilon\int |\nabla \rho| |\nabla w|^3 |\nabla^2 w|\md x\triangleq \sum_{i=1}^5M_i.
\end{aligned}
\end{equation}
where we have used (\ref{3.3}).
It is easily seen that
\begin{align*}
M_1 +M_3  &\les  C\|\nabla u\|_{L^\infty}\|\nabla w\|_{L^4}^4+ C\|\nabla H\|_{L^4}
 \|\nabla w\|_{L^4}^3,
\\[2mm]
M_2 +M_4  &\les  C\left(\|H\|_{L^\infty}+\|w\|_{L^\infty}\right)\|\nabla\rho\|_{L^4} \|\nabla w\|_{L^4}^3,
\end{align*}
and
\begin{align*}
M_5&\les  C\varepsilon\|\nabla \rho\|_{L^4}\left\||\nabla w|^2\right\|_{L^4}\left\||\nabla w||\nabla^2w|\right\|_{L^2}\\
&\les  C \varepsilon  \|\nabla \rho\|_{L^4}
\left\||\nabla w|^2\right\|_{L^2}^{\frac{1}{2}}\left\|\nabla (|\nabla w|^2)\right\|_{L^2}^{\frac{1}{2}} \left\||\nabla w||\nabla^2w|\right\|_{L^2}\\
&\les  C \varepsilon  \|\nabla \rho\|_{L^4} \left\|\nabla w\right\|_{L^4} \left\||\nabla w||\nabla^2w|\right\|_{L^2}^{\frac{3}{2}}\\
&\les  \frac{\varepsilon}{2}\overline\rho^{-1}\left\||\nabla w|^2|\nabla^2 w|\right\|_{L^2}^2 + C\varepsilon  \|\nabla \rho\|_{L^4}^4 \left\|\nabla w\right\|_{L^4}^4,
\end{align*}
which, inserted   into (\ref{4.11}), shows
\begin{equation}\label{4.12}
\begin{aligned}
\frac{\md}{\md t}\|\nabla w\|_{L^4}&\les  C\|\nabla u\|_{L^\infty}\|\nabla w\|_{L^4}+C \varepsilon \|\nabla w\|_{L^4} \|\nabla\rho \|_{L^4}^4\\
&\quad + C\left(\|H\|_{L^\infty}+\|w\|_{L^\infty}\right)\|\nabla\rho\|_{L^4}+C \|\nabla H\|_{L^4}.
\end{aligned}
\end{equation}

In view of (\ref{3.8}) and (\ref{3.37}), we obtain from  (\ref{4.10}) and (\ref{4.12})  that
\begin{equation}\label{4.13}
\begin{aligned}
&\frac{\md}{\md t}\left(\|\nabla\rho\|_{L^4}^4+\|\nabla w\|_{L^4}\right)\\
&\quad \les   C\left(1+\|\nabla u\|_{L^\infty}\right)\left(\|\nabla\rho\|_{L^4}^4+\|\nabla w\|_{L^4}\right) +\varepsilon\|\nabla w\|_{L^4}\|\nabla\rho\|_{L^4}^4 \\
&\qquad +C\left( \|(H,w)\|_{L^\infty}+\|(\nabla F,\nabla H)\|_{L^4}\right)\left(1+\|\nabla\rho\|_{L^4}^4 \right)\\
&\quad \les   C\left(1+\|\nabla u\|_{L^\infty}\right)\left(\|\nabla\rho\|_{L^4}^4+\|\nabla w\|_{L^4}\right) +\varepsilon\|\nabla w\|_{L^4}\|\nabla\rho\|_{L^4}^4 \\
&\qquad +C\left( \|(H,w)\|_{L^\infty}+\|\dot u\|_{L^4}\right)\left(1+\|\nabla\rho\|_{L^4}^4 \right).
\end{aligned}
\end{equation}

It remains to show that $\|\nabla u\|_{L^\infty}\in  L^1(0,T)$. Indeed, recalling the definitions of $F$ and $H$ given respectively in \eqref{3.6}$_1$ and (\ref{3.35}), we have by (\ref{2.12}) that
\begin{equation}
\begin{aligned}
\|\nabla u\|_{W^{1,4}}&\les C\left(\|\divg u\|_{W^{1,4}}+\|\nabla^\bot\cdot u\|_{W^{1,4}}\right)\\
&\les  C\left(\nu^{-1}\|F\|_{W^{1,4}}+ \nu^{-1}\|P-P(\widetilde{\rho})\|_{W^{1,4}}+\|H\|_{W^{1,4}}+\|w\|_{W^{1,4}}\right)\\
&\les  C\left(1+\nu^{-1}\|F \|_{H^1}+\|H\|_{H^1}+\|(\nabla F,\nabla H)\|_{L^4}+\|(\nabla \rho,\nabla w)\|_{L^4}\right)\\
&\les  C\left(1+\| \dot{u}\|_{L^2}+\|\dot u\|_{L^4}+\|(\nabla \rho,\nabla w)\|_{L^4}\right).
\end{aligned}\label{4.14}
\end{equation}
where we have also used Lemma \ref{lem2.2}, (\ref{3.3}), (\ref{3.4}), (\ref{3.8}), (\ref{3.12}) and (\ref{3.37}). Thus, by virtue of (\ref{3.4}) and (\ref{4.14}), we infer from the BKM's inquality   \eqref{2.13}  that
\begin{equation}\label{4.15}
\begin{aligned}
\|\nabla u\|_{L^\infty}
&\les C\left(\|\divg u\|_{L^{\infty}}+\|\nabla^\bot\cdot u\|_{L^\infty}\right)\ln \left(e+\|\nabla u\|_{W^{1,4}}\right) +C\\
&  \les C\left(1+\|\divg u\|_{L^{\infty}}+\|\nabla^\bot\cdot u\|_{L^\infty}\right)\ln \left(e+\|\nabla \rho\|_{L^4}+\|\nabla w\|_{L^4}\right) \\
&\quad+ C\left(\|\divg u\|_{L^{\infty}}+\|\nabla^\bot\cdot u\|_{L^\infty}\right)\ln \left(e+\| \dot{u}\|_{L^2}+\|\dot u\|_{L^4}\right).
\end{aligned}
\end{equation}

To be continued, let
$$
\Gamma (t)\triangleq e+\|\nabla\rho\|_{L^4}^4+\|\nabla w\|_{L^4}.
$$
Then, substituting (\ref{4.15}) into (\ref{4.13}), we find
\begin{equation}\label{4.16}
(\ln \Gamma (t))'\les C\left(1+\|\divg u\|_{L^{\infty}}+\|\nabla^\bot\cdot u\|_{L^\infty}\right)\ln \Gamma(t) +G (t) \\
\end{equation}
where
\begin{align*}
G(t)&\triangleq C\left(\|\divg u\|_{L^{\infty}}+\|\nabla^\bot\cdot u\|_{L^\infty}\right)\ln \left(e+\| \dot{u}\|_{L^2}+\|\dot u\|_{L^4}\right)\\
& \quad+C\left(1+\|(H,w)\|_{L^\infty}+\|\dot u\|_{L^4}+\varepsilon\|\nabla w\|_{L^4}\right)
\end{align*}

Thanks to (\ref{2.10}) and (\ref{3.4}), we have
$$
\varepsilon\|\nabla w\|_{L^4}\les  C\varepsilon \|\nabla w\|_{H^1} \in L^1(0,T),
$$
which, together with (\ref{3.4}), (\ref{4.1}), (\ref{4.2}) and the simple fact that $a\ln(1+ b)\les C(\beta)(a^{1+\beta}+b^{1+\beta})$ for any $0<a,b<\infty$ and $\beta\in(0,1)$, implies that $G(t)\in L^1(0,T)$. Thus, by (\ref{4.1}) and Gronwall inequality we conclude from (\ref{4.16}) that
$$
\ln \Gamma(t)\les C(T)\quad \Longrightarrow\quad  \Gamma(t)\les  C(T),\quad \forall\ 0\les t\les T,
$$
that is,
\begin{equation}\label{4.17}
\|\nabla\rho(t)\|_{L^4}+\|\nabla w(t)\|_{L^4}\les C(T),\quad \forall\ 0\les t\les T,
\end{equation}
which, combined with (\ref{4.1}) and (\ref{4.15}), also shows that
\begin{equation}\label{4.18}
\int_0^T\|\nabla u\|_{L^\infty}^{1+r}\md t\les C(T)\;\;{\text{for\ any}}\;\; r\in(0,1/4).
\end{equation}

Following the derivation of (\ref{4.17}) step by step (indeed, much easier), we can make use of (\ref{4.1}), (\ref{4.2}) and (\ref{4.18}) to get that
$$
\|\nabla\rho(t)\|_{L^2}+\|\nabla w(t)\|_{L^2}\les C(T),\quad \forall\ 0\les t\les T.
$$
This, together with (\ref{4.17}) and (\ref{4.18}), finishes the proof of Lemma \ref{lem4.2}.
\end{proof}

As an easy consequence of  the estimates achieved, we can prove the following refined estimates for the velocity and the micro-rotational velocity.

\begin{lem}\label{lem4.3}Let the conditions of Proposition \ref{3.1} and Lemma \ref{lem4.2} be satisfied. Then for any $0< T<\infty$,
\begin{equation}\label{4.19}
\begin{aligned}
&\sup_{0\les  t\les  T} \left(t\|\nabla u(t)\|_{H^1}^2+t\|u_t(t)\|_{L^2}^2 \right)\\
 &\qquad   + \int_0^T \left(\|\nabla u\|_{H^1}^2+\|(u_t,w_t)\|_{L^2}^2+t\|\nabla u_t\|_{L^2}^2\right)\md t\les C(T).
 \end{aligned}
\end{equation}
\end{lem}
\begin{proof}
Based upon (\ref{2.12}), (\ref{3.6})$_1$ and (\ref{3.35}), we easily derive from (\ref{3.3}), (\ref{3.4}), (\ref{3.8}), (\ref{3.12}), (\ref{3.37}) and (\ref{4.9}) that
\begin{align*}
\|\nabla u\|_{H^1}&\les C\left(\|\divg u\|_{H^1}+\|\nabla^\bot\cdot u\|_{H^1}\right)\\
&\les C\left(\| H\|_{H^1}+\nu^{-1}\|F\|_{H^1}+\|\rho-\widetilde\rho\|_{H^1}+\|w\|_{H^1}\right)\\
&\les C+C\|\dot u\|_{L^2},
\end{align*}
which, together with (\ref{3.4}) and (\ref{3.5}), yields
\begin{equation}\label{4.20}
\sup_{0\les t\les T}\left(t\|\nabla u(t)\|_{H^1}^2\right)+\int_0^T \|\nabla u\|_{H^1}^2 \md t\les  C(T).
\end{equation}

Since $u_t=\dot u-u\cdot\nabla u$, by (\ref{2.10})  we have
$$
\|u_t\|_{L^2} \les \|\dot{u}\|_{L^2} +\|u\cdot\nabla u\|_{L^2}
\les  \|\dot{u}\|_{L^2} +\|u\|_{H^1}\|\nabla u\|_{H^1},
$$
so that, it follows from (\ref{3.4}) and (\ref{4.20})  that
\begin{equation}\label{4.21}
\sup_{0\les t\les T}\left(t\|u_t(t)\|_{L^2}^2\right)+\int_0^T \|u_t\|_{L^2}^2 \md t\les  C(T).
\end{equation}

Using (\ref{2.10}), (\ref{2.11}), (\ref{2.12}), (\ref{3.4}) and (\ref{3.13}),  we deduce (noting that $V=\divg u$)
\begin{align*}
\|\nabla u_t\|_{L^2}&\les C\left(\|\nabla \dot u\|_{L^2}+\|\nabla (u\cdot \nabla u)\|_{L^2}\right)\\
&\les C \left(\|\nabla^\bot \cdot\dot{u}\|_{L^2}+\|\dot{V}\|_{L^2} +\|\nabla u\|_{L^4}^2 +\|u\|_{H^2}^2\right)\\
&\les C\left(\|\nabla^{\perp}\cdot\dot{u}\|_{L^2}+\|\dot{V}\|_{L^2}+\| u\|_{H^2}^2\right),
\end{align*}
which, combined  with (\ref{3.5})  and (\ref{4.20}), implies
\begin{equation}\label{4.22}
 \int_0^T t\|\nabla u_t\|_{L^2}^2\les C(T).
\end{equation}
In a similar manner,
\begin{align*}
\int_0^T\|w_t\|_{L^2}^2&\les C\int_0^T\left( \|\dot{w}\|_{L^2}^2+\|u\cdot\nabla w\|_{L^2}^2\right)\md t\\
&\les  C\int_0^T\left(\|\dot{w}\|_{L^2}^2+\|u\|_{H^1}^2\|\nabla w\|_{L^4}^2\right)\md t\les C(T),
\end{align*}
which, together with  (\ref{4.20}), (\ref{4.21}) and (\ref{4.22}), leads to \eqref{4.19}.
\end{proof}

\section{Proofs of main results}
\subsection{Proof of Theorem \ref{thm1.1}}\label{sec5.1} Based upon Proposition \ref{3.1}, Lemmas \ref{lem3.1}, \ref{lem3.2} and \ref{lem4.1}--\ref{lem4.3}, we can prove the global existence and uniqueness of strong solutions in a standard way.

\vskip 2mm

\begin{proof}[Proof of Theorem \ref{thm1.1}]
First, since the density is strictly away from vacuum, the local existence of classical solutions with smooth data can be shown as that in \cite{MN1980} (see also \cite{TPZ2021}) by using the standard linearized scheme and the fixed-point argument.

With the local existence result at hand, we then can proceed to prove the density is globally and uniformly bounded from above and below, based upon Proposition \ref{3.1} and the bootstrap argument. Indeed, since $\underline\rho\les \rho_0\les \overline\rho$, it follows from the local existence result and the continuity argument that there exists a positive time $\widetilde T>0$ such that $\underline\rho/4\les \rho(x,t)\les 2\overline\rho$ for all $(x,t)\in\R^2\times[0,\widetilde T]$. Define
\begin{equation}
T^*\triangleq\sup\left\{0\les t\les T\; \left|\; \frac{1}{4}\underline\rho\les \rho(x,t)\les 2\overline\rho, \ \ \forall\  x\in\R^2\right. \right\}.
\label{5.1}
\end{equation}

Obviously,  $T^*\ges \widetilde T$. It now suffices to show that $T^*=\infty$. Otherwise, if $T^*<\infty$, then by Proposition \ref{pro3.1} we know that
$$
\frac{1}{2}\underline\rho\les \rho(x,t)\les \frac{3}{2}\overline\rho\quad {\text{on}}\quad \R^2\times[0,T^*),
$$
provided the bulk viscosity is sufficiently large (i.e. $\nu\ges \nu_0$). So, by Lemmas \ref{lem3.1}, \ref{lem3.2} and \ref{lem4.1}--\ref{lem4.3} we see that the strong solutions $(\rho,u,w)$ exist on $[0,T^*)$. Then, using the local existence result and the continuity argument again, we conclude that there is some $T'>T^*$ such that (\ref{5.1}) is valid on $[0,T']$, which is in contradiction with the definition of $T^*$ in (\ref{5.1}). Thus, it holds that $T^*=\infty$. This, combined with Lemmas \ref{lem3.1}, \ref{lem3.2} and \ref{lem4.1}--\ref{lem4.3}, proves the global existence of strong solutions $(\rho,u,w)$ on $\R^2\times [0,T]$ for any $0<T<\infty$.
Note that the uniqueness of strong solutions can be shown by using the standard $L^2$-method and the details are omitted for simplicity. The proof of   Theorem \ref{thm1.1} is therefore completed.
\end{proof}

\begin{re} Analogously to the arguments in \cite{Danchin2023,Ho1995}, we can also prove the global existence of ``intermediate weak" solutions of the problem \eqref{1.1}--\eqref{1.5}
with discontinuous data $(\rho_0,u_0,w_0)$ satisfying \eqref{3.1}, provided the bulk viscosity is large enough. However, the uniqueness of such weak solutions is still unknown.
\end{re}

\subsection{Proof of Theorem \ref{thm1.2}}\label{sec5.2} Based upon the global uniform bounds established in Sections \ref{sec3}, \ref{sec4} and Aubin-Lions lemma (cf. \cite{JL1969,si1987}), the combined limits as $\nu\to\infty$ and $\varepsilon\to0$ can be justified in the same way as that in \cite{Danchin2023}. So, it remains to derive the convergence rates. To do this, let $(\rho,u,w)$ and $(\eta,v,\chi)$ be the solutions of the problems (\ref{1.1})--(\ref{1.5}) and (\ref{2.1})--(\ref{2.2}) on $\R^2\times[0,T]$, subject to the initial data   $(\rho_0, u_0,w_0)$ and $(\eta_0, v_0,\chi_0)$ with $\eta_0=\rho_0$, $v_0=\mathcal{P}u_0$ and $\chi_0=w_0$, respectively.

\begin{proof}[Proof of Theorem \ref{thm1.2}]  To begin, we first recall that
$$
u= \mathcal{P} u+\mathcal{Q}u\quad \text{with}\quad  \mathcal{P} u\triangleq(\mathbb{Id}+\nabla(-\Delta)^{-1}\divg )u,\quad \mathcal{Q}u\triangleq-\nabla(-\Delta)^{-1}\divg u
$$
and that
\begin{equation}
\begin{cases}
\divg (\mathcal{P}u)=0,\quad  \nabla^\bot\cdot(\mathcal{P}u)=\nabla^\bot\cdot u,\\[1mm]
 \nabla^\bot\cdot(\mathcal{Q}u)=0,\quad  \divg (\mathcal{Q}u)=\divg u.\label{5.2}
\end{cases}
\end{equation}

Using (\ref{2.12}), (\ref{3.4}),  (\ref{3.8}), (\ref{3.13}) and (\ref{4.9}), we know that
\begin{equation}\label{5.3}
\|\nabla (\mathcal{Q}u)(t)\|_{L^2}^2\les C\|\divg u\|_{L^2}^2\les C\nu^{-1},\quad 0\les t\les T,
\end{equation}
and
\begin{equation}\label{5.4}
\begin{aligned}
&\int_0^T\|\nabla(\mathcal{Q}u)\|_{H^1}^2\md t\les C\int_0^T\|\divg u\|_{H^1}^2\md t\\
&\quad \les C\nu^{-2}\int_0^T \left(\| F\|_{L^2}^2+\|\nabla F\|_{L^2}^2+\|P(\rho)-P(\widetilde\rho)\|_{H^1}^2\right)\md t\\
&\quad \les C\nu^{-2}\int_0^T \left(\|F\|_{L^2}^2+\|\dot u\|_{L^2}^2 \right)\md t+ C\nu^{-2}\les C\nu^{-1}.
\end{aligned}
\end{equation}

\vskip 2mm
\underline{\bf Step 1. The difference of the densities}
\vskip 2mm

Let $\rho^*\triangleq \rho^*(x,t)$ be the ``incompressible" density  related to the divergence-free part of $u$ and determined through the transport equation:
\begin{equation}\label{5.5}
 \rho^*_t+\mathcal{P}u\cdot\nabla \rho^*=0,\quad \rho^*|_{t=0}=\rho_0.
\end{equation}

By the divergence-free condition $\divg (\mathcal{P}u)=0$, we infer from (\ref{5.5}) that
\begin{equation}
\underline\rho\les \rho^*\les \overline\rho\quad {\text{and}}\quad \rho^*-\widetilde\rho \in  L^\infty(0, T; L^2\cap  L^\infty).\label{5.6}
\end{equation}
Moreover, by virtue of (\ref{2.12}), (\ref{4.9}) and (\ref{5.2}) we see that
$$
\nabla(\mathcal{P}u)\in L^{1+r}(0,T; W^{1,4})\hookrightarrow L^{1+r}(0,T; L^\infty),\quad\forall\ r\in(0,1/4),
$$
and hence, it is easily derived from (\ref{5.5}) that
\begin{equation}\label{5.7}
\|\nabla\rho^*(t)\|_{L^2}+\|\nabla\rho^*(t)\|_{L^4}\leq C(T),\quad \forall\ 0\les t\les T.
\end{equation}

With $\rho^*=\rho^*(x,t)$ at hand, we decompose the difference of $\rho-\eta$  into two parts,
$$\rho-\eta=(\rho-\rho^*)+(\rho^*-\eta)\triangleq \varphi+\phi\quad\text{with}\quad \varphi \triangleq\rho- \rho^*,\quad \phi\triangleq \rho^*-\eta.$$
Then, it follows from (\ref{1.1})$_1$, (\ref{2.1})$_1$, (\ref{5.2}) and (\ref{5.5}) that
\begin{equation}\label{5.8}
\varphi_t+ \mathcal{P}u\cdot\nabla \varphi=-\divg(\rho\mathcal{Q}u),\quad \varphi|_{t=0}=0
\end{equation}
and
\begin{equation}\label{5.9}
\phi_t+v\cdot\nabla \phi=-(\mathcal{P}u-v)\cdot\nabla \rho^*,\quad \phi|_{t=0}=0.
\end{equation}

Next, we aim to prove some $t$-growth and singular $t$-weighted estimates for $(\varphi,\phi)$. First, multiplying (\ref{5.8})  by $\varphi$  and integrating it by parts over $\R^2$, by (\ref{3.3}), (\ref{4.9}), (\ref{5.2}) and (\ref{5.3}) we deduce that for any $0\les t\les T$,
\begin{align*}
\frac{\md}{\md t} \|\varphi(t)\|_{L^2}^2&\les C\left( \|\rho\divg (\mathcal{Q}u)\|_{L^2}+\|\mathcal{Q}u\cdot\nabla \rho\|_{L^2}\right)\|\varphi\|_{L^2}\\
&\les C\left(\|\divg u\|_{L^2}+\|\mathcal{Q}u\|_{L^{4}}\|\nabla\rho\|_{L^4} \right)\|\varphi\|_{L^2}\\
&\les C \left(\|\divg u\|_{L^2}^2+\|\mathcal{Q}u\|_{L^2}\|\nabla (\mathcal{Q}u)\|_{L^2}+\|\varphi\|_{L^2}^2\right)\\
&\les  C \nu^{-\frac{1}{2}} +C\|\varphi\|_{L^2}^2,
\end{align*}
from which it follows that
\begin{equation}\label{5.10}
\|\varphi(t)\|_{L^2}^2\leq C(T)\nu^{-\frac{1}{2}} t,\quad\forall\ 0\les t\les T.
\end{equation}

Analogously,  multiplying (\ref{5.9})  by $\phi$  and integrating it by parts over $\R^2$, by (\ref{2.4}), (\ref{3.4}), (\ref{5.2}) and (\ref{5.7}) we have
\begin{equation}\label{5.11}
\frac{\md}{\md t}\|\phi(t)\|_{L^2}^2 \les
 C \|\mathcal{P}u-v\|_{L^4}\|\nabla \rho^*\|_{L^4}\|\phi\|_{L^2}\les  C\|\phi\|_{L^2},
\end{equation}
and hence,
\begin{equation}\label{5.12}
 \|\phi(t)\|_{L^2}\les C(T)t,\quad\forall\ 0\les t\les T.
\end{equation}

In terms of (\ref{2.10}) and (\ref{5.7}), we also derive from (\ref{5.11})$_1$ that
\begin{equation}\label{5.13}
\begin{aligned}
\frac{\md}{\md t}\|\phi(t)\|_{L^2}^2
&\les   C \left(\| \mathcal{P}u-v \|_{L^2} +\|\nabla(\mathcal{P}u-v) \|_{L^2}\right)\|\phi\|_{L^2}\\
&\les   C \left(\| \mathcal{P}u-v \|_{L^2}^2+\|\phi\|_{L^2}^2\right)+C\|\nabla(\mathcal{P}u-v) \|_{L^2}^2.
\end{aligned}
\end{equation}
Moreover, multiplying (\ref{5.13})$_1$ by $t^{-1}$, we have by Cauchy-Schwarz inequality that
\begin{align*}
&\frac{\md}{\md t}\left(t^{-1}\|\phi(t)\|_{L^2}^2\right)+ t^{-2}\|\phi\|_{L^2}^2\\
&\quad \les  C \left(\| \mathcal{P}u-v \|_{L^2} +\|\nabla(\mathcal{P}u-v) \|_{L^2}\right)\left(t^{-1}\|\phi\|_{L^2}\right) \\
&\quad \les  \frac{1}{2 t^2} \|\phi\|_{L^2}^2+ C \left(\| \mathcal{P}u-v \|_{L^2}^2 +\|\nabla(\mathcal{P}u-v) \|_{L^2}^2\right),
\end{align*}
which, together with (\ref{5.13}), shows that
\begin{equation}\label{5.14}
\begin{aligned}
&\frac{\md}{\md t}\left(\|\phi(t)\|_{L^2}^2+t^{-1}\|\phi(t)\|_{L^2}^2\right)+\frac{1}{2t^2}\|\phi\|_{L^2}^2\\
&\quad  \les C\left( \| \mathcal{P}u-v \|_{L^2}^2 +\|\phi\|_{L^2}^2\right)+C_1\|\nabla(\mathcal{P}u-v) \|_{L^2}^2.
\end{aligned}
\end{equation}
It is worth mentioning that  (\ref{5.12}) implies that $\|\phi(t)\|_{L^2}^2/t\les C(T) t\to0$ as $t\to0$.

\vskip 2mm

\underline{\bf Step 2. The difference of the velocities}

\vskip 2mm

In view of (\ref{3.6})$_1$ and  (\ref{5.2}), we derive from (\ref{3.20}) that
\begin{equation}\label{5.15}
\begin{aligned}
&\rho^* (\mathcal{P}u)_t+ \rho^* \mathcal{P}u\cdot\nabla(\mathcal{P}u)-(\mu+\zeta) \Delta (\mathcal{P}u)+\nabla \Lambda\\
&\quad=   -\rho (\mathcal{Q}u)_t-\varphi(\mathcal{P}u)_t -\varphi \mathcal{P}u\cdot\nabla (\mathcal{P}u)\\
&\qquad -\rho\mathcal{Q}u\cdot\nabla(\mathcal{P}u)-\rho u\cdot\nabla(\mathcal{Q}u)-2\zeta\nabla^\bot w,
\end{aligned}
\end{equation}
where $\Lambda\triangleq -F$ with $F$ being the ``effective viscous flux" in (\ref{3.6})$_1$. Then, subtracting  (\ref{2.1})$_2$ from (\ref{5.15}) gives
\begin{equation}\label{5.16}
\begin{aligned}
&\eta (\mathcal{P}u-v)_t+\eta v\cdot\nabla(\mathcal{P}u-v)-(\mu+\zeta)\Delta(\mathcal{P}u-v)+\nabla\pi \\
&\quad= -\rho (\mathcal{Q}u)_t-\varphi(\mathcal{P}u)_t-\phi(\mathcal{P}u)_t-\phi\mathcal{P}u\cdot\nabla(\mathcal{P}u)-\eta (\mathcal{P}u-v)\cdot\nabla(\mathcal{P}u)\\
&\qquad  -\varphi \mathcal{P}u\cdot\nabla (\mathcal{P}u) -\rho\mathcal{Q}u\cdot\nabla(\mathcal{P}u)-\rho u\cdot\nabla(\mathcal{Q}u)-2\zeta\nabla^\bot(w-
\chi),
\end{aligned}
\end{equation}
subject to  the vanishing initial condition $(\mathcal{P}u-v)|_{t=0}=0$. Here, $\pi\triangleq \Lambda-\Pi$ with $\Pi$ being the incompressible pressure in (\ref{2.1}).

Next, multiplying (\ref{5.16}) by $\mathcal{P}u-v$ and integrating it by parts over $\R^2$, by (\ref{2.1})$_1$ we obtain
\begin{equation}
\frac{1}{2}\frac{\md}{\md t}\|\sqrt\eta (\mathcal{P}u-v)\|_{L^2}^2+(\mu+\zeta)\|\nabla(\mathcal{P}u-v)\|_{L^2}^2
\triangleq\sum_{i=i}^7 A_i,
\label{5.17}
\end{equation}
where $A_i$ with $i=1,\ldots,7$ is given respectively by
\begin{align*}
A_1&\triangleq -\int\rho(\mathcal{Q}u_t)\cdot (\mathcal{P}u-v)\md x,\quad
A_2 \triangleq -\int (\varphi+\phi)(\mathcal{P}u_t)\cdot (\mathcal{P}u-v)\md x,\\
A_3 &\triangleq -\int (\varphi+\phi)(\mathcal{P}u)\cdot\nabla(\mathcal{P}u)\cdot (\mathcal{P}u-v)\mdx,  \\
A_4& \triangleq -\int \eta (\mathcal{P}u-v)\cdot\nabla(\mathcal{P}u)\cdot (\mathcal{P}u-v)\mdx, \\
A_5&\triangleq -\int \rho(\mathcal{Q}u)\cdot\nabla(\mathcal{P}u)\cdot (\mathcal{P}u-v)\md x,\\
A_6&\triangleq-\int \rho u\cdot\nabla(\mathcal{Q}u)\cdot (\mathcal{P}u-v)\md x,\quad
A_7 \triangleq  2\zeta\int (w-\chi) \nabla^\bot\cdot (\mathcal{P}u-v)\md x.
\end{align*}

We are now in a position of handling each term on the right-hand side of (\ref{5.17}).
First,  noting that  $\mathcal{P} f$ and $\mathcal{Q} g$ are orthogonal for any $f,g\in L^2$, we have
$$
\widetilde\rho\int (\mathcal{Q}u_t) \cdot (\mathcal{P}u-v)\mdx=0
$$
so that, the first term $A_1$ can be written as
\begin{align*}
A_1& = -\int \left(\rho-\widetilde\rho\right) (\mathcal{Q}u_t)  \cdot (\mathcal{P}u-v)\mdx\\
&= -\frac{\md}{\md t}\int \left(\rho-\widetilde\rho\right) (\mathcal{Q}u)  \cdot (\mathcal{P}u-v)\mdx+\tilde A_1
\end{align*}
where
\begin{equation}\label{5.18}
\begin{aligned}
\tilde A_1& \triangleq \int\rho_t(\mathcal{Q}u)  \cdot (\mathcal{P}u-v)\mdx+\int \left(\rho-\widetilde\rho\right) (\mathcal{Q}u)  \cdot (\mathcal{P}u_t-v_t)\mdx\\
& \les  C\|\mathcal{Q}u\|_{L^4}\left(\|\rho_t\|_{L^2}\|\mathcal{P}u-v\|_{L^4}+\|\rho-\widetilde\rho\|_{L^4}\|\mathcal{P}u_t-v_t\|_{L^2}\right)\\
& \les  C\|\mathcal{Q}u\|_{L^2}^\frac{1}{2}\|\nabla\mathcal{Q}u\|_{L^2}^\frac{1}{2}\left( \|\mathcal{P}u-v\|_{L^2}^\frac{1}{2} \|\nabla(\mathcal{P}u-v)\|_{L^2}^\frac{1}{2}+ \|(u_t,v_t)\|_{L^2}\right)\\
&\les \frac{\mu+\zeta}{12}\|\nabla(\mathcal{P}u-v)\|_{L^2}^2+  C\|\mathcal{P}u-v\|_{L^2}^2+C \nu^{-\frac{1}{4}} \|(u_t,v_t)\|_{L^2} +C\nu^{-\frac{1}{2}},
\end{aligned}
\end{equation}
where we have used (\ref{2.10}), (\ref{3.3}), (\ref{3.4}), (\ref{4.9}) and (\ref{5.3}). As a result,
\begin{align*}
A_1&\les  -\frac{\md}{\md t}\int \left(\rho-\widetilde\rho\right) (\mathcal{Q}u)  \cdot (\mathcal{P}u-v)\mdx+\frac{\mu+\zeta}{12}\|\nabla(\mathcal{P}u-v)\|_{L^2}^2\\
&\quad+ C \|\mathcal{P}u-v\|_{L^2}^2+C\nu^{-\frac{1}{4}} \|(u_t,v_t)\|_{L^2} +C\nu^{-\frac{1}{2}}.
\end{align*}

Next, using (\ref{2.10}), (\ref{3.4}) and (\ref{5.10}), we deduce
\begin{align*}
A_2&\les C\left(\|\varphi\|_{L^2}+\|\phi\|_{L^2}\right)\|\mathcal{P}u_t\|_{L^4}\|\mathcal{P}u-v\|_{L^4}\\
&\les Ct^{-\frac{1}{2}} \left(  \|\varphi\|_{L^2}+ \|\phi\|_{L^2}\right)\left(\sqrt t\|u_t\|_{H^1}\right)\|\mathcal{P}u-v\|_{L^2}^\frac{1}{2}\|\nabla(\mathcal{P}u-v)\|_{L^2}^\frac{1}{2}\\
&\les \frac{\mu+\zeta}{12}\|\nabla(\mathcal{P}u-v)\|_{L^2}^2+C\|\mathcal{P}u-v\|_{L^2}^2+C\left(t\|u_t\|_{H^1}^2\right)\left(t^{-1}\|\phi\|_{L^2}^2+\nu^{-\frac{1}{2}}\right),
\end{align*}
and analogously,
\begin{align*}
A_3&\les C\left(\|\varphi\|_{L^2}+\|\phi\|_{L^2}\right)\|\mathcal{P}u\|_{H^1}\|\nabla(\mathcal{P}u)\|_{H^1}\|\mathcal{P}u-v\|_{L^4}\\
&\les  \frac{\mu+\zeta}{12}\|\nabla(\mathcal{P}u-v)\|_{L^2}^2+C\|\mathcal{P}u-v\|_{L^2}^2  +C\|\nabla u\|_{H^1}^2\left( \|\phi\|_{L^2}^2  +\nu^{-\frac{1}{2}}\right).
\end{align*}

By virtue of (\ref{2.4}) and (\ref{2.11}), we find
$$
A_4\les C\|\nabla(\mathcal{P}u)\|_{L^\infty}\|\mathcal{P}u-v\|_{L^2}^2
\les C\|\nabla u\|_{W^{1,4}}\|\mathcal{P}u-v\|_{L^2}^2.
$$

In terms of (\ref{2.10}), (\ref{3.4}) and (\ref{5.3}), we have
\begin{align*}
A_5&\les C\|\mathcal{Q}u\|_{L^4}\|\nabla (\mathcal{P}u)\|_{L^2}\|\mathcal{P}u-v\|_{L^4}\\
&\les C\|\mathcal{Q}u\|_{L^2}^\frac{1}{2}\|\nabla(\mathcal{Q}u)\|_{L^2}^\frac{1}{2}\|\mathcal{P}u-v\|_{L^2}^\frac{1}{2}\|\nabla(\mathcal{P}u-v)\|_{L^2}^\frac{1}{2}\\
&\les  \frac{\mu+\zeta}{12}\|\nabla(\mathcal{P}u-v)\|_{L^2}^2+C\|\mathcal{P}u-v\|_{L^2}^2+C\nu^{-\frac{1}{2}},
\end{align*}
and similarly,
\begin{align*}
A_6&\les C\|u\|_{L^4}\|\nabla(\mathcal{Q}u)\|_{L^2} \|\mathcal{P}u-v\|_{L^4}\\
&\les  \frac{\mu+\zeta}{12}\|\nabla(\mathcal{P}u-v)\|_{L^2}^2+C\|\mathcal{P}u-v\|_{L^2}^2+C\nu^{-1}.
\end{align*}

Finally, it is easily seen that
$$
A_7\les \frac{\mu+\zeta}{12}\|\nabla(\mathcal{P}u-v)\|_{L^2}^2+C\|w-\chi\|_{L^2}^2.
$$

Thus, putting the estimates of $A_i$ ($i=1,\ldots t$) into (\ref{5.17}), we arrive at
\begin{equation}
\begin{aligned}
& \frac{\md}{\md t}\|\sqrt\eta (\mathcal{P}u-v)\|_{L^2}^2+(\mu+\zeta)\|\nabla(\mathcal{P}u-v)\|_{L^2}^2\\
&\quad\les A_0'(t)  +C\left(1+\|\nabla u\|_{W^{1,4}}\right)\| \mathcal{P}u-v\|_{L^2}^2  + C\|w-\chi\|_{L^2}^2 \\
&\qquad +C\left(\|\phi\|_{L^2}^2+t^{-1}\|\phi\|_{L^2}^2\right)\left(\|\nabla u\|_{H^1}^2+t\|u_t\|_{H^1}^2\right) +C\nu^{-\frac{1}{2}}\\
&\qquad+C\nu^{-\frac{1}{2}}\left(\|\nabla u\|_{H^1}^2+t\|u_t\|_{H^1}^2\right)+C \nu^{-\frac{1}{4}} \|(u_t,v_t)\|_{L^2},
\end{aligned}\label{5.19}
\end{equation}
where
$$
A_0\triangleq -2 \int(\rho-\widetilde\rho)(\mathcal{Q}u)\cdot (\mathcal{P}u-v)\md x.
$$

\vskip 2mm

\underline{\bf Step 3. The difference of the micro-rotational velocities}

\vskip 2mm

Let  $\dot\chi\triangleq\chi_t+v\cdot\nabla\chi$. Then, subtracting (\ref{2.1})$_3$ from (\ref{1.1})$_3$ gives
\begin{align*}
&\rho(w-\chi )_t+\rho u\cdot\nabla(w-\chi)+4\zeta(w-\chi)-\varepsilon\Delta(w-\chi)\\
&\quad =\varepsilon\Delta \chi-(\varphi+\phi)\dot \chi-\rho(\mathcal{P}u-v)\cdot\nabla\chi-\rho\mathcal{Q}u\cdot\nabla\chi\\
&\qquad +2\zeta\nabla^{\bot}\cdot(\mathcal{P}u-v)+2\zeta\nabla^{\bot}\cdot(\mathcal{Q}u),
\end{align*}
which, multiplied by  $w-\chi$ and integrated by parts over $\R^2$, yields
\begin{equation}
\begin{aligned}
&\frac{1}{2}\frac{\md}{\md t}\|\sqrt\rho (w-\chi)\|_{L^2}^2+4\zeta \|w-\chi \|_{L^2}^2+\varepsilon\|\nabla(w-\chi )\|_{L^2}^2\\
&\quad =- \varepsilon\int \nabla \chi \cdot\nabla (w-\chi )\md x-\int \dot\chi(\varphi+\phi) (w-\chi )  \md x\\
&\qquad -\int \rho (w-\chi ) (\mathcal{P}u-v)\cdot\nabla \chi \md x-\int \rho (w-\chi ) (\mathcal{Q}u)\cdot\nabla \chi \md x\\
&\qquad  +2\zeta\int (w-\chi)\nabla^{\bot}\cdot(\mathcal{P}u-v)\mdx+2\zeta\int (w-\chi )  \nabla^{\bot}\cdot(\mathcal{Q}u)\md x\\
&\quad\triangleq B_1+\ldots+B_6.
\end{aligned}\label{5.20}
\end{equation}

Thanks to   (\ref{2.4}) and (\ref{5.10}), the first two terms $B_1$ and $B_2$ can be bounded by
\begin{align*}
& B_1+B_2\les C\varepsilon\|\nabla\chi\|_{L^2}\|\nabla(w-\chi)\|_{L^2}+C\|\dot\chi\|_{L^\infty} \left(\|\varphi\|_{L^2}+\|\phi\|_{L^2}\right) \|w-\chi\|_{L^2}\\
 &\quad \les \frac{\varepsilon}{2}\|\nabla(w-\chi)\|_{L^2}^2 + C\|\dot\chi\|_{L^\infty}\left(\|w-\chi\|_{L^2}^2+\|\phi\|_{L^2}^2\right)+C\varepsilon+C\nu^{-\frac{1}{2}}\|\dot\chi\|_{L^\infty}.
 \end{align*}

Using (\ref{2.4}), (\ref{2.10}),  (\ref{3.3}), (\ref{3.4}) and (\ref{5.3}), we have
\begin{align*}
&B_3+B_4
\les C\|w-\chi\|_{L^2}\left(\|\mathcal{P}u-v\|_{L^4}+\|\mathcal{Q}u\|_{L^4}\right)\|\nabla\chi\|_{L^4}\\
&\quad \les C\|w-\chi\|_{L^2}\left(\|\mathcal{P}u-v\|_{L^2}^\frac{1}{2}\|\nabla(\mathcal{P}u-v)\|_{L^2}^\frac{1}{2}
+\|\mathcal{Q}u\|_{L^2}^\frac{1}{2}\|\nabla\mathcal{Q}u\|_{L^2}^\frac{1}{2}\right)\\
&\quad \les \frac{\mu+\zeta}{8} \|\nabla(\mathcal{P}u-v)\|_{L^2}^2+C\left(\|w-\chi\|_{L^2}^2+\|\mathcal{P}u-v\|_{L^2}^2\right)+C\nu^{-\frac{1}{2}}.
\end{align*}

Since $\nabla^\bot\cdot (\mathcal{Q}u)=0$ due to (\ref{5.2}), it holds that $B_6=0$. Hence,
$$
B_5+B_6 \les  \frac{\mu+\zeta}{8} \|\nabla(\mathcal{P}u-v)\|_{L^2}^2+C \|w-\chi\|_{L^2}^2.
$$

Thus, inserting the above estimates of $B_i$ ($i=1,\ldots,6$) into (\ref{5.20}), we get
\begin{equation}
\begin{aligned}
&\frac{\md}{\md t}\|\sqrt\rho (w-\chi)\|_{L^2}^2+\zeta \|w-\chi \|_{L^2}^2+\varepsilon\|\nabla(w-\chi )\|_{L^2}^2\\
&\quad\les  \frac{\mu+\zeta}{2} \|\nabla(\mathcal{P}u-v)\|_{L^2}^2+ C\left(\|\mathcal{P}u-v\|_{L^2}^2+ \|w-\chi\|_{L^2}^2\right)\\
&\qquad+C\|\dot\chi\|_{L^\infty}\left(\|w-\chi\|_{L^2}^2+\|\phi\|_{L^2}^2 \right) + C\nu^{-\frac{1}{2}}\left(1+\|\dot\chi\|_{L^\infty}\right)+C\varepsilon.
\end{aligned}\label{5.21}
\end{equation}

\vskip 2mm

\underline{\bf Step 4. Proving the convergence rate}

\vskip 2mm

With (\ref{5.14}), (\ref{5.19}) and (\ref{5.21}) at hand, we are now ready to prove the convergence rate of the combined limits as $\nu\to\infty$ and $\varepsilon\to0$. First, it follows from (\ref{5.19}) and (\ref{5.21}) that
\begin{equation}
\begin{aligned}
& \frac{\md}{\md t}\left(\|\sqrt\eta (\mathcal{P}u-v)\|_{L^2}^2+\|\sqrt\rho (w-\chi)\|_{L^2}^2\right)\\
&\qquad+\frac{\mu+\zeta}{2}\|\nabla(\mathcal{P}u-v)\|_{L^2}^2+\zeta \|w-\chi \|_{L^2}^2+\varepsilon\|\nabla(w-\chi )\|_{L^2}^2\\
&\quad\les A_0'(t)  +C\left(\| \mathcal{P}u-v\|_{L^2}^2  + \|w-\chi\|_{L^2}^2 \right)\left(1+\|\nabla u\|_{W^{1,4}}+\|\dot \chi\|_{L^\infty}\right)\\
&\qquad +C\left(\|\phi\|_{L^2}^2+t^{-1}\|\phi\|_{L^2}^2\right)\left(\|\nabla u\|_{H^1}^2+t\|u_t\|_{H^1}^2+\|\dot\chi\|_{L^\infty}\right)+C\varepsilon\\
&\qquad+C\nu^{-\frac{1}{2}}\left(1+\|\nabla u\|_{H^1}^2+t\|u_t\|_{H^1}^2+\|\dot\chi\|_{L^\infty}\right)+C \nu^{-\frac{1}{4}} \|(u_t,v_t)\|_{L^2},
\end{aligned}\label{5.22}
\end{equation}
with $A_0$ being the same one as in (\ref{5.19}). Next, multiplying (\ref{5.22}) by a suitably large number $M \triangleq 4C_1/(\mu+\zeta)$ and adding it to (\ref{5.14}), by (\ref{3.3}) we obtain
\begin{equation}\label{5.23}
\mathcal{E}'(t)+\mathcal{F}(t) \leq M A_0'(t)+ C
\mathcal{G}(t)\mathcal{E}(t) +C\nu^{-\frac{1}{4}}\mathcal{G}(t)
+C\varepsilon,
\end{equation}
due to the fact that $\nu\geq 1$. Here,
\begin{align*}
\mathcal{E}(t) &\triangleq M\left(\|\sqrt\eta (\mathcal{P}u-v)\|_{L^2}^2+\|\sqrt\rho (w-\chi)\|_{L^2}^2\right)+\|\phi\|_{L^2}^2+t^{-1}\|\phi\|_{L^2}^2,\\
\mathcal{F}(t) &\triangleq C_1\|\nabla(\mathcal{P}u-v)\|_{L^2}^2+M\left(\zeta \|w-\chi \|_{L^2}^2+\varepsilon\|\nabla(w-\chi )\|_{L^2}^2\right)+\frac{1}{2t^2}\|\phi\|_{L^2}^2,
\end{align*}
and
$$
\mathcal{G}(t) \triangleq 1+\|\nabla u\|_{H^1}^2+t\|u_t\|_{H^1}^2+\|(u_t,v_t)\|_{L^2}+\|\nabla u\|_{W^{1,4}}+\|\dot \chi\|_{L^\infty}.
$$

It follows from Lemma \ref{lem2.1}, (\ref{3.3}), (\ref{3.4}) and (\ref{5.3}) that for any $0\les t\les T$,
\begin{equation}\label{5.24}
\begin{aligned}
MA_0(t)&\leq C\|\rho-\widetilde\rho \|_{L^4}\|\mathcal{Q}u\|_{L^4}\|\mathcal{P}u-v\|_{L^2}\\
&\leq C\|\mathcal{Q}u\|_{L^2}^\frac{1}{2}\|\nabla(\mathcal{Q}u)\|_{L^2}^\frac{1}{2}\|\sqrt\eta(\mathcal{P}u-v)\|_{L^2}\\
&\leq \frac{M}{2}\|\sqrt\eta(\mathcal{P}u-v)\|_{L^2}^2+C\nu^{-\frac{1}{2}}.
\end{aligned}
\end{equation}
Moreover, thanks to Lemma \ref{lem2.1} again, we infer from (\ref{2.1})$_3$ that
$$
\|\dot \chi\|_{L^\infty}\les C\left(\|\nabla v\|_{L^\infty}+\|\chi\|_{L^\infty}\right)\in L^1(0,T),
$$
from which, (\ref{2.4}), (\ref{4.9}) and (\ref{4.19}), it follows that $
\mathcal{G}(t)\in L^1(0,T)$. Hence, by (\ref{5.24}) we conclude from (\ref{5.23}) and Gronwall inequality that  for any $0\les t\les T$,
$$
\mathcal{E}(t)+\int_0^t\mathcal{F}(s)\md s\leq C(T)\left(\nu^{-\frac{1}{4}}+\varepsilon\right),
$$
This, together with Lemma \ref{lem2.1}, (\ref{3.3}) and (\ref{5.10}), proves (\ref{1.12}). Here, we have also used the fact that $\|\varphi(t)\|_{L^2}^2,t^{-1}\|\phi(t)\|_{L^2}^2\to0$ as $t\to0$, due to (\ref{5.10}) and (\ref{5.12}). Therefore, combining this with (\ref{5.3}) and (\ref{5.4}) finishes the proof of  Theorem \ref{thm1.2}.
\end{proof}

\subsection{Proof of Theorem \ref{thm1.3}} \label{sec5.3} This subsection aims to improve the convergence rate of the global solutions $(\rho,u,w)$ as $\nu\to\infty$, under the condition that  $\rho_0-\widetilde\rho$ decays sufficiently fast to zero at infinity (see (\ref{1.15})). We start with the proof of (\ref{1.16}).

\vskip 2mm

\begin{lem}\label{lem5.1} In addition to \eqref{2.3}, assume that for some $1/2<\theta<\infty$,
\begin{equation}\label{5.25}
\bar x^\theta(\eta_0-\widetilde\rho)\in L^4\;\; {\text{with}}\;\; \bar x\triangleq (e+|x|^2)^{\frac{1}{2}}\log^2 (e+|x|^2).
\end{equation}
Then, besides \eqref{2.4}, it holds that
\begin{equation}\label{5.26}
\|\bar x^\theta(\eta-\widetilde\rho)(t)\|_{L^4}\leq C,\quad\forall\ 0\les t\les T.
\end{equation}
\end{lem}
\begin{proof} Indeed, multiplying (\ref{2.1})$_1$ by $\bar x^{4\theta}|\eta-\widetilde \rho|^2(\eta-\widetilde\rho)$ and integrating it by parts, we obtain by the divergence-free condition $\divg v=0$ that
$$
\frac{\md}{\md t}\|\bar x^\theta (\eta-\widetilde\rho)\|_{L^4}^4\les C\|v\|_{L^\infty}\|\bar x^\theta(\eta-\widetilde\rho)\|_{L^4}^4,
$$
which, combined with (\ref{2.4}), (\ref{5.25}) and Gronwall inequality, leads to (\ref{5.26}).
\end{proof}

\vskip 2mm

With Lemma \ref{lem5.1} at hand, we are now ready to prove Theorem \ref{thm1.3}.

\vskip 2mm

\begin{proof}[Proof of Theorem \ref{thm1.3}] From the arguments in Subsection \ref{sec5.2}, we see that it only remains to reconsider the term $\tilde A_1$ given in (\ref{5.18}). For this purpose, recalling that $\rho=\eta+\varphi+\phi$, we have
\begin{equation}\label{5.27}
\begin{aligned}
 \tilde A_1 &= \int\rho_t(\mathcal{Q}u)  \cdot (\mathcal{P}u-v)\mdx+\int \left(\rho-\widetilde\rho\right) (\mathcal{Q}u)  \cdot (\mathcal{P}u_t-v_t)\mdx\\
& = \int\rho_t(\mathcal{Q}u)  \cdot (\mathcal{P}u-v)\mdx+\int \left(\varphi+ \phi\right) (\mathcal{Q}u)  \cdot (\mathcal{P}u_t-v_t)\mdx \\
&\quad +\int \left(\eta-\widetilde\rho\right) (\mathcal{Q}u)  \cdot (\mathcal{P}u_t-v_t)\mdx\triangleq \tilde A_{1,1}+\tilde A_{1,2}+\tilde A_{1,3},
\end{aligned}
\end{equation}
in which the first term on the right-hand side can be bounded by
\begin{align*}
\tilde A_{1,1}& \les  C\|\mathcal{Q}u\|_{L^4} \|\rho_t\|_{L^2}\|\mathcal{P}u-v\|_{L^4} \\
& \les  C\|\mathcal{Q}u\|_{L^2}^\frac{1}{2}\|\nabla\mathcal{Q}u\|_{L^2}^\frac{1}{2} \|\mathcal{P}u-v\|_{L^2}^\frac{1}{2} \|\nabla(\mathcal{P}u-v)\|_{L^2}^\frac{1}{2} \\
& \les \frac{\mu+\zeta}{12}\|\nabla(\mathcal{P}u-v)\|_{L^2}^2+  C\|\mathcal{P}u-v\|_{L^2}^2+ C\nu^{-\frac{1}{2}},
\end{align*}
where we have used (\ref{2.10}), (\ref{3.3}), (\ref{3.4}) and (\ref{5.3}). For the second term $\tilde A_{1,2}$, by (\ref{5.3}) and (\ref{5.10}) we find
\begin{align*}
\tilde A_{1,2}& \les  C\|\mathcal{Q}u\|_{L^4} \left(\|\varphi\|_{L^2}+\|\phi\|_{L^2}\right)\left(\|u_t\|_{L^4}+\|v_t\|_{L^4}\right) \\
& \les  C\|\mathcal{Q}u\|_{L^2}^\frac{1}{2}\|\nabla(\mathcal{Q}u)\|_{L^2}^\frac{1}{2}  \left(\|\varphi\|_{L^2}+\|\phi\|_{L^2}\right)\|(u_t,v_t)\|_{H^1} \\
& \les C\nu^{-\frac{1}{4}}\|\varphi\|_{L^2}\|(u_t,v_t)\|_{H^1} +C\nu^{-\frac{1}{4}} \|\phi\|_{L^2} \|(u_t,v_t)\|_{H^1}\\
&\les C\nu^{-\frac{1}{2}} \sqrt t \|(u_t,v_t)\|_{H^1}+ Ct^{-1}\|\phi\|_{L^2}^2+ C\nu^{-\frac{1}{2}}\left(  t \|(u_t,v_t)\|_{H^1}^2\right).
\end{align*}

The treatment of the third term $\tilde A_{1,3}$ depends on the application of Lemma \ref{lem5.1} and Hardy inequality on the 2D whole domain (see, e.g., \cite[Sec. 1.4]{Ladyzhenskaya}). Indeed, let $\dot W^{1,4/3}(\R^2)$ be the standard homogeneous Sobolev space. Then, it follows from  Sobolev embedding inequality $\dot W^{1,4/3}(\R^2)\hookrightarrow L^4(\R^2)$, H\"{o}lder inequality and Hardy inequality on $\R^2$ that for any $\Phi\in \dot W^{1,4/3}$ and $1/2<\theta<\infty$,
\begin{equation}
\begin{aligned}
\|\bar x^{-\theta} \Phi\|_{L^4}&\les C\|\nabla(\bar x^{-\theta} \Phi)\|_{L^\frac{4}{3}}\\
&\les C\|\bar x^{-\theta}\nabla \Phi\|_{L^{\frac{4}{3}}} +C\|\Phi \bar x^{-\theta-1}\nabla \bar x \|_{L^\frac{4}{3}}\\
&\les C\|\nabla \Phi\|_{L^2}\|\bar x^{-\theta}\|_{L^4}+C\|\bar x^{-1}\Phi\|_{L^2}\|\bar x^{-\theta} \nabla\bar x\|_{L^4}\\
&\les C\|\nabla\Phi\|_{L^2}.
\end{aligned}\label{5.28}
\end{equation}
Thus, it follows from (\ref{5.10}), (\ref{5.26}) and  (\ref{5.28}) that
\begin{align*}
\tilde A_{1,3}&\les C\|\bar x^\theta(\eta-\widetilde\rho)\|_{L^4}\|\bar x^{-\theta} (\mathcal{Q}u)\|_{L^4}\left(\|u_t\|_{L^2}+\|v_t\|_{L^2}\right) \\
&\les C\|\nabla(\mathcal{Q}u)\|_{L^2}\left(\|u_t\|_{L^2}+\|v_t\|_{L^2}\right)\\
&\les C\nu^{-\frac{1}{2}}\left(\|u_t\|_{L^2}+\|v_t\|_{L^2}\right).
\end{align*}

Putting the estimates of $\tilde A_{1,i}$ with $i=1,2,3$ into (\ref{5.27}) and combining it with $A_1$ in Subsection \ref{sec5.2}, we obtain the following refined estimate:
\begin{equation}\label{5.29}
\begin{aligned}
A_1&\les  -\frac{\md}{\md t}\int \left(\rho-\widetilde\rho\right) (\mathcal{Q}u)  \cdot (\mathcal{P}u-v)\mdx+\frac{\mu+\zeta}{12}\|\nabla(\mathcal{P}u-v)\|_{L^2}^2\\
&\quad+ C\left( \|\mathcal{P}u-v\|_{L^2}^2+t^{-1}\|\phi\|_{L^2}^2\right)+ C\nu^{-\frac{1}{2}}\left(  t \|(u_t,v_t)\|_{H^1}^2\right) \\
&\quad+C\nu^{-\frac{1}{2}} \left(\|(u_t,v_t)\|_{L^2} +\sqrt t\|(u_t,v_t)\|_{H^1}\right)+C\nu^{-\frac{1}{2}} .
\end{aligned}
\end{equation}
With (\ref{5.29}) at hand, we then can prove (\ref{1.17}) by repeating the proof of Theorem \ref{thm1.2} step by step. Moreover, based upon  (\ref{5.26}), (\ref{5.28}) with $\Phi$ being replaced by $\mathcal{Q}u$ and  Hardy inequality, we arrive at
(\ref{1.16}) and (\ref{1.18}). The proof of Theorem \ref{thm1.3} is therefore finished.
\end{proof}
\vskip 4mm

\noindent {\bf Acknowledgments}.
S. Liu was partly supported by the National Natural Science Foundation of China (Grant No. 12226347) and the Natural Science Foundation of Liaoning Province
(Grant No. 2023-MS-137). J. Zhang was partially supported by  the National Natural Science Foundation of China (Grant Nos. 12631009, 12471226, 12226344, 12131007, 12071390).

\vskip 2mm

\noindent{\bf Data Availability.} Data sharing not applicable to this article as no data-sets were
generated or analyzed during the current study.

\vskip 2mm

\noindent{\bf Declarations}
\vskip 2mm

\noindent{\bf Conflict of interest.} On behalf of all authors, the corresponding author states that there is no conflict of interest.


\vskip .1in

\begin{thebibliography}{00}


\bibitem{Amirat}
Y. Amirat, K. Hamdache, Weak solutions to the equations of motion for compressible magnetic fluids, J. Math. Pures Appl. 91 (2009) 433-467.

\bibitem{AKM1990}
S. A. Antontsev, A. V. Kazhikov, V. N. Monakhov, Boundary Value Problems in Mechanics of Nonhomogeneous Fluids. North-Holland, Amsterdam, 1990.


\bibitem{BRF2003}
J. Boldrini, M. A. Rojas-Medar, E. Fern\'{a}ndez-Cara, Semi-Galerkin approximation and strong solutions to the equations of the nonhomogeneous asymmetric fluids, J. Math. Pures Appl. (9) 82 (11) (2003) 1499-1525.


\bibitem{Brezis1974}
J. P. Bourguignon, H. Brezis, Remarks on the Euler equation, J. Funct. Anal. 15 (1974) 341-363.

\bibitem{Chen2015}
M. T. Chen, X. Y. Xu, J. W. Zhang, Global weak solutions of 3D compressible micropolar fluids with discontinuous initial data and
vacuum, Commun. Math. Sci. 13 (1) (2015) 225-247.




\bibitem{CZ2019}
Z. M. Chen, X. P. Zhai, Global large solutions and incompressible limit for the compressible Navier-Stokes equations, J. Math. Fluid Mech. 21 (2019), no. 2, Paper No. 26, 23 pp.

\bibitem{Coifman1993} R. Coifman, P. L. Lions, Y. Meyer, S. Semmes, Compensated compactness and Hardy spaces, J. Math. Pures Appl. 72 (1993) 247-286.

\bibitem{Meyer1975}
R. R. Coifman, Y. Meyer, On commutators of singular integrals and bilinear singular integrals,  Trans. Am. Math.Soc. 212 (1975) 315-331.

\bibitem{DM2017}
R. Danchin, P. Mucha, Compressible Navier-Stokes system: large solutions and incompressible limit, Adv. Math. 320 (2017) 904-925.


\bibitem{DM2019}
R. Danchin, P. B. Mucha, From compressible to incompressible inhomogeneous flows in the case of large data, Tunis. J. Math. 1 (2019) 127-149.

\bibitem{Danchin2023}
R. Danchin, P. Mucha, Compressible Navier-Stokes equations with ripped density, Commun. Pure Appl. Math. 76 (11) (2023)
3437-3492.



\bibitem{DZ2010}
B. Q. Dong, Z. F. Zhang, Global regularity of the 2D micropolar fluid flows with zero angular viscosity, J. Differ.
Equ. 249 (2010) 200-213.


\bibitem{Eringen1964}
A. C. Eringen, Simple microfluids, Int. J. Eng. Sci. 2 (1964) 205-217.

\bibitem{Eringen1966}
A. C. Eringen, Theory of micropolar fluids, J. Math. Mech. 16 (1966) 1-18.


\bibitem{Fe2004}
E. Feireisl, Dynamics of Viscous Compressible Fluids, Oxford University Press, Oxford, 2004.

\bibitem{FNP2001}
E. Feireisl, A. Novotn\'{y}, H. Petzeltov\'{a}, On the global existence of globally defined weak solutions to the Navier-Stokes equations of isentropic compressible fluids,  J. Math. Fluid Mech. 3 (2001) 358-392.

\bibitem{Galdi}
G. P. Galdi, An Introduction to the Mathematical Theory of Navier-Stokes Equations, Vol. I: Linearized Steady Problems. Springer Tracts in Natural Philosophy, Springer-Verlag, New York, 38, 1994.



\bibitem{Ho1995}
D. Hoff, Global solutions of the {N}avier-{S}tokes equations for multidimensional compressible flow with discontinuous initial data, J. Differ. Equ. 120 (1) (1995) 215-254.

\bibitem{Huang}
X. D. Huang, J. Li, Z. P. Xin,  Serrin-type criterion for the three-dimensional viscous compressible flows,  SIAM J. Math. Anal. 43 (4) (2011) 1872-1886.

\bibitem{Kato}
T. Kato, Remarks on the Euler and Navier-Stokes equations in $\mathbb{R}^2$, Proc. Symp. Pure Math. 45 (1986) 1-7.



\bibitem{Ladyzhenskaya}
O. A. Ladyzhenskaya, The Mathematical Theory of Viscous Incompressible Flow, 2nd ed., Gordon and Breach, New York, 1969.

\bibitem{LXZ2026} F. B. Lin, X. Y. Xu, J. W. Zhang, The vanishing limit of angular viscosity of the 2D micropolar equations for nonhomogeneous incompressible fluids on bounded domain, J. Differ. Equ. 477 (2026) Art. 114594.

\bibitem{JL1969}
J. L. Lions, Quelques m${\rm\acute{e}}$thodes de r${\rm\acute{e}}$solution de problemes aux limites non lin${\rm\acute{e}}$aires, Dunod, Paris; Gauthier-Villars, Paris, 1969, xx+554 pp.


\bibitem{Lions1996}
P. L. Lions, Mathematical Topics in Fluid Mechanics, Vol. 1: Incompressible Models. Oxford University Press, New York, 1996.

\bibitem{Li1998}
P. L. Lions, Mathematical Topics in Fluid Mechanics, Vol. 2: Compressible Models. Oxford University Press, New York, 1998.

\bibitem{MN1980}
A. Matsumura, T. Nishida, The initial value problem for the equations of motion of
viscous and heat-conductive gases, J. Math. Kyoto Univ. 20 (1980) 67-104.

\bibitem{LZ2026}
S. Q. Liu, J. W. Zhang, Global regularity and incompressible limit of 2D compressible Navier-Stokes equations with large bulk viscosity. Sci. China Math. (2026). https://doi.org/10.1007/s11425-025-2565-0.

\bibitem{LXZ2026-1}
S. Q. Liu, X. Y. Xu, J. W. Zhang, On the 2D Cauchy Problem of Compressible Navier-Stokes Equations with Large Bulk Viscosity. Preprinted, 2026.

\bibitem{Lu1988}
G.  {\L}ukaszewicz, On nonstationary flows of asymmetric fluids. Rend. Accad. Naz. Sci. XL Mem. Mat.
 (5) 12 (1) (1988) 83-97.

\bibitem{Lu1989}
G. {\L}ukaszewicz, On the existence, uniqueness and asymptotic properties for solutions of flows of asymmetric fluids, Rend. Accad. Naz. Sci. XL Mem. Mat. (5) 13 (1) (1989) 105-120.

\bibitem{Lu1999}
G. {\L}ukaszewicz, Micropolar Fluids, Theory and Applications, Birkh\"{a}user, Boston, 1999.

\bibitem{Mu1998-1}
N. Mujakovi\'{c}, One-dimensional flow of a compressible viscous micropolar fluid: a local existence theorem. Glas. Mat. Ser. III 33(53) (1998), no. 1, 71-91.

\bibitem{Mu1998-2}
N. Mujakovi\'{c},
One-dimensional flow of a compressible viscous micropolar fluid: a global existence theorem. Glas. Mat. Ser. III 33(53) (1998), no. 2, 199-208.


\bibitem{Nirenberg}
L. Nirenberg, On elliptic partial differential equations, Ann. Scuola Norm. Sup. Pisa. 13 (3) (1959) 115-162.

\bibitem{QCZ2023}
C. Y. Qian, H. Chen, T. Zhang,  Global existence of weak solutions for 3D incompressible inhomogeneous asymmetric fluids, Math. Ann. 386 (2023) 1555-1593.

\bibitem{si1987}
J. Simon, Compactness sets in the space $L^p(0,T;B)$, Ann. Math. Pura Appl. 146 (1987) 65-96.

\bibitem{Song2023}
Z. Song, The global well-posedness for the 3-D compressible micropolar system in the critical besov space, Z. Angew. Math. Phys. 72 (2021) 160.

\bibitem{Te1984}
R. Temam, Navier-Stokes Equations. North-Holland, Amsterdam, 1984.

\bibitem{TPZ2021}
L. L. Tong, R. H. Pan, Z. Tan, Decay estimates of solutions to the compressible micropolar fluids system in $\mathbb{R}^3$.  J. Differ. Equ. 293 (2021) 520-552.


\end{thebibliography}
\end{document}